\documentclass[11pt, letterpaper]{amsart}

\usepackage{amsmath}
\usepackage{amsthm}
\usepackage{enumitem}
\usepackage{amssymb}
\usepackage{tikz-cd} 
\usepackage{hyperref}
\hypersetup{colorlinks=true}
\usepackage{dsfont}
\usepackage{mdframed}
\usetikzlibrary{decorations.pathmorphing}

\usetikzlibrary{decorations.markings}
\tikzset{negated/.style={
        decoration={markings,  
            mark= at position 0.5 with {
                \node[transform shape] (tempnode) {$\backslash$};
            }
        },
        postaction={decorate}
    }
}

\theoremstyle{plain}
\newtheorem{theorem}{Theorem}[section]

\newtheorem{proposition}[theorem]{Proposition}
\newtheorem{lemma}[theorem]{Lemma}
\newtheorem{corollary}[theorem]{Corollary}

\newtheorem*{theorem*}{Theorem}

\theoremstyle{definition}
\newtheorem{definition}[theorem]{Definition}
\newtheorem*{definition*}{Definition}

\theoremstyle{remark}
\newtheorem{remark}{Remark}[section]

\newtheorem{question}{Question}
\newtheorem{assumption}{Assumption}

\DeclareMathOperator{\Aut}{Aut}

\DeclareMathOperator{\GL}{\mathbf{GL}}
\DeclareMathOperator{\SL}{\mathbf{SL}}

\DeclareMathOperator{\rank}{rank}
\DeclareMathOperator{\Ad}{Ad}
\DeclareMathOperator{\id}{id}
\DeclareMathOperator{\Fix}{Fix}

\DeclareMathOperator{\Comm}{Comm}
\DeclareMathOperator{\Rad}{Rad}
\DeclareMathOperator{\Stab}{Stab}

\title{The structure of approximate lattices in linear groups}
\author[Simon Machado]{Simon Machado \\
ETHZ}
\email{smachado@ethz.ch}
\date{\today}
\begin{document}

\begin{abstract}
Approximate lattices are aperiodic generalisations of lattices of locally compact groups  that were first studied in seminal work of Yves Meyer. They are defined as those uniformly discrete approximate subgroups (symmetric subsets stable under multiplication up to a finite error) of locally compact groups that have finite co-volume. Meyer showed that approximate lattices of Euclidean spaces (a.k.a. Meyer sets) are related to lattices in higher-dimensional Euclidean spaces via the cut-and-project construction.  A fundamental challenge of the theory of approximate lattices is to extend Meyer's theorem beyond Euclidean spaces.

 Our main result provides a complete structure theorem for approximate lattices valid in all linear algebraic groups over local fields and their finite products, in particular providing the most general extension of Meyer's theorem to date. Our proof relies on an extension of a theorem of Lubotzky--Mozes--Raghunathan to approximate lattices in $S$-adic semi-simple groups, a notion of cohomology tailored to the study of approximate subgroups, a universality statement complementing a recent result of Hrushovski and a generalisation of a result of Burger and Monod about bounded cohomology of lattices.
\end{abstract}

\maketitle
{
  \hypersetup{linkcolor=black}
  \tableofcontents
}

\section{Introduction}
Approximate lattices are approximate subgroups of locally compact groups that are discrete and have finite co-volume. A subset $\Lambda$ of a locally compact group $G$ is an \emph{approximate subgroup} if it is symmetric ($\Lambda = \Lambda^{-1}$), contains the identity and satisfies 
$$\Lambda^2:=\{\lambda_1\lambda_2 \in G: \lambda_1 , \lambda_2 \in \Lambda\} \subset F\Lambda :=\{f\lambda \in G : f \in F, \lambda \in \Lambda\}$$ for some $F \subset G$ finite. It has \emph{finite co-volume} when there is $\mathcal{F} \subset G$ with finite Haar measure such that $\Lambda \mathcal{F} = G$. Meyer was the first to study approximate lattices in his seminal monograph \cite{meyer1972algebraic}. Inspired by a construction due to Bohr, Meyer also introduced the so-called \emph{model sets} which are often understood as the `ideal' approximate lattices. A model set is a subset of a locally compact group $G$ built from a lattice $\Gamma$ in a product  $G \times H$ where $H$ is another locally compact group. It is obtained by first cutting a strip of $\Gamma$ along $G$ and then projecting it to $G$. More precisely, one chooses a relatively compact symmetric neighbourhood of the identity $W_0$ of $H$ - the \emph{window} - and defines the model set
$$ M(G,H,\Gamma,W_0):=p_G\left(\Gamma \cap (G \times W_0)\right)$$ 
where $p_G: G \times H \rightarrow G$ denotes the natural projection. The triple $(G,H,\Gamma)$ is the \emph{cut-and-project scheme} of which $G$ is the \emph{physical space} and $H$ is the \emph{internal space}.

One of the main features of the theory initiated by Meyer is the beautiful classification of approximate lattices in locally compact abelian groups he obtained in \cite{meyer1972algebraic}, see also \cite{moody1997meyer} for a survey. It asserts that every approximate lattice of a locally compact abelian group is contained in a model set.  The main challenge in the theory of approximate lattices consists in extending Meyer's theorem beyond abelian groups \cite[Problem 1]{bjorklund2016approximate}. Our main result fulfils this goal in the generality of linear algebraic groups over characteristic $0$ local fields and their finite products, which we refer to as $S$-\emph{adic linear groups}. An already striking consequence is the following high-rank versus rank-one alternative: 

\begin{theorem}[High rank v rank one]\label{THEOREM: FIRST THEOREM}
Let $\Lambda$ be an approximate lattice in an $S$-adic linear group $G$. If there is no factor $p: G \rightarrow L$ with $L$ an $S$-rank one simple $S$-adic linear group and $p(\Lambda )$ contained in a lattice, then $\Lambda$ is contained in a model set. In contrast, if there is a rank one factor $p: G \rightarrow S$ with $p(\Lambda)$ contained in a lattice, then $\Lambda$ can be twisted into an approximate lattice not contained in a model set.
\end{theorem}

Meyer's theorem therefore generalises to $S$-adic linear groups, unless $\Lambda$ factors to a rank one lattice in which case it never does. Our main result (Theorem \ref{THEOREM: MAIN THEOREM}) nevertheless describes the structure of \emph{all}  approximate lattices in $S$-adic linear groups thus providing a generalisation of Meyer's theorem even when it seems to fail.  The methods are robust and can be generalised to positive characteristic,  Lie groups and beyond \S \ref{Subsection: Remarks and consequences}.

\subsection{Historical background and main result}

Approximate lattices of Euclidean spaces - also known as Meyer sets \cite{moody1997meyer} - are central objects of study in aperiodic order (i.e. the study of aperiodically ordered sets). They describe a number of fascinating aperiodic objects such as Penrose's aperiodic tiling (P3) \cite{deBruijn1981algebraicI, deBruijn1981algebraicII}, the recently discovered `hat' tiling \cite{baake2023dynamics}, certain mathematical models of quasi-crystals \cite{MR3136260, NirOlevskii2015PoissonSummation} and, perhaps more surprisingly, the Pisot--Vijayaraghavan--Salem numbers of a given number field \cite{meyer1972algebraic}. But the theory of aperiodic order naturally extends to non-commutative spaces such as the isometry group of $\mathbb{R}^n$ \cite{baake2023dynamics}, the hyperbolic plane, symmetric spaces and buildings \cite{zbMATH02132027, zbMATH01239670, MR1452434} or much more general homogeneous spaces  \cite{cordes2020foundations,zbMATH00109492, zbMATH02072462}.

With this in mind, Bj\"{o}rklund and Hartnick pondered over the possibility of a structure theory \emph{à la} Meyer in non-commutative spaces. They generalized in \cite{bjorklund2016approximate} Meyer's cut-and-project construction and launched the systematic study of approximate lattices, motivated by the following fundamental question.

\begin{question}[Problem 1, \cite{bjorklund2016approximate}]\label{Question: Structure approximate lattices} 
Let $\Lambda$ be an approximate lattice. Is $\Lambda$ commensurable with a model set? 
\end{question}

Two approximate subgroups $\Lambda_1,\Lambda_2 \subset G$ are \emph{commensurable} if there is $F \subset G$  finite such that $\Lambda_1 \subset F\Lambda_2$ and $\Lambda_2 \subset F\Lambda_1$. Following \cite{hrushovski2020beyond}, we call an approximate lattice \emph{laminar} if it is commensurable with a model set.

Two bodies of work offered particularly optimistic perspectives towards the resolution of Question \ref{Question: Structure approximate lattices}. First of all, the striking developments of the theory of finite approximate subgroups of non-commutative groups \cite{MR3348442, MR2833482} culminating in the Breuillard--Green--Tao theorem \cite{MR3090256}. Results from the finite framework can indeed be adapted in many ways \cite{MR3345797, machado2020apphigherrank}.  In parallel, the fascinating rigidity results enjoyed by irreducible lattices in products of groups (see for instance \cite{MR4036451, zbMATH07180377, BurgerMonod, zbMATH05013729, zbMATH01656531}).  Meyer's intuition indeed sees approximate lattices as originating from a lattice in a product $G \times H$ where we are given no (or only partial) knowledge of $H$. Through that, our approach is particularly influenced by Burger and Monod's work on bounded cohomology \cite{BurgerMonod}.

The first results towards an answer to Question \ref{Question: Structure approximate lattices} were established in \cite{machado2020approximate, machado2019infinite}. The author generalised there Meyer's theorem to nilpotent and solvable Lie groups by adapting Meyer's (and Schreiber's) original approach \cite{meyer1972algebraic, schreiber1973approximations}.  The first  paradigm change appeared in \cite{machado2019goodmodels} where tools from finitary additive combinatorics led to extensions of Meyer's theorem to approximate lattices in amenable groups. This approach also offered new insight even when specialised to approximate lattices in locally compact abelian groups. Meyer's theorem was then further extended to approximate lattices of semi-simple algebraic groups in a breakthrough paper by Hrushovski \cite{hrushovski2020beyond}. The absence of amenability however required drastically new ideas embodied in Hrushovski's \emph{quasi-model theorem}  \cite[Thm 4.2]{hrushovski2020beyond}. His work answered Question \ref{Question: Structure approximate lattices} in a particularly difficult case and, at the same time, generalised Margulis' arithmeticity theorem to the aperiodic framework (we also mention an ergodic-theoretic approach  \cite{machado2020apphigherrank}).  As a complement to these results,  it was shown in \cite{machado2019goodmodels} that approximate lattices of linear groups could be decomposed into an amenable part and a semi-simple part.  At first glance, this hinted at the possibility of an affirmative answer to Question \ref{Question: Structure approximate lattices} in the generality of $S$-adic linear groups.

 Hrushovski however built the first example of non-laminar approximate lattice in a central extension of $\SL_2(\mathbb{R})$ \cite[\S 7.9]{hrushovski2020beyond}.  This indicated that extensions of extremely regular approximate lattices could behave surprisingly and led Hrushovski to the following refinement of Question \ref{Question: Structure approximate lattices}:
 \begin{question}[Question 7.11, \cite{hrushovski2020beyond}]\label{Question: Hrushovski}
 Given $\Lambda$ an approximate lattice of an $S$-adic algebraic group $G$, write $\Gamma_s$ the projection of $\langle \Lambda \rangle$ to $G / \Rad(G)$.  Is it true that all approximate lattices projecting to $\Gamma_s$ are laminar if and only if the second bounded cohomology $H^2_b(\Gamma_s;V)$ of $\Gamma_s$ with coefficients in $V$  vanishes for all finite dimensional unitary representations $V$ of $\Gamma_s$? 
 \end{question}

Going further than the scope of Question \ref{Question: Structure approximate lattices},  our main result (Theorem \ref{THEOREM: MAIN THEOREM}) provides a complete structure theorem for approximate lattices in $S$-adic linear groups even when Question \ref{Question: Structure approximate lattices} admits a negative answer. We also answer Question \ref{Question: Hrushovski} in the negative but show that laminarity can indeed be detected by a different cohomological criterion.  Our main result parametrises \emph{extensions}. 

Namely, consider a Zariski-connected $S$-adic linear group $G$ with radical $\Rad(G)$ and $\Lambda \subset G$ a Zariski-dense approximate lattice. Write $\Lambda_s \subset G/R$ the projection of $\Lambda$ and $\Lambda_r := \Lambda^2 \cap R$ which are approximate lattices in $G/R$ and $R$ respectively (Theorem \ref{Theorem: Radical is hereditary}). Then any approximate lattice whose projection to $G/\Rad(G)$ and intersection with $\Rad(G)$ are commensurable with $\Lambda_s$ and $\Lambda_r$ is an \emph{extension of $\Lambda_s$ by $\Lambda_r$}. We denote by $\Comm_G(\Lambda_s;\Lambda_r)$ the space of such extensions modulo commensurability and conjugacy.

\begin{theorem}[Main Theorem]\label{THEOREM: MAIN THEOREM}
Let $G,\Lambda_s, \Lambda_r$ be as above. Let $P$ be the family of rank one factors $p: G/R \rightarrow S$ of $G/R$ such that $p(\Lambda_s)$ is contained in a lattice of $S$ and let $\pi$ be the diagonal map. Define the subgroup $Z_{G,1}$ of $Z$ made of all the elements centralising $\ker \pi$. Write $\Gamma_1$ the lattice $\pi(\langle \Lambda_s \rangle)$. Then:
\begin{enumerate}
\item \emph{(Physical + internal parametrisation)} there is a \emph{bijective} natural map 
$$ \Phi: \left(H^1(\Gamma_1; Z_{G,1}) \times \widetilde{QC}(\Gamma_1; V) \right)/H^1(\Gamma_1; \Delta) \rightarrow \Comm_G(\Lambda_s;\Lambda_r)$$
where $(Z_{G,1},V, \Delta)$ denotes an explicit cut-and-project scheme related to $\Lambda_r \cap Z_{G,1}$;

\item \emph{(Laminarity)}
$\Phi(\alpha)$ is laminar if and only if the class of $d\alpha$ in $H^2_b(\Gamma_1;V)$ is trivial; 
\item \emph{(Splitting)} $\Phi(\alpha)$ splits - i.e. for all $\Lambda_0 \in \Phi(\alpha)$ such that $\Lambda_0^2 \cap S_0$ is an approximate lattice for some Levi subgroup $S_0$ - if and only if $\alpha$ is trivial. 
\end{enumerate}
\end{theorem}

The key takeaway is that the structure of an approximate lattice is naturally determined by a quasi-cocycle unique modulo coboundaries.  Theorem \ref{THEOREM: MAIN THEOREM} has many consequences besides Theorem \ref{THEOREM: FIRST THEOREM}  such as the existence of \emph{uncountably many} (abstract) commensurability classes of non-laminar approximate lattices,  an arithmeticity theorem for higher-rank perfect linear groups (Theorem \ref{Theorem: Arithmeticity for perfect groups}) and existence of approximate subgroups with exotic algebraic and ergodic-theoretic behaviour \cite[Prop. 1.2]{mac2023definitions}.

\subsection{Detail of the paper}
We give below a brief breakdown of Theorem \ref{THEOREM: MAIN THEOREM}. It is the amalgamation of two distinct results: a parametrisation statement and a higher-rank vanishing theorem.

\subsubsection{Parametrisation by quasi-cocycles}
To understand the structure of elements in $\Comm(\Lambda_s, \Lambda_r)$ we proceed at first as with group extensions.  Given an approximate subgroup $\Lambda$ projecting to $\Lambda_s$, we consider a section of the projection $\Lambda \rightarrow \Lambda_s$.  This section satisfies functional identities - see \eqref{Eq: quasi-cocycle equation} - reminiscent of quasi-morphisms. Using the intuition provided by Meyer's theorem and partial arithmeticity results, one can show the section yields a \emph{quasi-cocycle} $q$ of $\Lambda_s$ taking value in a finite dimensional vector space $V$ built as the internal space of $\Lambda^2 \cap Z$ where $Z$ is the centre of $\Rad(G)$ i.e. : 
$$\forall m \geq 0, \sup_{\lambda_1, \lambda_2 \in \Lambda_s^m}|q(\lambda_1\lambda_2) - q(\lambda_1)  - \lambda_1 \cdot q(\lambda_2)| \leq  C(m) < \infty.$$
We obtain that way a parametrisation of extensions in terms of the space $\widetilde{QC}(\Lambda_s,V)$ of quasi-cocycles of $\Lambda_s$ modulo bounded maps.

\subsubsection{Cohomology for approximate subgroups}
These quasi-cocycles are better understood as part of a more general framework.  We introduce below a notion of cohomology tailored to the study of approximate subgroups.  It is defined thanks to a complex made of those group cochains of $\langle \Lambda \rangle$ that are bounded on $\Lambda$ and its powers (but may be unbounded over $\langle \Lambda \rangle$).  Equipped with the  inhomogeneous differential this provides a notion of cohomology that naturally sits between group cohomology and bounded cohomology i.e. 
$$H^*_b(\langle \Lambda \rangle ;V) \rightarrow H^*_b(\Lambda;V) \rightarrow H^*(\langle \Lambda \rangle;V).$$
The relation between approximate subgroups and bounded cohomology was first observed in \cite{hrushovski2020beyond} and we argue here that the image of $H^2_b(\langle \Lambda \rangle;V)$ in $H^2_b(\Lambda;V)$ is the correct object to consider.  This idea enables us to complement Hrushovski's quasi-model theorem with a universality statement (see \S \ref{Section: Cohomology for approximate subgroups} for related results due to Krupiński and Pillay). This yields in turn that laminarity (or lack thereof) is precisely measured by the class of $dq$ in $H_b^2(\Lambda;V)$,  see (2) of Theorem \ref{THEOREM: MAIN THEOREM}.

\subsubsection{Higher-rank vanishing}
As might perhaps be surprising, the space $\widetilde{QC}(\Lambda_s;V)$ does not appear in the statement of Theorem \ref{THEOREM: MAIN THEOREM} - only the quasi-cocycles of rank one factors seem to be relevant.  This is in fact a consequence of the second main result of this paper,  a vanishing theorem for quasi-cocycles under a higher-rank condition:
\begin{theorem}[Vanishing]\label{Theorem: Vanishing in higher-rank}
Let $\Lambda_s$ be an approximate lattice in a semi-simple $S$-adic linear group $L$.  Suppose that there is no rank one factor $p:L \rightarrow S_0$ such that $p(\Lambda_s)$ is contained in a lattice. Then for all finite dimensional vector spaces $B$ over some local field equipped with an action of $\langle \Lambda_s \rangle$ by linear operator such that $\Lambda_s$ acts as a family of uniformly bounded operators, $$\widetilde{QC}(\Lambda_s;B) =\{0\}.$$
\end{theorem}

The finite dimensionality assumption can be removed when $L$ has rank at least $2$ but is necessary for $L$'s of rank $1$ as $\widetilde{QC}(\Lambda_s, V)$ may not vanish otherwise.  Theorem \ref{Theorem: Vanishing in higher-rank} is inspired by work of Burger and Monod on bounded cohomology of irreducible lattices \cite{BurgerMonod}. A number of difficulties are however specific to our set-up.  We have to work with $L^p$-induction over non-transitive spaces  and for $p < 1$.  The former is dealt with by relying on the  methods of \cite{machado2020apphigherrank, BjorklundHartnickKarasik},  while we surprisingly exploit operator norm variations to solve the latter.  

\subsubsection{Distortion of approximate lattices}
A crucial ingredient at play in the proof of Theorem \ref{Theorem: Vanishing in higher-rank} is an extension of the Lubotzky--Mozes--Raghunathan theorem \cite{LubotzkyMozesRaghunathan} we obtain below (\S \ref{Section: Distortion of approximate lattices in semi-simple groups}). It implies that approximate lattices satisfy an integrability condition necessary to define an $L^p$-induction scheme for non-uniform approximate lattices.  We do not show however that approximate lattices are undistorted in their ambient group, as this is false for approximate lattices in rank one.

\subsubsection{Rank one examples}\label{Subsubection: Rank one examples}
Let us describe in detail an instance of the above construction. Consider the $\mathbb{Q}[\sqrt{2}]$-quadratic form 
$$q(x,y):=x - \sqrt{2}y.$$
We will build an example in the semi-direct product $\mathbb{O}(q,\mathbb{R}) \ltimes \mathbb{R}^2$. Let $\overline{\cdot}$ denote the Galois conjugation of $\mathbb{Q}[\sqrt{2}]$ i.e. $\overline{\sqrt{2}}=-\sqrt{2}$. Then the isometry groups $\mathbf{O}(q,\mathbb{Q}[\sqrt{2}]), \mathbf{O}(\bar{q},\mathbb{Q}[\sqrt{2}]) \subset \SL_2(\mathbb{Q}[\sqrt{2}])$ are Galois conjugates. The subgroup $\Gamma:=\mathbf{O}(q,\mathbb{Z}[\sqrt{2}])$ embeds as a uniform lattice in $\mathbf{O}(q,\mathbb{R})$ and, via $\overline{\cdot}$, as a dense subgroup of the compact group $\mathbf{O}(\overline{q},\mathbb{R})$. 
Given any class $\alpha \in EH^2_b(\Gamma, \mathbb{R}^2)$ - which is infinite dimensional by \cite{BestvinaBrombergFujiwara2016Bounded} - where $\Gamma$ acts on $\mathbb{R}^2$ through $\mathbf{O}(\overline{q},\mathbb{R})$, there is a quasi-cocycle $b$ of $\Gamma$ taking values in $\mathbb{Z}[\sqrt{2}]^2\subset \mathbb{R}^2$ such that $$\gamma_1,\gamma_2 \mapsto b(\gamma_1\gamma_2) - b(\gamma_1) - \gamma_1 \cdot b(\gamma_2) \in \alpha.$$ This is because $\mathbb{Z}[\sqrt{2}]^2$ is dense in $\mathbb{R}^2$ and stable under the action of $\Gamma$. 

\begin{proposition}[Non-Meyer approximate lattices]\label{Proposition: A counter-example to a general Meyer-type theorem}
The subset
$$\Lambda_{\alpha}= \{(\gamma, b(\gamma) + \xi) \in \mathbf{O}(q,\mathbb{R}) \ltimes \mathbb{R}^2 : \gamma \in \mathbf{O}(q,\mathbb{Z}[\sqrt{2}]), \xi \in \mathbb{Z}[\sqrt{2}]^2 \text{ with } |\overline{\xi}| < 1\}$$
is an approximate lattice.  $\Lambda_{\alpha}$ is not laminar if and only if $\alpha$ is not zero.
\end{proposition}

\subsection{Outline of the paper}
We introduce in \S \ref{Subsection: Notations} the preliminaries needed, ranging from elementary results concerning approximate subgroups, good models and approximate lattices to $S$-adic linear groups and PVS numbers.  We investigate in \S \ref{Section: Cohomology for approximate subgroups} the uniqueness of quasi-models via bounded cohomology for approximate subgroups.  We study in \S \ref{Section: Distortion of approximate lattices in semi-simple groups} the geometric properties of arithmetic approximate lattices with the generalisation of the Lubotzky--Mozes--Raghunathan theorem as our goal.  In \S \ref{Section: Conjugation-multiplication phenomena} we unveil the link between approximate lattices in $S$-adic linear groups with radical and quasi-cocycles. We provide a first glance at the endgame and exploit the characterisation of laminarity through cohomology. We then explain in \S \ref{Section: Induction for star-approximate lattices and Boundedness of quasi-cocycles} how to induce quasi-cocycles to the ambient group and enforce the strategy due to Burger and Monod \cite{BurgerMonod}, proving Theorem \ref{Theorem: Vanishing in higher-rank}. We finally conclude the proof of Theorem \ref{THEOREM: MAIN THEOREM} in \S \ref{Section: Proof of the main theorems}. 

\subsection{Acknowledgements}

I am deeply grateful to Emmanuel Breuillard for inspiring discussions and constant support. I owe an intellectual debt to Michael Bj\"{o}rklund, Tobias Hartnick and Ehud Hrushovski who were kind enough to share some of their expert knowledge and understanding of the topic on multiple occasions. I wish to thank Marc Burger, Jack Button, Alexis Chevalier, Alex Gamburd, Sebastian Hurtado, Yifan Jing, Constantin Kogler,  Krzysztof Krupiński, Thibaut Mazuir and Peter Sarnak for interesting and insightful discussions at various stages of this project. This material is based upon work supported by the National Science Foundation under Grant No. DMS-1926686.

\section{Preliminaries}\label{Subsection: Notations}
Given subsets $X$ and $Y$ of $G$ define $XY:=\{xy : x \in X, y \in Y\}$, $X^0:=\{e\}$, $X^1:=X$ and $X^{n+1}=X^nX$ for all $n \geq 0$. Write also $\langle X \rangle$ the subgroup generated by $X$. Note that when $X=X^{-1}$, $\langle X \rangle = \bigcup_{n \geq 0} X^n$. We also define $X^y:=y^{-1}Xy$ and $^{y}X:=yXy^{-1}$. 
\subsection{Approximate subgroups and commensurability}
An element $g \in G$ \emph{commensurates} $X$ if $gXg^{-1}$ and $X$ are commensurable. Then $\Comm_G(X) \subset G$ denotes the subgroup of elements commensurating $X$.  Similarly,  $\alpha \in \Aut(G)$ \emph{commensurates} $X$ if $\alpha(X)$ and $X$ are commensurable. Commensurability is an equivalence relation and it is useful to think of the condition $X \subset FY \cap YF$ as a conjunction of $X \subset FY$ and $X \subset YF$.  Two symmetric subsets $X$ and $Y$ (e.g.  approximate subgroups) are commensurable if and only if there is $F \subset G$ finite such that $X \subset FY$ and $Y \subset FX$.  We present well-known and useful results, see e.g. \cite[\S 2.1]{machado2019goodmodels} and references therein for detail.

\begin{lemma}[Intersections and commensurability]\label{Lemma: Intersection of commensurable sets}
Take $X,Y_1,\ldots Y_n$ subsets of a group $G$. Assume that there exist $F_1, \ldots, F_n \subset G$ finite such that $X \subset F_iY_i$ for all $i \in \{1,\dots, n\}$. Then there is $F' \subset X$ with $|F'| \leq |F_1|\cdots |F_n|$ such that $$X \subset F'\left( Y_1^{-1}Y_1 \cap \cdots \cap Y_n^{-1}Y_n\right).$$ 
\end{lemma}

We will often use Lemma \ref{Lemma: Intersection of commensurable sets} in the context of an action: 

\begin{corollary}\label{Corollary: Intersection with stabiliser}
Let $X$ be a subset of a group $G$. Suppose that $G$ acts on a set $S$ and that there is $s \in S$ such that $X \cdot s$ is finite. Then $X$ is covered by finitely many translates of $X^{-1}X \cap \Stab(s)$. 
\end{corollary}

Lemma \ref{Lemma: Intersection of commensurable sets} also yields information on intersections of approximate groups. 

\begin{lemma}\label{Lemma: Intersection of approximate subgroups}
Let $\Lambda_1, \ldots, \Lambda_n$ be $K_1,\ldots,K_n$-approximate subgroups of some group. We have: 
\begin{enumerate}

\item if $k_1, \ldots, k_n \geq 2$, then $\Lambda_1^{k_1} \cap \cdots \cap \Lambda_n^{k_n}$ is a $K_1^{2k_1-1}\cdots K_n^{2k_n-1}$-approximate subgroup;
\item if $\Lambda_1', \ldots, \Lambda_n'$ is a family of approximate subgroups such that $\Lambda_i'$ is commensurable with $\Lambda_i$ for all $1 \leq i \leq n$, then $\Lambda_1'^2 \cap \cdots \cap \Lambda_n'^2$ is commensurable with $\Lambda_1^2\cap \cdots \cap \Lambda_n^2$.
\end{enumerate}

\end{lemma}

 \subsection{Good models}\label{Subsubsection: Good models}
 An effective way to build approximate subgroups goes through the construction of \emph{good models}. Let $\Lambda$ be an approximate subgroup of a group $\Gamma$ that commensurates it. A group homomorphism $f: \Gamma \rightarrow H$ with target a locally compact group $H$ is called a \emph{good model (of $(\Lambda, \Gamma)$)} if:
\begin{enumerate}
\item $f(\Lambda)$ is relatively compact;
 \item there is $U \subset H$ a neighbourhood of the identity such that $f^{-1}(U) \subset \Lambda$.
\end{enumerate}
Any approximate subgroup commensurable with an approximate subgroup that admits a good model is said \emph{laminar}. The introduction of good models is justified by:
\begin{proposition}[\S 3.5, \cite{machado2019goodmodels}]\label{Proposition: Equivalence good models and model sets}
An approximate lattice $\Lambda$ in some locally compact group is laminar if and only if it is commensurable with a model set.
\end{proposition}

In \cite{MR2833482}, \cite{MR3090256} and \cite{MR3345797} a more detailed definition is written. This other definition takes into account potential extra structure. In what follows we will only need a few instances of good models enjoying extra structure - and the fact that they do is automatic.

\begin{lemma}[Bohr-type compactification,  \S 4.4, \cite{machado2019goodmodels}]\label{Lemma: Bohr compactification and abstract automorphisms}
Let $\Lambda$ be an approximate subgroup of a group $\Gamma$ such that $Comm_{\Gamma}(\Lambda)=\Gamma$. There is a group homomorphism $f_0: \Gamma \rightarrow H_0$ with $H_0$ locally compact and $\overline{f_0(\Lambda)}$ compact such that for all other group homomorphisms $f: \Gamma \rightarrow H$ with $H$ locally compact and $f(\Lambda)$ relatively compact, there is a continuous group homomorphism $\phi: H_0 \rightarrow H$ such that $f=\phi \circ f_0$.

In particular, for any group endomorphism $a$ of $\Gamma$ such that $a(\Lambda)$ is commensurable with $\Lambda$ there is a unique continuous group endomorphism $\alpha$ of $H_0$ such that the following diagram commutes
\[\begin{tikzcd}
\Gamma \arrow{r}{f} \arrow[swap]{d}{a} & H_0 \arrow{d}{\alpha} \\
\Gamma \arrow{r}{f} & H_0
\end{tikzcd}
\]

Moreover, if $\Lambda$ has a good model, then $f_0$ is a good model of $\Lambda$. 
\end{lemma}
 And a counterpart for the `smallest' good model: 
 
\begin{proposition}[Prop 3.6, \cite{machado2019goodmodels}]\label{Proposition: Minimal commensurable approximate subgroup}
Let $\Lambda$ be an approximate subgroup of some group. Suppose that $\Lambda$ is laminar. Then there is an approximate subgroup $\Lambda'$ commensurable with $\Lambda$ and a good model $f: \langle \Lambda' \rangle \rightarrow H$ with target a connected Lie group and dense image. Moreover: 
\begin{enumerate}
\item if $\Lambda'' \subset \Lambda'$ is any approximate subgroup commensurable with $\Lambda'$, $\langle \Lambda'' \rangle$ has finite index in $\langle \Lambda' \rangle$;
\item we can choose $H$ without compact normal subgroup and such an $H$ is unique. 
\end{enumerate}
\end{proposition}

\begin{proof}
(2) is proved in \cite{machado2019goodmodels}.  Let $\Lambda''$ and $\Lambda'$ be as in (1). Choose $F \subset \langle \Lambda' \rangle$ such that $\Lambda'\subset F\Lambda''$. Let $W \subset H$ be a neighbourhood of the identity such that $f^{-1}(W) \subset \Lambda'$. We know that $\overline{f(\Lambda'')}$ has non-empty interior by the Baire category theorem. So $f(\Lambda'')$ generates a dense subgroup of $H$. Therefore, $$\langle \Lambda' \rangle \subset f^{-1}(W)\langle \Lambda'' \rangle \subset F\langle \Lambda''\rangle.$$

\end{proof}

Laminar approximate subgroups are moreover stable under many operations. For instance, both homomorphic images and inverse images of laminar sets are laminar \cite[Lem. 3.3]{machado2019goodmodels}.

 \subsection{Quasi-homomorphisms as models}\label{Subsubsection: Quasi-homomorphisms as models}
 Unfortunately, not all approximate subgroups have good models. Counter-examples were found in \cite{hrushovski2019amenability}, \cite{machado2019goodmodels} and \cite{hrushovski2020beyond} and all involved the notion of \emph{quasi-homomorphism} in an essential way.

 \begin{definition}\label{Definition: Hrushovski's quasi-morphisms}
 Let $f: \Gamma \rightarrow H$ be a map between a group and a locally compact group. We say that $f$ is a \emph{quasi-homomorphism} if the defect $D:=\{f(\gamma_1\gamma_2)f(\gamma_2)^{-1}f(\gamma_1)^{-1} : \gamma_1, \gamma_2\in \Gamma\}$ of $f$ is contained in a compact subset $K$ of $H$ that is \emph{normal} in $H$ i.e. $ \forall h \in H, K^h=K$. 
 \end{definition}

Quasi-homomorphisms are studied in \cite{hrushovski2020beyond} by way of considering the subgroups $A$ topologically generated by normal compact subsets. These subgroups are coined \emph{rigid subgroups} in \cite[App. C]{hrushovski2020beyond} and we will encounter them in this work as well.  Quasi-homomorphisms are related to approximate groups in the same way group homomorphisms are related to groups:
 
 \begin{lemma}\label{Lemma: Graphs of quasi-homomorphisms, continuous case}
 Let $f: \Gamma  \rightarrow H$ be a symmetric quasi-homomorphism with defect $D$ contained in $K$ a  neighbourhood of the identity in $A:=\overline{\langle D \rangle}$ that is normal in $H$. Then $\Gamma_{f,K}:=\{(\gamma,f(\gamma)k) : \gamma \in \Gamma, k\in K\}$ is an approximate subgroup, $f(\Gamma)K$ is an approximate subgroup and for all relatively compact symmetric neighbourhoods of the identity $W_1, W_2 \in H$, the subsets $f^{-1}(W_1K^{-1})$ and $f^{-1}(W_2K^{-1})$ are commensurable approximate subgroups.
 \end{lemma}
 This can be deduced from results in \cite[\S 5.11]{hrushovski2020beyond}.
 A deep result of Hrushovski asserts that this relationship can be reversed.
 
 \begin{theorem}[Quasi-model theorem, Thm 4.1 \cite{hrushovski2020beyond}]\label{Theorem: Hrushovski's quasi-models}
 Let $\Lambda$ be an approximate subgroup of some group $\Gamma$ that commensurates it. There is a quasi-homomorphism $f:\Gamma  \rightarrow H$ with defect contained in a normal compact subset $K \subset H$ such that for all relatively compact neighbourhoods of the identity $W \subset H$, $\Lambda$ is commensurable with $f^{-1}(WK)$.

 Moreover, by restricting $f$ to a finite index subgroup $\Gamma_0 \subset \Gamma$ we may assume that $A:=\overline{\langle K \rangle}$ is isomorphic to a Euclidean space on which $G$ acts by isometries  and $K \subset A$ is a Euclidean ball about the identity, $f(\Gamma_0)$ is relatively dense in $H$ and $\overline{f(\Gamma_0)K}$ projects surjectively to $H/A$.We will call such a map a \emph{quasi-model (of $(\Lambda,\Gamma)$)}.
 \end{theorem}

 The moreover part is obtained as a combination of \cite{hrushovski2020beyond} and \cite[Lem. 2.8]{mac2023definitions}.  Since the defect of $f$ is contained in $A$, the subset $f(\Gamma_0)A$ is a subgroup.  This can be used to derive a number of properties of $f$.  Our point-of-view will however be to study the approximate subgroups $\Gamma_{f,K}$ and $f(\Gamma)K$ as we are able to draw more precise information from their structure.

 \subsection{Approximate subgroups of automorphism groups}
 We look now at the interaction of two approximate subgroups of topological groups: one acting upon the other through continuous automorphisms.  We equip the group of automorphisms with a natural topology called the \emph{Braconnier topology}. A neighbourhood basis of the Braconnier topology is defined as follows: for all $K \subset G$ compact and all $U \subset G$ open, 
 $$\mathcal{U}(K,U):=\{\alpha \in \Aut(G): \forall k \in K, \alpha(k)k^{-1} \in U \wedge\alpha^{-1}(k)k^{-1} \in U\}.$$
Approximate subgroups that are pre-compact in the Braconnier topology are close to being relatively compact.

\begin{lemma}[Lem.  A.8,  \cite{mac2023sadic}]\label{Lemma: Bounded approximate subgroups of of automorphisms are relatively compact}
Let $G$ be a locally compact group and $W$ be a neighbourhood of the identity. Let $\mathcal{A}$ be an approximate subgroup of the group $\Aut(G)$ of automorphisms of $G$ such that $\mathcal{A}\cdot W$ is relatively compact and $\overline{\langle \mathcal{A} \rangle}$ (in the Braconnier topology) contains the inner automorphisms. Then there is $C$ a compact normal subgroup stable under $\mathcal{A}$ such that the image of $\mathcal{A}$ in $\Aut(G/C)$ is relatively compact.
\end{lemma}

This point-of-view will be useful when considering intersections of approximate subgroups with normal subgroups. The following simple corollary provides a converse to that observation. 

\begin{corollary}[Cor. A.9, \cite{mac2023sadic}]\label{Corollary: Bounded approximate subgroups of of automorphisms are relatively compact}
Let $\Lambda$ be an approximate subgroup generating a group $\Gamma$. Let $\Lambda_N$ be another approximate subgroup of $\Gamma$ that has a good model. Suppose that $\Lambda_N^{\Lambda}=\bigcup_{\lambda \in \Lambda}\lambda\Lambda_N\lambda^{-1}$ is commensurable with $\Lambda_N$. Then there is $\Lambda'$ commensurable with $\Lambda$ that normalises a subgroup $N$ containing $\langle \Lambda_N \rangle$ as a finite index subgroup. 
\end{corollary}

We conclude this section going further in this direction with a result that enables us to perform many reductions in the proof of (2) of Theorem \ref{THEOREM: MAIN THEOREM}.

\begin{lemma}\label{Lemma: Laminarity from laminarity in decompositions}
Let $\Lambda$ be an approximate lattice in a locally compact group $G$. Suppose that $G=HL$ where $H$ is a closed subgroup and $L$ is a closed normal subgroup. Suppose moreover that $H \cap \Lambda^2$ and $L \cap \Lambda^2$ are both approximate lattices in $H$ and $L$ respectively. Then $\Lambda$ is laminar if and only if $H \cap \Lambda^2$ and $L \cap \Lambda^2$ are laminar. 
\end{lemma}

\begin{proof}
Suppose first that $\Lambda$ is Meyer. Then $\Lambda$ is commensurable with $\Lambda'$ that has a good model. So $\Lambda'^2 \cap H$ and $\Lambda'^2 \cap L$ are approximate subgroups that have good models (\cite[Lem. 3.3]{machado2019goodmodels}). According to Lemma \ref{Lemma: Intersection of approximate subgroups},  $\Lambda^2 \cap H$ and $\Lambda^2 \cap L$ are laminar. 

Conversely, suppose that $\Lambda^2 \cap H$ and $\Lambda^2 \cap L$ are laminar. We have
$$ \bigcup_{\lambda \in \Lambda} \left(\Lambda^2 \cap L\right)^\lambda \subset \Lambda^4 \cap L$$
which is commensurable with $\Lambda^2 \cap L$ (Lemma \ref{Lemma: Intersection of approximate subgroups}). Let $f: \langle \Lambda \rangle \cap L \rightarrow H_0$ be a good model of an approximate subgroup commensurable with $\Lambda^2 \cap L$ and suppose that $f$ is a Bohr compactification (Lemma \ref{Lemma: Bohr compactification and abstract automorphisms}). For every group automorphism $\alpha$ of $\langle \Lambda \rangle \cap L$ commensurating $\Lambda^2 \cap L$, there is a unique continuous group automorphism $\overline{\alpha}$ of $H_0$ such that $f \circ \alpha = \overline{\alpha} \circ f$. Given any $\gamma \in \langle \Lambda \rangle$, let $c(\gamma) \in \Aut(\langle \Lambda \rangle \cap L)$ be the action by conjugation by $\gamma$ restricted to $\langle \Lambda \rangle \cap L$. Then $ \bar{c}: \langle \Lambda \rangle \rightarrow \Aut(H_0) $ defined by $\gamma \mapsto \overline{c(\gamma)}$ is a group homomorphism. And $\bar{c}(\Lambda)\left(\overline{f\left(\Lambda^2 \cap L\right)}\right) \subset \overline{f\left(\Lambda^4 \cap L\right)}$ which is compact by Lemma \ref{Lemma: Intersection of approximate subgroups}. According to Lemma \ref{Lemma: Bounded approximate subgroups of of automorphisms are relatively compact} there is a normal compact subgroup $C \subset H_0$ invariant under $\bar{c}(\langle \Lambda \rangle)$ such that the projection of $\bar{c}(\Lambda)$ to $\Aut(H_0/C)$ is relatively compact. Upon quotienting out $C$, we may assume that $C=\{e\}$ and $\bar{c}(\Lambda)$ is relatively compact. Choose a sequence $(\Lambda_{H,n})_{n \geq 0}$ such that $\Lambda_{H,n}^2 \subset \Lambda_{H,n-1}$ and $\Lambda_{H,n-1}$ commensurable with $\Lambda^2 \cap H$ for all $n \geq 1$, which exists because $\Lambda^2 \cap H$ is laminar e.g. \cite[Prop. 1.2]{machado2019goodmodels}. Fix $W$ a symmetric compact neighbourhood of the identity in $H_0$. There are symmetric neighbourhoods of the identity $(W_n)_{n \geq 0}$ with $(\bar{c}(\Lambda_{H,0})(W_n))^2 \subset W_{n-1}$ for all integers $n \geq 1$ and $W_0=W$. Define $\Lambda_L := f^{-1}(W_0)$ and $\Lambda_{L,n}:=f^{-1}(W_n)$ for all $n \geq 0$. Set now $\Lambda_n := \Lambda_{H,n}\Lambda_{L,n} \cup \left(\Lambda_{H,n}\Lambda_{L,n}\right)^{-1}$. Our choice of $\Lambda_{H,n}$ and $\Lambda_{L,n}$ implies that $\Lambda_n^2 \subset \Lambda_{n-1}$ for all $n \geq 1$.

If we can show that $\Lambda_n$ is commensurable with $\Lambda$ for all $n \geq 0$, then $\Lambda$ is laminar by a folklore characterisation of good models (e.g. \cite[Prop. 1.2]{machado2019goodmodels}). According to Proposition \ref{Proposition: Intersection and projections approximate lattices w/ closed subgroups}, because $\Lambda^2 \cap L$ is an approximate lattice, the projection $\Lambda'$ of $\Lambda$ to $G/L$ is an approximate lattice. So the projection  $\Lambda''$ of $\Lambda_{H,n}$ to $H/H \cap L$ is uniformly discrete. Hence, it is an approximate lattice (Proposition \ref{Proposition: Intersection and projections approximate lattices w/ closed subgroups} again). Identifying $H/H\cap L$ and $G/L$, we find that $\Lambda'$ and $\Lambda''$ are commensurable by \cite[Lem.  A.4]{hrushovski2020beyond}. So there is a finite subset $F \subset \langle \Lambda \rangle$ such that for every $\lambda \in \Lambda$ we have $\lambda f^{-1}\lambda'^{-1} \in L$ for some $\lambda' \in \Lambda_{H,n}$ and some $f \in F$. So $\Lambda$ is covered by finitely many translates of $\Lambda_{H,n}\left(\Lambda F^{-1}\Lambda_{H,n} \cap L\right)$. But $\Lambda F^{-1}\Lambda_{H,n} \cap L$ is commensurable with $\Lambda^2 \cap L$ and, thus, $\Lambda_{L,n}$ (Lemma \ref{Lemma: Intersection of approximate subgroups}). So $\Lambda$ is covered by finitely many translates of $\Lambda_{H,n}\Lambda_{L,n}$.\end{proof}
\subsection{$S$-adic linear groups}
Let $S \subset S_{\mathbb{Q}}$ be a finite subset of inequivalent places of $\mathbb{Q}$. An $S$-adic linear group, or $S$-adic group, is any group $G$ such that there is a family of linear algebraic $(\mathbb{Q}_v)_{v \in S}$-groups $(\mathbf{G}_{v})_{v \in S}$ such that $G = \prod_{v \in S} \mathbf{G}_{v}(\mathbb{Q}_v)$ equipped with the obvious Hausdorff topology. Given a $\mathbb{Q}$-linear group $\mathbf{G}$, the group of adelic points $\mathbf{G}(\mathbb{A}_S)$ is a typical example of an $S$-adic group. If $X \subset G$ is any subset, we define its \emph{Zariski-closure} as the product of the Zariski-closures of the $p_v(X)$ where $p_v: G \rightarrow \mathbf{G}_{v}(\mathbb{Q}_v)$ denotes the natural projection. All topological properties (closed, open, connected, etc) are understood in the latter topology unless they are preceded by the prefix 'Zariski'. In addition, $\overline{X}$ denotes the closure in the Hausdorff topology unless stated otherwise. If P is a property of algebraic groups, we will say that $G$ has P if and only if all $\mathbf{G}_v$ have P. For instance, we will say that $G$ is semi-simple (potentially with non-trivial centre) if and only if $\mathbf{G}_{v}$ is a semi-simple $k_v$-group for all $v \in S$. We write $C_G(X)$ the centraliser of $X$ in $G$. 
Define $\Rad(G)$ as the maximal Zariski-connected soluble normal subgroup of $G$. We call $\Rad(G)$ the \emph{radical} of $G$. It is equal to $\prod_{v \in S} \Rad(\mathbf{G}_{v}(\mathbb{Q}_v))$. When $G$ is semi-simple, we define the $S$-rank of $G$ as 
$ \rank_S(G) :=\sum_{v \in S} \rank_{K_v} (\mathbf{G}_{v}).$ For a general introduction to such groups see \cite{MR1090825}.

According to the Levi decomposition theorem \cite[VIII Theorem 4.3]{MR620024}, if $G$ is Zariski-connected we have that $G=R \ltimes U$ where $R$ is reductive and $U$ is unipotent. Moreover, there are $S \subset R$ semi-simple and $T \subset R$ a product of tori over the fields $\mathbb{Q}_v$ that centralise one another and such that the map $S \times T \rightarrow R$ has finite kernel (\cite[\S 22.o]{MR3729270}). We will call a \emph{reductive} (resp. \emph{semi-simple}) \emph{Levi subgroup} of $G$ any Zariski-closed subgroup that projects isomorphically to $R$ (resp. $S$). Levi subgroups are characterised as the maximal reductive (resp. semi-simple) subgroups of $G$. Moreover, any two reductive (resp. semi-simple) Levi subgroups are conjugate to one another via an element of $U$. When we write `Levi subgroup' without precision, we mean semi-simple Levi subgroup.

For every $v \in S$ take $\mathbf{H}_v \subset \mathbf{G}_v$ an algebraic $\mathbb{Q}_v$-subgroup. If $H$ denotes $\prod_{v \in S} \mathbf{H}_v(\mathbb{Q}_v)$, then the natural map $G/H \rightarrow \prod_{v \in S} (\mathbf{G}_v/\mathbf{H}_v)(\mathbb{Q}_v)$ is injective and has finite index image \cite{BorelSerre}. When $G$ and $H$ are assumed unipotent, it becomes a continuous isomorphism. This can be seen through the equivalence between unipotent $S$-adic groups and their Lie algebras that we use a number of times, see \cite[IV.2.4]{zbMATH03670601} for references and \cite[II]{raghunathan1972discrete} for the case of real groups. We refer to \cite{MR3729270, springer2010linear} for background on algebraic groups, to \cite{raghunathan1972discrete} for background on Lie groups and their lattices, to \cite{MR1090825, MR3092475} for background on $S$-adic groups and \cite{https://doi.org/10.48550/arxiv.2204.01496, mac2023sadic} for background on approximate lattices in $S$-adic groups.

  \subsection{Pisot--Vijayaraghavan--Salem numbers of a number field}\label{Subsubsection: Pisot--Vijayaraghavan--Salem numbers of a number field}
 
 Let $K$ be a number field. Write $S_K$ the set of all equivalence classes of places of $K$. For any $v \in S_K$ let $|\cdot|_v$ denote an absolute value arising from $v$. The completion of $K$ with respect to $|\cdot|_v$ will be denoted $K_v$. Note that the space obtained is independent of the choice of $|\cdot |_v$. When $v$ is non-Archimedean let us denote by $O_v$ the valuation ring of $K_v$. We write $\mathbb{A}_K = \prod_{v \in S_K}'K_v$ the ring of adeles with the usual topology. Here, $\prod'$ denotes the restricted product with respect to the valuation rings $O_v$. If $S$ denotes a subset of $S_K$, then we write moreover $\mathbb{A}_{K,S}=\prod_{v \in S}'K_v$ and $\mathbb{A}_K^S = \prod_{v \notin S}'K_v$. In particular, $\mathbb{A}_K= \mathbb{A}_{K,S} \times \mathbb{A}_K^S$.  See \cite{NeukirchAlgebraicNumberTheory} for this and more.

We also define the set $\mathcal{O}_{K,S}$ as the subset of those elements $x$ of $K$ such that $|x|_v \leq 1$ for all $v \notin S$. Again, $\mathcal{O}_{K,S}$ does not depend on the choice of absolute values $(| \cdot |_v)_{v \in S}$. When $S$ contains all the Archimedean places of $S_K$, $\mathcal{O}_{K,S}$ is the \emph{ring of $S$-integers}. If $S$ is equal to the set of all Archimedean places of $K$, $\mathcal{O}_{K,S}=:\mathcal{O}_K$ is the ring of algebraic integers of $K$.  Finally, when $S$ consists of a single real valuation $v$, $\mathcal{O}_{K,S}$ is the \emph{Pisot--Vijayaraghavan--Salem numbers} of $K \subset K_v\simeq \mathbb{R}$ - they admit a fascinating sum-product characterisation due to Meyer \cite[\S II.13]{meyer1972algebraic}. 
When $K = \mathbb{Q}$, we simply write $\mathbb{A}_{S}=\mathbb{A}_{K,S}$, $\mathbb{A}^{S}=\mathbb{A}_{K}^{S}$ and $\mathcal{O}_{K,S}=\mathbb{Z}_S$.

Although not a ring in general, $\mathcal{O}_{K,S}$ is an approximate subgroup stable under products (an \emph{approximate ring} \cite{krupinski2023locally}). This is reflected in:
\begin{lemma}[Lem. 2.2, \cite{mac2023sadic}]\label{Lemma: Pisot approximate rings and polynomials}
Let $K, S$ be as above and let $P \in K[X]$. Then:
\begin{enumerate}
\item $P(\mathcal{O}_{K,S})$ is contained in finitely many translates of $\mathcal{O}_{K,S}$;
\item if, in addition, $P(0) =0$, then there is $\Lambda$ contained in and commensurable with $\mathcal{O}_{K,S}$ such that $P(\Lambda) \subset \mathcal{O}_{K,S}$. Moreover, $\Lambda$ is a model set associated to $(K, \mathbb{A}_{K,S}, \mathbb{A}_K^S)$.
\end{enumerate} 
\end{lemma}

The subsets $\mathcal{O}_{K,S}$ furthermore allow us to build a rich family of uniformly discrete approximate subgroups and model sets of $S$-adic algebraic groups. We follow here the same approach as \cite{https://doi.org/10.48550/arxiv.2204.01496}.

\begin{definition}[\emph{$\mathcal{O}_{K,S}$ points of a linear algebraic group}]
With $K,S$ as above. If $\mathbf{G} \subset \GL_n$ is a $K$-subgroup, then define
$$\mathbf{G}(\mathcal{O}_{K,S}):=\{g \in \mathbf{G}(K) : g - \id ,g^{-1} - \id \in \mathbf{M}_{n\times n}\left(\mathcal{O}_{K,S}\right) \},$$
where $\mathbf{M}_{n\times n}\left(\mathcal{O}_{K,S}\right)$ denotes the set of $n \times n$ matrices with entries in $\mathcal{O}_{K,S}$. 
\end{definition}

 When $S=S_1 \sqcup S_2,$ $\mathbf{G}(\mathbb{A}_{K,S})$ is commensurable with $\mathbf{G}(\mathbb{A}_{K,S_1}) \times \mathbf{G}(\mathbb{A}_{K,S_2}).$
 If $S_{\infty}$ denotes the set of infinite places of $K$, then the diagonal embedding 
 $$\mathbf{G}(\mathcal{O}_{K,S \cup S_{\infty}}) \subset \mathbf{G}(\mathbb{A}_{K,S \cup S_{\infty}})$$
makes $\mathbf{G}(\mathcal{O}_{K,S \cup S_{\infty}})$ into a discrete subgroup. When $\mathbf{G}$ is moreover known to have no $K$-characters, $\mathbf{G}(\mathcal{O}_{K,S \cup S_{\infty}})$ is a lattice in $\mathbf{G}(\mathbb{A}_{K,S \cup S^{\infty}})$ by the Borel--Harish-Chandra theorem \cite{BorelHarish-Chandra}. By applying the general cut-and-project construction and noticing that $\Lambda:=\mathbf{G}(\mathcal{O}_{K,S})$ is a model set associated with the cut-and-project scheme $(\mathbf{G}(\mathbb{A}_{K,S}), \mathbf{G}(\mathbb{A}_{K,S_{\infty}\setminus S}), \mathbf{G}(\mathcal{O}_{K,S \cup S_{\infty}}))$ we have: 

\begin{proposition}[Borel--Harish-Chandra for approximate lattices, Prop. 2.4, \cite{mac2023sadic}]\label{Proposition: BHC}
Let $K$ be a number field and $S$ a set of inequivalent places. Let $\mathbf{G} \subset \GL_n$ be a $K$-subgroup. Then $\mathbf{G}(\mathcal{O}_{K,S}) \subset \mathbf{G}(\mathbb{A}_{K,S})$ is an approximate lattice if and only if $\mathbf{G}$ has no non-trivial $K$-characters. Moreover, if $\mathbf{G}$ is $K$-anisotropic, then $\mathbf{G}(\mathcal{O}_{K,S})$ is uniform. 

The map $\mathbf{G}(\mathcal{O}_{K,S \cup S_{\infty}}) \rightarrow \mathbf{G}(\mathbb{A}_{K,S_{\infty}})$ is a good model of $\mathbf{G}(\mathcal{O}_{K,S})$, $\mathbf{G}(\mathbb{A}_{K,S_{\infty}})$ has finitely many connected components and any approximate subgroup $\Lambda$ commensurable with $\mathbf{G}(\mathcal{O}_{K,S})$ contained in $\langle \mathbf{G}(\mathcal{O}_{K,S}) \rangle$ generates a finite index subgroup of $\mathbf{G}(\mathcal{O}_{K,S \cup S'})$. 
\end{proposition}

Considering in particular the additive group $\mathbf{G}=\mathbf{G}_a$,  Proposition \ref{Proposition: BHC} recovers the fact known to Meyer that $\mathcal{O}_{K,S}$ is a model set in $\mathbb{A}_{K,S}$ \cite{meyer1972algebraic}.
 
 \subsection{Other arithmetic subsets}\label{Subsection: S-adic algebraic groups}  Compared with the PVS construction  (\S \ref{Subsubsection: Pisot--Vijayaraghavan--Salem numbers of a number field}) we give here a more flexible arithmetic construction. 

\begin{definition}[Generalized arithmetic approximate subgroups]\label{Definition: GAAS}
 If $\mathbf{G}$ is a Zariski-connected linear group defined over $\mathbb{Q}$ and $G,H$ are two $S$-adic linear subgroups such that $G \times H \subset \mathbf{G}(\mathbb{A}_S)$, we define a \emph{generalized arithmetic approximate subgroup} (GAAS) as any approximate subgroup $\Lambda$ of $G$ commensurable with $p_{G}\left( \mathbf{G}(\mathbb{Z}_S) \cap G \times W_0\right)$ where $p_{G}: G \times H \rightarrow G$ is the natural projection, $W_0 \subset H$ is a symmetric relatively compact neighbourhood of the identity and $\mathbf{G}(\mathbb{Z}_S)$ is embedded diagonally in $\mathbf{G}(\mathbb{A}_S)$.
 \end{definition}

Generalized arithmetic approximate subgroups are uniformly discrete and extremely regular. They are laminar and share many properties with the arithmetic approximate lattices defined in \S \ref{Subsubsection: Pisot--Vijayaraghavan--Salem numbers of a number field} while being more general.  They are easily seen to occur via product-conjugation phenomena:

\begin{proposition}[Prop. 4.15, \cite{mac2023sadic}]\label{Proposition: Arithmeticity from action on subgroups}
Let $\Lambda$ be an approximate lattice in an $S$-adic linear group $G$. Suppose that $\Lambda$ is Zariski-dense and let $N$ be a normal Zariski-closed Zariski-connected nilpotent $S$-adic linear subgroup such that $\Lambda^2 \cap N$ is an approximate lattice in $N$. Then the projection of $\Lambda$ to $G/C_G(N)$ is a GAAS approximate lattice. 
\end{proposition}

This will be a key tool in the reduction of the proof of Theorem \ref{THEOREM: MAIN THEOREM} to the case of groups with abelian radical.

\subsection{Intersection theorems for approximate lattices}\label{Section: Intersection theorems for approximate lattices}

The following result relating projections of approximate lattices to intersections with kernels will be used repeatedly.

\begin{proposition}[Prop. 6.2, \cite{machado2019goodmodels}]\label{Proposition: Intersection and projections approximate lattices w/ closed subgroups}
Let be  a locally compact group $G$. Let $N$ be a closed normal subgroup and let $p:G \rightarrow G/N$ denote the natural projection.  If $\Lambda \subset G$ is an approximate lattice,  then the following are equivalent: 
\begin{enumerate}
\item $p(\Lambda)$ is an approximate lattice in $G/N$;
\item $p(\Lambda)$ is uniformly discrete;
\item $\Lambda^2 \cap N$ is an approximate lattice in $N$. 
\end{enumerate}
Conversely,  if $\Lambda$ is an approximate subgroup in $G$,  and both $\Lambda^2 \cap N$ and $p(\Lambda)$ are approximate lattices in $N$ and $G/N$ respectively, then $\Lambda$ is an approximate lattice. 
\end{proposition}

We use it in combinations with the fact that uniform approximate lattices intersect centralisers into approximate lattices.  The same goes for radicals.

\begin{lemma}[Lem. 4.4, \cite{mac2023sadic}]\label{Lemma: Intersection with centraliser}
Let $\Lambda$ be a uniform approximate lattice in a locally compact group $G$. Let $\alpha \in \Aut(G)$ be any element commensurating $\Lambda$. Write $E(\alpha)$ the closed subgroup defined by $\{g \in G: \alpha(g)=g\}$. Then $E(\alpha) \cap \Lambda^2$ is a uniform approximate lattice in $E(\alpha)$. 
\end{lemma}

\begin{theorem}[Prop 1.2, \cite{mac2023sadic}]\label{Theorem: Radical is hereditary}
Let $\Lambda$ be an approximate lattice in an $S$-adic linear group $G$ that generates a Zariski-dense subgroup. Then:
\begin{enumerate}
\item \emph{(Radical)} $\Rad(G) \cap \Lambda^2$ is a uniform approximate lattice in $\Rad(G)$;
\item \emph{(Nilpotent radical)} write $N$ the largest Zariski-connected normal nilpotent subgroup of $G$,  $N \cap \Lambda^2$ is a uniform approximate lattice in $N$;
\item \emph{(Unipotent commutators)} $[G,\Rad(G)] \cap \Lambda^2$ is a uniform approximate lattice in $[G,\Rad(G)]$.
\end{enumerate}
\end{theorem}

Considering inner automorphisms, Lemma \ref{Lemma: Intersection with centraliser} and Theorem \ref{Theorem: Radical is hereditary} yield: 

\begin{lemma}[Corollary 6.4, \cite{machado2019goodmodels}]\label{Lemma: Intersection with centre}
Let $\Lambda$ be an approximate lattice in an $S$-adic linear group $G$ and suppose that $\langle \Lambda \rangle$ is Zariski-dense. Then $Z_G \cap \Lambda^2$ is a uniform approximate lattice in $Z_G$. 
\end{lemma}

A similar proof also establishes the following intersection result which is particularly relevant to Theorem \ref{THEOREM: MAIN THEOREM}.

\begin{corollary}\label{Lemma: Intersection with centraliser of co-rank 1 subgroups}
With $\Lambda \subset G$ as above and $Z_{G,1}$ as in Theorem \ref{THEOREM: MAIN THEOREM}. Then $\Lambda^2 \cap Z_{G,1}$ is a uniform approximate lattice in $Z_{G,1}$.
\end{corollary}

The same techniques lead to other intersection theorems in $S$-adic groups.  We mention here intersections with certain isotypic components in $S$-adic linear groups with abelian radical and intersections of approximate lattices with Levi subgroups, see \cite{mac2023sadic}. These technical results  play a key role.

\section{Cohomology  and uniqueness of quasi-models}\label{Section: Cohomology for approximate subgroups}

The purpose of this section is to prove uniqueness and universality statements for the quasi-models introduced in \cite{hrushovski2020beyond}.  The most general one being:

  \begin{theorem}[Quasi-models are uniquely determined]\label{Theorem: Uniqueness of quasi-models without cohomology}
 Let $\Lambda$ be an approximate subgroup contained in a group $\Gamma$ that commensurates it. Suppose that we are given $f_i: \Gamma \rightarrow H_i$ with defect $K_i$ generating a Euclidean space $A_i$ as in the second part of Theorem \ref{Theorem: Hrushovski's quasi-models} for $i=1,2$. Then there are a quasi-homomorphism $\phi: H_1 \rightarrow H_2$ with defect contained in $A_2$ and bounded maps $k_i: \Gamma \rightarrow A_i$ for $i=1,2$ such that $$\phi \circ (f_1k_1) = f_2 k_2.$$
 Here, $f_i k_i$ is defined from $f_i$ and $k_i$ by pointwise multiplication.  
 \end{theorem}
 
 A similar universality result was obtained independently by Krupiński and Pillay in \cite{KrupinskiPillay}. Their method is completely different and is the consequence of a beautiful new proof of the main result of \cite{hrushovski2020beyond}. This enables them to accommodate for potential additional structure. We indicate however that they do not prove a local version nor the cohomological interpretation that we exploit below (Corollary \ref{Corollary: Uniqueness of quasi-models, local version}). 
 
 Our method of proof follows the philosophy of the so-called Goursat lemma.  We obtain below an approximate subgroup that projects (essentially) bijectively to both factors - forcing it to be (up to an error) the graph of the quasi-homomorphism denoted $\phi$ above.  For various reasons however, this is achieved in multiple steps and exploits a novel notion of cohomology.
 
 \subsection{Cohomology for approximate subgroups}
The cohomology defined below requires a notion of boundedness provided by bornologies. A \emph{bornology} $\mathcal{B}$ of a set $X$ is a subset of  the set $\mathcal{P}(X)$ of subsets of $X$ that is stable under inclusions, finite unions and covers the whole of $X$. Given a bornology $\mathcal{B}$ and a subset $Y \subset X$, we say that $Y$ is \emph{bounded} if $Y \in \mathcal{B}$. Given any metric space $(X,d)$, a natural bornology is provided by the set of subsets of finite diameter. Given an approximate subgroup $\Lambda$ of a group $\Gamma$ that commensurates it, the bornology $\mathcal{B}_\Lambda$ made of all the subsets of $\Gamma$ that are covered by finitely many translates of $\Lambda$ is a bornology of $\Gamma$. It is moreover stable under the action of $\Gamma$ by left \emph{and} right multiplication.

Given an approximate subgroup $\Lambda$ of some group $\Gamma$ that commensurates it, an abelian subgroup $V$ equipped with a bornology $\mathcal{B}$ and an action of $\Gamma$ leaving $\mathcal{B}$ invariant (e.g. $\mathcal{B}_d$ for a left-invariant distance, or $\mathcal{B}_{\Lambda'}$ for $\Lambda'$ an approximate subgroup commensurated by $\Lambda$), we consider the sets of \emph{bounded} cochains 
$$C^n_b(\Lambda, V):= \{c: \Gamma^n \rightarrow V : \forall X \in \mathcal{B}_{\Lambda}, c(X, \ldots, X) \text{ bounded}\}$$
with the usual coboundary homomorphism \begin{align*}
\left(d^{n+1}c \right)(g_{1},\ldots ,g_{n+1})=g_{1}c (g_{2},&\dots ,g_{n+1})\\
&+ \sum _{i=1}^{n}(-1)^{i}c \left(g_{1},\ldots ,g_{i-1},g_{i}g_{i+1},\ldots ,g_{n+1}\right) \\
&\  \  \  \  \  \ \ \ \ \ \ \ \ +(-1)^{n+1}c (g_{1},\ldots ,g_{n}).
\end{align*}
We denote by $Z_b^n(\Lambda; V)$ the kernel of $d^{n+1}$, $B_b^n(\Lambda;V)$ the image of $d^n$ and $H_b^n(\Lambda;V)$ the resulting cohomology.

 Let us explain how Theorem \ref{Theorem: Uniqueness of quasi-models without cohomology} relates to cohomology following \cite[\S 5.24]{hrushovski2020beyond}. Given a quasi-model $f:\Gamma \rightarrow H$ as in the second part of Theorem \ref{Theorem: Hrushovski's quasi-models} define $\alpha$ for all $\gamma_1,\gamma_2 \in \Gamma$ by $$\alpha(\gamma_1,\gamma_2):= f(\gamma_1\gamma_2)f(\gamma_2)^{-1}f(\gamma_1)^{-1}.$$ Then  $\alpha \in C^2_b(\Gamma;A)$ (note here the use of $\Gamma$ instead of $\Lambda$). The class $\overline{\alpha}\in H^2_b(\Lambda;A)$ of $\alpha $ can also be recovered in the following way. Let $L$ denote $H/A$ and say a map between locally compact spaces is \emph{locally bounded} if it sends relatively compact subsets to relatively compact subsets. By choosing a locally bounded Borel section $s$ of the natural projection $H \rightarrow L$, we can identify $H$ with $H=:L \ltimes_{\beta} A$ where $\beta(l_1, l_2):= s(l_1l_2)s(l_2)^{-1}s(l_1)^{-1}$ and $L \ltimes_{\beta} A$ denotes the set $L \times A$ equipped with the multiplication
 $$(l_1, a_1)\cdot (l_2, a_2) \mapsto (l_1l_2, a_1 + l_1\cdot a_2 + \beta(l_1,l_2)).$$
 Write $\bar{f}$ the composition of $f$ and the natural projection $H \rightarrow L$.  We have that $\alpha$ and $\bar{f}^*\beta$ have the same class in $H^2_b(\Lambda;A)$. In what follows, when we write that some locally compact group is equal to a group of the form $L \ltimes_{\beta} A$ we implicitly assume that, as above, this is done thanks to a locally bounded section. In particular, a subset of $H \ltimes_{\beta} A$ is relatively compact if and only if its projections to $H$ and $A$ are relatively compact.

  We can rephrase Theorem \ref{Theorem: Uniqueness of quasi-models without cohomology} in that language as well. 
 
 \begin{theorem}[Quasi-models are uniquely determined, technical version]\label{Theorem: Uniqueness of quasi-models}
 Let $\Lambda$ be an approximate subgroup contained in a group $\Gamma$ that commensurates it. Suppose that we are given quasi-models $f_i: \Gamma \rightarrow H_i$ with defect generating a Euclidean space $A_i$ as in the second part of Theorem \ref{Theorem: Hrushovski's quasi-models} for $i=1,2$. If $\alpha_i$ is the $2$-cocycle defined by $\alpha_i(\gamma_1, \gamma_2):= f_i(\gamma_1\gamma_2)f_i(\gamma_2)^{-1}f_i(\gamma_1)^{-1}$, then there are: 
 \begin{enumerate}
 \item normal vector subspaces $A_i' \subset A_i$ normal in $H_i$;
 \item a linear isomorphism $\phi:A_1' \rightarrow A_2'$ equivariant for the respective $\Gamma$-actions on $A_1'$ and $A_2'$ that arise from $f_1$  and $f_2$ respectively;
 \item $2$-cocycles $\alpha_i' \in C^2_b(\Gamma; A_i')$;
 \item approximate subgroups $\Lambda_i^*:=f^{-1}(W_iA_i)$ for some neighbourhood of the identity $W_i \subset H_i$;
 \end{enumerate}
 such that $\alpha_i'$ and $\alpha_i$ have the same class in $H^2_b(\Lambda_i^*;A_i)$ and $\alpha_2$ and $\phi(\alpha_1)$ have the same class in $H^2_b(\Lambda_1^* \cap \Lambda_2^*; A_2')$. 
 \end{theorem}

It is often convenient to pass to commensurable approximate subgroups supported on finite index subgroups. Laminarity of approximate subgroups remains unchanged under this operation. So does cohomology. 
 
 \begin{lemma}\label{Lemma: Invariance cohomology finite index subgroups}
 Let $\Lambda \subset \Gamma$ be an approximate subgroup of a group it is commensurated by. Let $\Gamma_0 \subset \Gamma$ be a finite index subgroup, $\Lambda_0 :=\Lambda^2 
 \cap \Gamma_0$ and $A$ a divisible abelian group equipped with a bornology and an action of $\Gamma$ by locally bounded group automorphisms. Then the restriction map $H^*_b(\Lambda; A) \rightarrow H^*_b(\Lambda_0;A)$ is injective. 
 \end{lemma}
 
 The proof can be adapted from the proof of \cite[\S 8.6]{MR1840942} with the help of a transfer operator. Note here that all the groups considered are discrete, which simplifies proofs considerably. Lemma \ref{Lemma: Invariance cohomology finite index subgroups} and the second part of Theorem \ref{Theorem: Hrushovski's quasi-models} thus enable us to pass from $\Gamma$ to $\Gamma_0$ without losing much information. We therefore make the following assumption:
 
 \begin{assumption}\label{Assumption: Quasi-model}
 We consider a quasi-model $f: \Gamma \rightarrow H$ whose defect is contained in a Euclidean ball $K$ of a Euclidean space $A$ contained and normal in $H$, on which $H$ acts by isometries. We moreover assume that the map $\bar{f}$ defined as the composition of $f$ and the natural projection $H \rightarrow H/A$ is such that $\bar{f}(\Lambda)$ is dense in a subset with non-empty interior. 
 \end{assumption}

\subsection{Key proposition}
 
The following is the crux of the proof of \ref{Theorem: Uniqueness of quasi-models}: it gives the desired conclusion under slightly restrictive assumptions on the quasi-models considered. 

\begin{proposition}\label{Proposition: Crux uniqueness quasi-models}
Let $\Lambda$ and $\Gamma$ be as above. Let $L$ be a locally compact group, $A_1,A_2$ be Euclidean spaces with $L$-action by isometries, $\beta_1, \beta_2$ two locally bounded Borel $2$-cocycles and $f: \Gamma \rightarrow L \ltimes_{\beta_1} A_1$, $g:\Gamma \rightarrow L \ltimes_{\beta_2} A_2$ be two quasi-homomorphisms with relatively dense range, $f(\Lambda)$ relatively compact and $g$ a quasi-model satisfying Assumption \ref{Assumption: Quasi-model}. Suppose moreover that the defect of $f$ and $g$ are contained in $A_1$ and $A_2$ respectively and that the group homomorphisms $\Gamma \rightarrow L$ induced by $f$ and $g$ are equal. Then there is an $L$-equivariant surjective linear map $\psi: A_2 \rightarrow A_1$ and $b \in C^1_b(\Lambda; A_1)$ such that 
$$ \alpha_1 = db + \psi(\alpha_2)$$
where $\alpha_1$ and $\alpha_2$ are the bounded $2$-cocycles of $\Gamma$ defined by $f(\gamma_1\gamma_2)f(\gamma_2)^{-1}f(\gamma_1)^{-1}$ and $g(\gamma_1\gamma_2)g(\gamma_2)^{-1}g(\gamma_1)^{-1}$ respectively. 

Moreover, $b$ is the sum of a map $b_1: \Gamma \rightarrow A_2$ factoring through a locally bounded map $b_1': L \rightarrow A_2$ and a uniformly bounded map $b_2: \Gamma \rightarrow A_2$.
\end{proposition}

\begin{proof}
The proof of Proposition \ref{Proposition: Crux uniqueness quasi-models} is based on a Goursat lemma of sorts. Let $K$ denote a Euclidean ball in $A_1 \times A_2$. Write $H_f:= L \ltimes_{\beta_1} A_1$ and $H_g:= L \ltimes_{\beta_2} A_2$, and $p_g: H_f \times H_g \rightarrow H_g$ and $p_f: H_f \times H_g \rightarrow H_f$ the natural projections. Define the subset $$\Xi:= \overline{\{ (f(\gamma),g(\gamma)) \in H_f \times H_g : \gamma \in \Gamma\}}K.$$
As in Lemma \ref{Lemma: Graphs of quasi-homomorphisms, continuous case}, we have that $\Xi$ is an approximate subgroup. The projection of $\Xi$ to  $H_f$ contains $f(\Gamma)$ and the projection to $H_g$ contains $g(\Gamma)$, so both are relatively dense. For any neighbourhood of the identity $W \subset H_f \times H_g$, the intersection $\Xi^2 \cap W^2H_f$ is contained in 
$$\overline{\{ (f(\gamma),g(\gamma)) : g(\gamma) \in W^2K^3\}K^3}.$$
But $g^{-1}(W^2K^3)$ is covered by finitely many translates of $\Lambda$ since $g$ is a quasi-model. Hence, $f\left(g^{-1}(W^2K^3)\right)$ is relatively compact. In turn, $\Xi^2 \cap W^2H_f$ is relatively compact. Finally, the projection of $\Xi$ to $L \times L$ is contained in the diagonal $\Delta_L:=\{(l,l) \in L \times L : l \in L\}$ by assumption. It is in fact equal to it. Indeed, it contains  $$\{(\bar{f}(\lambda), \bar{g}(\lambda)): \lambda \in \Lambda\} = \{(\bar{g}(\lambda), \bar{g}(\lambda)): \lambda \in \Lambda\}$$
by our assumptions, and the latter is dense in a subset of $\Delta_L$ with non-empty interior by  Assumption \ref{Assumption: Quasi-model}. But $\overline{\{(f(\lambda), g(\lambda)): \lambda \in \Lambda\}}$ is a compact subset of $\Xi$, so the projection of $\Xi$ to $L \times L$ contains the subset with non-empty interior $\overline{\{(\bar{g}(\lambda), \bar{g}(\lambda)): \lambda \in \Lambda\}}$. But, by Assumption \ref{Assumption: Quasi-model} again, the projection of $\Xi$ is dense in $\Delta_L$, so we have proved our claim.

 Choose $W$ a symmetric relatively compact neighbourhood of the identity in $H_f \times H_g$. Define $\Xi_W:= \Xi^2 \cap W^2(A_1 \times A_2)$. According to Lemma \ref{Lemma: Intersection of approximate subgroups}, $\Xi_W$ is an approximate subgroup commensurated by $\Xi$ and by the version of Schreiber's lemma found in \cite[Lem. 2.9]{mac2023definitions} there is a vector subspace $V \subset A_1 \times A_2$ normalised by $\Xi$ and a compact subset $C_1$ normalising $V$ such that $\Xi_W \subset C_1V$ and $V \subset C_1\Xi_W$. Since the conjugation action of $H_f \times H_g$ on $A_1 \times A_2$ factorises through $L \times L$, $V$ is stable under the action of the image of $\Xi$ in $L \times L$ i.e. $V$ is stable under the action by the diagonal $\Delta_L$. Furthermore, we claim that $V$ projects surjectively to both $A_1$ and $A_2$. Indeed, note that $p_g(\Xi)$ and $p_f(\Xi)$ are relatively dense in $H_g$ and $H_f$ respectively and choose any symmetric relatively compact neighbourhood $W'$ of the identity in $\Delta_L$ such that the projection of $\Xi_W$ to $L \times L$ is contained in $W'$. Let $\Xi_W'$ denote all the elements of $\Xi^2$ that project into $W'$ and notice that $\Xi_W'$ is commensurable with $\Xi_W$ (Lemma \ref{Lemma: Intersection of commensurable sets}).  Take $\xi \in \Xi$ and suppose that $p_g(\xi)$ is contained in $p_g(W) A_2$. Then $\xi \in \Xi_W'$ because the projection of $\Xi$ to $L\times L$ is contained in the diagonal. Since $p_g(\Xi)$ is relatively dense in $H_g$, we find that there is a compact subset $C_2 \subset H_g$ such that $A_2 \subset C_2p_g(\Xi_W')$. Hence, $A_2 \subset C_2p_g(C_1) p_g(V)$ i.e $A_2 = p_g(V)$. Symmetrically, $A_1=p_f(V)$.  Now, $V \cap A_2$ is contained in $C \Xi^2 \cap H_g$ which is compact, so $V \cap A_2$ is trivial. By Goursat's lemma, $V$ is the graph of a surjective linear map $\psi:A_2 \rightarrow A_1$. But $V$ is normalised by the diagonal subgroup in $L \times L$. So $\psi$ is equivariant with respect to the $L$-actions on $V_1$ and $V_2$. Finally, since $\Xi$ contains a compact subset $K'$ that projects surjectively to a neighbourhood of the identity of $\Delta_L$, $\Xi^2$ contains $K'K$ which is a neighbourhood of the identity of the inverse image of $\Delta_L$ in $H_f \times H_g$. We deduce that $V \subset \Xi^m$ for $m$ sufficiently large. 
 
 Again, because $\Xi$ projects surjectively to $\Delta_L$ and contains $K'$, we can choose in $\Xi^2$ a closed subset of the form
 $$ \Delta_1:=\{ ((l,\tau(l)),(l,\tau'(l))) : l \in L\} $$
 with $\tau$ and $\tau'$ locally bounded.
 Now, 
 $$ \Delta_2:= V\Delta_1 = \{ ((l,\tau(l) - \psi(\tau'(l)) + \psi(v)),(l,v)) : l \in L, v \in V_2\}$$
 is contained in $\Xi^{m+2} $. Conversely, notice that $\Xi \subset \Delta_1(A_1 \times A_2)$.  Since $\Delta_1 \subset \Xi^2$ this means $\Xi \subset \Delta_1 \left(\Xi^3 \cap (A_1 \times A_2)\right)$.  Similarly,  $\Xi^3 \cap \left(A_1 \times A_2\right) \subset VC_1$ implies $\Xi^3 \cap \left(A_1 \times A_2\right) \subset VC_3$ where $C_3:=C_1 \cap \left(A_1 \times A_2\right)$.  So $\Xi \subset \Delta_2C_3 $.  Finally, notice that $\Delta_2C_3 \subset \Delta_2 C_4$ where $C_4:= p_f(C_3) \psi \circ p_g(C_3) \subset A_1$ because $V$ is the graph of $\psi$.
 
 Write $a(l) := \tau(l) - \psi(\tau'(l))$ for all $l \in L$ and define a map $\phi: H_g \rightarrow H_f$ by $\phi(l,v) := (l,a(l) + \psi(v))$.  In other words, $\phi$ is the map whose graph is $\Delta_2$. But $\Xi \subset \Delta_2 C_4$, so for all $\gamma \in \Gamma$ we have $(f(\gamma), g(\gamma)) \in \Delta_2C_4$. That is to say, $(f(\gamma), g(\gamma)) = (\phi(g(\gamma)) b(\gamma), g(\gamma))$ for some $b(\gamma) \in C_4$. In other words, $ f(\gamma)^{-1}\phi(g(\gamma))$ is bounded uniformly in $\gamma$ and contained in $A_1$. Therefore, 
 \begin{align*}
 \alpha_1(\gamma_1, \gamma_2) & =   \psi(\alpha_2(\gamma_1,\gamma_2)) + db(\gamma_1, \gamma_2) + da(\bar{g}(\gamma_1),\bar{g}(\gamma_1)).
 \end{align*}

\end{proof}

We will mainly use the uniqueness statement (Theorem \ref{Theorem: Uniqueness of quasi-models}) to characterise laminarity. An analysis of the proof of Proposition \ref{Proposition: Crux uniqueness quasi-models}  reveals that we do not use that the $L$-action is by isometries. This provides the following corollary of a local nature that will be enough for our purpose:

\begin{corollary}[Characterising laminarity via cocycles]\label{Corollary: Uniqueness of quasi-models, local version}
Let $\Lambda$ be an approximate subgroup of a group $\Gamma$ that commensurates it. Suppose that $\Lambda$ is laminar. Let $H$ be a locally compact group of the form $L \ltimes_{\beta} A$ with $A \simeq \mathbb{R}^n$ equipped with a continuous $L$-action and $\beta$ a locally bounded Borel cocycle. Let $f: \Gamma \rightarrow H$ be a map such that:
\begin{enumerate}
\item the map $\overline{f}$, obtained by the composition of $f$ and the natural projection $H \rightarrow L$, is a group homomorphism;
\item $f(\Lambda)$ is relatively compact;
\item $f^{-1}(C)$ is covered by finitely many translates of $\Lambda$ for all $C \subset H$ compact;
\item Set $\Lambda^*:= \overline{f}^{-1}(W)$ for some compact neighbourhood of the identity $W$ containing  $\overline{f}(\Lambda)$, then the defect $\{f(\gamma_1\gamma_2)f(\gamma_2)^{-1}f(\gamma_1)^{-1} : \gamma_1,\gamma_2 \in \Lambda^*\}$ of $f$ \emph{restricted to} $\Lambda^*$ is contained in $K \subset A$ compact.
\end{enumerate}
 Then there is $b \in  C^1_b(\Lambda^*; A)$ such that for all $\gamma_1, \gamma_2 \in \Lambda^*$ we have $$f(\gamma_1\gamma_2)f(\gamma_2)^{-1}f(\gamma_1)^{-1} = db(\gamma_1, \gamma_2).$$ 
\end{corollary}

\subsection{Proof of uniqueness}
We show now Theorems \ref{Theorem: Uniqueness of quasi-models without cohomology} and \ref{Theorem: Uniqueness of quasi-models}.  The drawback of Proposition \ref{Proposition: Crux uniqueness quasi-models} is that we start with strong assumptions on the targets of $f$ and $g$. Its use is facilitated by: 

\begin{lemma}\label{lemma: Reduction uniqueness of quasi-models}
Let $\Lambda$ be an approximate subgroup of a group $\Gamma$ that commensurates it. Let $L$ be a locally compact group, $A$ be a vector space with an $L$-action by isometries, $\beta$ a locally bounded Borel $2$-cocycles in $Z^2(L;A)$ and $f: \Gamma \rightarrow L \ltimes_{\beta} A$ a quasi-model satisfying Assumption \ref{Assumption: Quasi-model} with defect contained in $A$. Let also $f_0: \Gamma \rightarrow H_0$ be the Bohr compactification of $\Lambda$ (Lemma \ref{Lemma: Bohr compactification and abstract automorphisms}). Then there is a rigid normal subgroup $A' \subset H$ such that $H/A' \simeq L$ and the composition of $f_0$ with the isomorphism $H/A' \simeq L$ is $\bar{f}$.
\end{lemma}

\begin{proof}
Write $H_f$ for $L \ltimes_{\beta} A$. For any symmetric relatively compact open neighbourhood of the identity $W \subset H_f$ sufficiently large, $ \Lambda^*=f^{-1}(W^2A)$ is an approximate subgroup containing $\Lambda$ that admits $\bar{f}: \Gamma \rightarrow L$ as a good model. Besides, since $\bar{f}(\Lambda)$ is relatively compact, there is a surjective continuous group homomorphism $\phi: H_0 \rightarrow L$ such that $\bar{f} = \phi \circ f_0$. Identify $L$ and $H_0/\ker \phi$. Write $A':=\ker \phi$ and notice that $A' \subset \overline{f_0(\Lambda^*)}$.

First of all, we show that the subgroup $A'$ is compactly generated. Since $A$ is a Euclidean space, it is compactly generated. The approximate subgroup $W^2A \cap (f(\Gamma)K)^2 \subset W^2A \cap f(\Gamma)K^2$ is relatively dense in $W^2A$ because $f(\Gamma)K$ is relatively dense in $H_f$. So there is a compact subset $C \subset H$ such that for every $\lambda \in \Lambda^*$ there are $\lambda_0 = e, \ldots, \lambda_n=\lambda \in f^{-1}(W^2A) = \Lambda^*$ satisfying $f(\lambda_{i+1})^{-1}f(\lambda_i) \in C$ for all $i \in \{0, \ldots,n-1\}$. Since $f$ is a quasi-model, there is therefore $X \subset \Gamma$ commensurable with $\Lambda$ such that $\lambda_{i+1}^{-1}\lambda_i \in X$ for all $i \in \{0, \ldots,n-1\}$ and $X$ is independent of the choice of $\lambda$. Choose a compact neighbourhood of the identity $W_0 \subset H_0$ such that $f_0(\Lambda^*) \subset W_0A'$. For any $a \in A'$ choose $\lambda \in \Lambda^*$ such that $f_0(\lambda)^{-1}a \in W_0$. Let $\lambda_0=e, \ldots, \lambda_n \in \Lambda^*$ be as above with respect to $\lambda$ and write $C':= \overline{f_0(X)}$ which is compact because $X$ is covered by finitely many translates of $\Lambda$. For all $i \in \{0, \ldots,n\}$ there is $w_i \in W_0$ and $a_i \in A'$ such that $f_0(\lambda_i)w_i^{-1} = a_i$. Then for all $i \in \{0, \ldots, n-1\}$, $$a_{i+1}^{-1}a_i \in \left(W_0C'^{-1}C'W_0^{-1}\right) \cap A'.$$ So $A'$ is generated by the compact subset $\left(W_0C'^{-1}C'W_0^{-1}\right) \cap A'$.

Now, we show that $A'$ is rigid. Write $C$ a compact generating subset of $A'$ and choose a neighbourhood basis $(W_n)_{n \geq 0}$ at $e$ in $L$ made of compact subsets. Since $f_0(\Lambda)$ is dense in a subset with non-empty interior and $f_0(\Gamma)$ is dense, there is $X \subset \Gamma$ commensurable with $\Lambda$ such that $C \subset \overline{f_0(X)}$. Write $B:=\overline{f(X)}$. Then $B$ is a compact subset of $H_f$. For all $a \in C$ there is a sequence $(x_n)_{n \geq 0}$ of elements of $X$ such that $f_0(x_n) \rightarrow a$ as $n$ goes to $\infty$.  As $\bar{f}(x_n) \rightarrow e$ and upon considering a subsequence, we may moreover take $(x_n)_{n \geq 0}$ such that $\bar{f}(x_n) \in W_n$ for all $n \geq 0$. Since $f(X)$ is relatively compact, we can take a further subsequence and assume that $(f(x_n))_{n \geq 0}$ admits a limit too. But $\bar{f}(x_n) \in W_n$ for all $n \geq 0$, so the limit belongs to $A$. For all $\gamma \in \Gamma$ the same observation yields more generally that the sequence $\left(f(\gamma) f(x_n) f(\gamma)^{-1}\right)_{n \geq 0}$ goes to an element of $f(\gamma) B f(\gamma)^{-1} \cap A$ as $n$ goes to $\infty$. But $B':= \left(\bigcup_{\gamma \in \Gamma} f(\gamma) B f(\gamma)^{-1}\right) \cap A$ is relatively compact since $A$ is rigid. Take any relatively compact neighbourhood of the identity $W \subset H_f$. Then $X':=f^{-1}(WB')$ is covered by finitely many translates of $\Lambda$ and $\gamma x_n \gamma^{-1} \in X'$ for $n$ large enough. In particular, if $C''$ denotes the compact subset $\overline{f_0(X')}$, we have $\bigcup_{\gamma \in \Gamma} f_0(\gamma)Cf_0(\gamma)^{-1} \subset C''$. Hence, $A'$ is rigid. 
\end{proof}

\begin{proof}[Proof of Theorem \ref{Theorem: Uniqueness of quasi-models}.]
Let $f_0: \Gamma \rightarrow H_0$ be the Bohr compactification as in Lemma \ref{lemma: Reduction uniqueness of quasi-models}. Let $\bar{f_1}$ denote the composition of $f_1$ and the natural projection $H \rightarrow L$. Then there is a continuous group homomorphism $\phi: H_0 \rightarrow L$ such that $\bar{f_1} = \phi \circ f_0$. Let $N$ denote the kernel of $\phi$ which is rigid by Lemma \ref{lemma: Reduction uniqueness of quasi-models}. Upon passing to a finite index subgroup $\Gamma$ and quotienting out a compact subgroup of $N$, we have that $N \simeq \mathbb{R}^n \times \mathbb{Z}^m$. By considering the supergroup $\left(H_0 \times \mathbb{R}^m\right)/\mathbb{Z}^m$ instead of $H_0$ - where $\mathbb{Z}^m$ is embedded diagonally  -  we may assume that $N$ is connected. So there is a locally bounded Borel cocycle $\beta$ such that $H_0 = L \ltimes_{\beta} N$. By Proposition \ref{Proposition: Crux uniqueness quasi-models}, we find an $L$-equivariant linear map $\psi: A_1 \rightarrow N$ such that
$$\psi(\alpha_1) = db$$
with $b=b_1 + b_2$ as in Proposition \ref{Proposition: Crux uniqueness quasi-models} and we used $f_0(\gamma_1\gamma_2)f_0(\gamma_2)^{-1}f_0(\gamma_1)^{-1}=e$.
By the surjectivity of $\psi$ we find $b= \psi(v)$ for some choice of $v \in C_b^1(\Lambda;A_1)$ taking values in the orthogonal of $A_1'=\ker \psi$ which also admits a decomposition $v=v_1 + v_2$ as in Proposition \ref{Proposition: Crux uniqueness quasi-models}.

Let $p_{\perp}$ denote the orthogonal projection to $\ker \psi$. We have $p_{\perp}(\alpha_1) - \alpha_1 = dv$.  Define $\tilde{f}:  \Gamma \rightarrow L \ltimes_{p_{\perp}(\beta)} A_1'$ by $\tilde{f}(\gamma) = (\tau(\gamma), p_{\perp}(a(\gamma)))$ where $a(\gamma)$ is the $A_1$ component of $f(\gamma)$. Notice that $L \ltimes_{p_{\perp}(\beta)} A_1'$ is simply the quotient by the orthogonal of $A_1'$ and thus has a well-defined structure as a locally compact group. Notice also that $\tilde{f}$ is a quasi-homomorphism because $p_{\perp}(\alpha_1) $ is bounded. Moreover, $f(\Lambda)$ is relatively compact as both $\tau(\Lambda)$ and $a(\Lambda)$ are. Finally, given any neighbourhood of the identity $W \in L \ltimes_{p_{\perp}(\beta)} A_1'$, $$\tilde{f}^{-1}(W) \subset \tau^{-1}(W) \cap p_{\perp}(a) ^{-1}(W).$$ 
But if we  note that
$$(\bar{f}(\lambda), v(\lambda) + p_{\perp}(a(\lambda))) =f(\lambda)$$ 
and 
$$(\bar{f}(\lambda), p_{\perp}(a(\lambda))) =\tilde{f}(\lambda)$$ 
we find that $\bar{f}(\lambda)$ is uniformly bounded as $\lambda$ ranges through $\tilde{f}^{-1}(W)$ and, hence, so is $v(\lambda)$ (recall that $v$ satisfies a decomposition as in the `moreover' part of Proposition \ref{Proposition: Crux uniqueness quasi-models}). Since $p_{\perp}(a(\lambda))$ is also uniformly bounded  as $\lambda$ ranges through $\tilde{f}^{-1}(W)$, we find that $f(\lambda)$ is uniformly bounded  as $\lambda$ ranges through $\tilde{f}^{-1}(W)$. Thus, $f( \tilde{f}^{-1}(W))$ is relatively compact. So $\tilde{f}$ is also a quasi-model satisfying Assumption \ref{Assumption: Quasi-model} and $p_{\perp}(\alpha_1) = \alpha_1 - dv$. 

By applying this procedure to $f_2$ as well, the discussion tells us that it is enough to prove Theorem \ref{Theorem: Uniqueness of quasi-models} in the situation where $H_1=L \ltimes_{p_{\perp}(\beta_1)} A_1'$, $H_2= L \ltimes_{p_{\perp}(\beta_2)} A_2'$ and $\tilde{f}_1$ and $\tilde{f}_2$ satisfy Assumption \ref{Assumption: Quasi-model}. But, then, Theorem \ref{Theorem: Uniqueness of quasi-models} is simply Proposition \ref{Proposition: Crux uniqueness quasi-models}. So the proof is complete. 

\end{proof}

Finally, we can prove Theorem \ref{Theorem: Uniqueness of quasi-models without cohomology}.

\begin{proof}[Proof of the of Theorem \ref{Theorem: Uniqueness of quasi-models without cohomology}.]
Let $\alpha_1',\alpha_2', A_1',A_2', \phi$ be given by Theorem \ref{Theorem: Uniqueness of quasi-models}.  Choose $b_1 \in C^1_b(\Lambda_1^*;A_1)$ such that $\alpha_1 - \alpha_1' = db$. In particular, $db$ is uniformly bounded. This implies that $$\sup_{\gamma \in \Gamma, \lambda \in \Lambda^*} || b(\gamma \lambda^*) - b(\gamma)|| < \infty$$
where $|| \cdot ||$ denotes a Euclidean norm on $A_1$. Thus, we can find a decomposition $b = b_1 + b_1'$ with $b_1'$ uniformly bounded and $b_1$ factorises through $\bar{f_1}:\Gamma \rightarrow H_1/A_1$ i.e. $b_1=b_1'' \circ \bar{f_1}$ for some locally bounded map $b_1''$.  Define $\phi_1: H_1 \rightarrow H_1$ by $\phi_1(h) = hb_1''(p_1(h))^{-1}.$ The map $\phi_1$ is a quasi-homomorphism.  Build similarly $b_2, b_2'$ and $b_2''$ as above with respect to $f_2$.  Define $\phi_2: H_2 \rightarrow H_2$ by $\phi_2(h) = hb_2''(p_2(h))$.  Finally, choose $z \in C^1_b(\Lambda_1^* \cap \Lambda_2^*;A_2')$ such that $\phi(\alpha_1') - \alpha_2' = dz$. As above, we can decompose it as $z=z_1 + z_2$ with $z_2$ uniformly bounded and $z_1 = z_1' \circ \bar{f_2'}$. Define $\psi: H_1 \rightarrow H_2$ by $\psi(h)=h(z_1'(h))^{-1}$. Then the quasi-homomorphism we are looking for is $\phi_2\circ \psi \circ \phi_1$ and $k_1(\gamma):= b_1'(\gamma)^{-1}$ and $k_2(\gamma) := b_2'(\gamma)^{-1}z_2(\gamma)^{-1}$. 
\end{proof}

\section{Lubotzky--Mozes--Raghunathan-type theorem}\label{Section: Distortion of approximate lattices in semi-simple groups}

The goal of this section is to prove Theorem \ref{THEOREM: EFFICIENT MESSY GENERATOIN}. Let $\mathbf{G} \subset \GL_n$ defined over a number field $K$. From each embedding $\mathbf{G}(K_v) \subset \GL_n(K_v)$ arises a norm $||\cdot||_v$ defined by  $$||m||_v = \sup_{1\leq i,j \leq n} |m_{ij}|_v$$ for $m=(m_{ij})_{1\leq i,j \leq n} \in \GL_n(K_v)$. For $m \in \GL_n(K_v)$ there is another quantity that is sometimes easier to handle:
$$||m||_v':= ||m-\id||_v+1.$$
We denote by $||\cdot||_S'$ (resp. $||\cdot||_S$) the product $\prod_{v \in S}|| \cdot ||_v'$ (resp. $\prod_{v \in S}|| \cdot ||_v$). Then $|| \cdot ||_S'$ is sub-multiplicative. The main inspiration for this section is the result by Lubotzky, Mozes and Raghunathan concerning the comparison between the Riemannian and word distances on non-uniform higher-rank lattices. 

\begin{theorem*}[\cite{MR1828742}]
Let $\Gamma$ be an irreducible lattice of a semi-simple $S$-adic linear group $G$ of rank at least $2$. Then the word metric on $\Gamma$ is equivalent to the restriction of any left-invariant Riemannian metric of $G$ to $\Gamma$. 
\end{theorem*}

Here, the Riemannian metric simply denotes the distance $d_S(g_1, g_2):=\log ||g_2^{-1}g_1||_S'$. One side of the comparison is in fact trivial. Indeed, if $\gamma \in \Gamma$ can be written as a word of size $l$ in the letters coming from $F \subset \Gamma$, then $d_S(e,\gamma) \leq l \sup_{f \in F} \log ||f||_S'$ by sub-multiplicativity of $|| \cdot ||_S'$.

Our aim is to provide generalisations of this theorem to approximate lattices. We mention here that Hartnick and Witzel offer in \cite{https://doi.org/10.48550/arxiv.2204.01496} a look at related geometric problems for approximate lattices in totally disconnected $S$-adic semi-simple groups. Before we move on to the rest of this section, let us define the property we will establish:

\begin{definition}\label{Definition: Efficient generation}
Let $\Lambda$ be a discrete approximate subgroup in an $S$-adic linear group $G$. We say that $\Lambda$ is \emph{efficiently generated (by $F$)} if there are a finite subset $F \subset \langle \Lambda \rangle$ and a constant $C >0$ such that for all $\lambda \in \Lambda$ there are $\lambda_0=e, \ldots, \lambda_m=\lambda \in \Lambda$ with $m \leq C \log || \lambda ||'_S$ satisfying: 
 $$\forall i \in \{1; \ldots;m-1\}, \exists j \in \{1; \ldots;i\},\  \lambda_{i+1}\lambda_j^{-1} \in F\left(\lambda_i\lambda_j^{-1}\right)F^{-1}.$$
\end{definition}

If $\Lambda$ is a group this is equivalent to $\Lambda$ being undistorted.  We prove: 

\begin{theorem}[Efficient messy generation]\label{THEOREM: EFFICIENT MESSY GENERATOIN}
Let $\Lambda$ be an irreducible approximate lattice in an $S$-adic algebraic group $G$. Suppose that $\langle\Lambda\rangle$ is dense or that $G$ has $S$-rank at least $2$.  Then $\Lambda$ is efficiently generated.
\end{theorem}

Given the arithmeticity theorem \cite[Thm 7.4]{hrushovski2020beyond}, it is enough to work with \emph{arithmetic} approximate lattices. Moreover, since higher-rank lattices are undistorted by \cite{MR1828742} thus \emph{a fortiori} efficiently generated, we may assume that $\Lambda$ is not contained in a lattice. In the remaining parts of this section, we therefore make the following assumption. 

\begin{assumption}\label{Assumption: Arithmeticity}
There are a number field $K$, a finite set of places $S$ of $K$ and a simply connected absolutely simple $K$-group $\mathbf{G} \subset \GL_n$ such that $ G = \prod_{v \in S} \mathbf{G}(K_v)$  and $\Lambda$ is commensurable with $\mathbf{G}(\mathcal{O}_{K,S})$ embedded diagonally in $G$.  Also, $\langle \Lambda \rangle$ is commensurable with $\mathbf{G}(\mathcal{O}_{K,\hat{S}})$ with $S \subset \hat{S}$ and $|\hat{S}|\geq 2$.
\end{assumption}

\subsection{Efficient generation and bounds on quasi-cocycles}

Efficient generation is useful in establishing polynomial bounds for the growth of quasi-cocycles which we will need to define an induction scheme in \S \ref{Section: Induction for star-approximate lattices and Boundedness of quasi-cocycles}.

\begin{proposition}\label{Proposition: Bounds on quasi-cocycles from efficient generation}
Let $\Lambda$ and $G$ be as in Assumption \ref{Assumption: Arithmeticity}. Let $V$ be a Banach space equipped with a norm $|| \cdot ||$ and an action of $\langle \Lambda \rangle$ by continuous operators. Suppose that the family of operators corresponding to $\Lambda$ is uniformly bounded (equivalently, equicontinuous). Choose $q \in QC(\Lambda;V)$. If $\Lambda$ is efficiently generated, then $||q(\lambda)|| \preceq ||\lambda||_S$ as $\lambda$ ranges through $\Lambda$, where the implied constants depend on $\Lambda$ and $q$. 
\end{proposition}

 We write $f(x) \preceq g(x)$ for two functions $f,g:X \rightarrow \mathbb{R}_{\geq 0}$ when there is a positive polynomial $p$ such that $f(x) \leq p(g(x))$ for all $x \in X$. Note that, while being somewhat ambiguous, it will be clear from the context what the maps $f,g$ and the set $X$ are. We point out that in Proposition \ref{Proposition: Bounds on quasi-cocycles from efficient generation} we have used $|| \cdot ||_S$ rather than $|| \cdot ||_S'$. This is because when $g$ runs through $G$, $||g||_S' \preceq ||g||_S$ and $||g||_S \preceq ||g||_S'$. So the two objects are interchangeable. We thus drop the $'$ from now on, but we note that $||\cdot ||_S'$ is used implicitly below when appealing to sub-multiplicativity. 
 
\begin{proof}
Let $F \subset \langle \Lambda \rangle$ and $C_0>0$ be as in Definition \ref{Definition: Efficient generation}. Take $\lambda \in \Lambda$ and choose $\lambda_0=e, \ldots, \lambda_m= \lambda \in \Lambda$ as in Definition \ref{Definition: Efficient generation}.  We will show by induction that there are constants $c_1,c_2 \geq 1$ such that for all $l \in \{0;\ldots;m\}$, $||q(\lambda_l)|| \leq c_1^l +  c_2$. For $l \in \{0;\ldots;m\}$ there is some $j \in \{0;\ldots;l\}$ such that $ \lambda_{l+1}\lambda_j^{-1} \in F(\lambda_{l}\lambda_j^{-1})F^{-1}$. Since $q$ is a quasi-cocycle,

$$||q(\lambda_{l+1}\lambda_j^{-1})|| \leq A \left(||q(\lambda_l)|| + ||q(\lambda_j)||\right) + B $$
for some constants $A,B \geq 1$ depending on $F$ and $q$ but not on $\lambda$ or $l$. Using again that $q$ is a cocycle, 
$$||q(\lambda_{l+1})|| \leq C \left(||q(\lambda_l)|| + ||q(\lambda_j)||\right) + D$$
for some other constants $C,D\geq 1$ depending on $F$ and $q$ but not on $\lambda$ or $l$. We thus find by induction on $l$, 
$$||q(\lambda)|| \leq c_1^m + c_2$$
where $c_1 = 2C$ and $c_2$ is a constant depending on $C$ and $D$ only. Since $m \leq C_0 \log||\lambda||_S$, we find that $||q(\lambda)|| \preceq ||\lambda||_S$.  
\end{proof}

As a direct corollary of the logarithmic bound in Definition \ref{Definition: Efficient generation} we find:

\begin{corollary}\label{Corollary: Logarithmic bound quasi-morphisms}
With the notation of Proposition \ref{Proposition: Bounds on quasi-cocycles from efficient generation}. Suppose in addition that $\Lambda$ acts trivially on $V$. Then $||q(\lambda)|| =O(\log(||\lambda||_S))$. 
\end{corollary}

To prove Theorem \ref{THEOREM: EFFICIENT MESSY GENERATOIN}, we mimic the strategy from \cite{LubotzkyMozesRaghunathan} (note however that our result works even for some rank one approximate lattices).

\subsection{Reductions to factors and subvarieties}

We make a few elementary observations that simplify greatly the proofs.

\begin{lemma}\label{Lemma: Efficient generation and commensurability}
Let $\Lambda_1, \Lambda_2$ be two commensurable approximate lattices of some semi-simple $S$-adic algebraic group. If $\Lambda_2 \subset \langle \Lambda_1 \rangle$ and $\Lambda_2$ is efficiently generated, then so is $\Lambda_1$.
\end{lemma}

\begin{proof}
Let $F_1 \subset \langle \Lambda_1 \rangle$ be a finite symmetric subset such that $\Lambda_2$ is efficiently generated by $F_1$. Choose $F_2 \subset \langle \Lambda_1 \rangle$ such that  $\Lambda_1 \subset F_2\Lambda_2$ and $\Lambda_2 \subset F_2\Lambda_1$. Then $\Lambda_1$ is efficiently generated by $F_1F_2$ with the same implied constants. 
\end{proof}

\begin{lemma}\label{Lemma: Quasi-isometric decomposition approximate lattices in product}
Let $\Lambda$ be a uniform approximate lattice in a product $G_1 \times \cdots \times G_n$ of locally compact groups. Choose compact symmetric neighbourhoods of the identity $W_1 \subset G_2 \times \cdots \times G_n, W_2 \subset G_1 \times G_3 \times \cdots \times G_n,\ldots, W_n \subset G_1 \times \cdots \times G_{n-1}$. Write $\Lambda_i = \left(G_i \times W_i\right) \cap \Lambda^2$ for all $i \in \{1, \ldots, n\}$. Then:
\begin{enumerate}
\item the subset $\Lambda_1\cdots \Lambda_n$ is commensurable with $\Lambda$; 
\item If $\rho: G \rightarrow \mathbb{R}_{\geq 0}$ denotes a continuous submultiplicative map, $p_i: G \rightarrow G_i$ denotes the natural projection and $\lambda =\lambda_1 \cdots \lambda_n \in \Lambda_1\cdots \Lambda_n$, then 
$$\rho(\lambda_i) \preceq \rho(p_i(\lambda))\preceq \rho(\lambda).$$
\end{enumerate}
\end{lemma}

\begin{proof}
We will prove the result for $n = 2$, the general result then follows by induction on $n$. Note that $\Lambda_1$ and $\Lambda_2$ are relatively dense in $G_1 \times W_1$ and $W_2 \times G_2$ respectively (Lemma \ref{Lemma: Intersection of approximate subgroups}). Choose two compact subsets $K_1, K_2$ such that $G_1 \times \{e\} \subset \Lambda_1K_1$ and $\{e\}  \times G_2\subset \Lambda_2K_2$. If $g=(g_1,g_2) \in G_1 \times G_2$, we can find $\lambda_1 \in \Lambda_1$ such that $\lambda_1^{-1}g_1 \in K_1$ and $\lambda_2 \in \Lambda_2$ such that $\lambda_2^{-1}(e,p_2(\lambda_1)^{-1}g_2) \in K_2$. Hence, $$(\lambda_1\lambda_2)^{-1}g \in \left(\left(W_2p_1(K_1)\right) \times \{e\} \right)K_2=:K.$$ So $X=\Lambda_1\Lambda_2$ is relatively dense. But $X \subset \Lambda^4$. So $X^{-1}\Lambda \subset \Lambda^5$. Take any $\lambda \in \Lambda$, we have $\lambda_1\in \Lambda_1,\lambda_2 \in \Lambda_2$ such that $(\lambda_1\lambda_2)^{-1}\lambda \in K \cap \Lambda^5=:F$. Since $\Lambda^5$ is uniformly discrete, $F$ must be finite. So $\Lambda \subset XF$. This proves (1).  

Moreover, $p_1(\lambda_1\lambda_2)=p_1(\lambda_1)p_1(\lambda_2)$ and $p_1(\lambda_2) \in W_2$. Thus, $$\rho(p_1(\lambda_1)) \preceq \rho(p_1(\lambda_1\lambda_2)) = \rho(p_1(\lambda)).$$
But $\lambda_1 \in G_1 \times W_1$ so 
$$ \rho(\lambda_1) \preceq \rho(p_1(\lambda))\preceq \rho(\lambda).$$
So (2) is proved.
\end{proof}

\begin{lemma}\label{Lemma: Reduction to subvarieties via regular isomorphism}
Let $\mathbf{H}$ be a $K$-subgroup of $\mathbf{G}$. Suppose that there are two $K$-subvarieties $\mathbf{X},\mathbf{Y}$ of $\mathbf{H}$ such that the multiplication map $\mathbf{X} \times \mathbf{Y} \rightarrow \mathbf{H}$ is a regular $K$-isomorphism.  If $h \in \mathbf{H}(\mathbb{A}_{K,S})$, $x(h) \in \mathbf{X}(\mathbb{A}_{K,S})$ and $y(h) \in \mathbf{Y}(\mathbb{A}_{K,S})$ are such that $x(h)y(h) = h $, then:
\begin{enumerate}
 \item $||x(h)||_S \preceq ||h||_S$ and $||y(h)||_S \preceq ||h||_S$;
 \item if $e \in \mathbf{X}(K)$ and $e \in \mathbf{Y}(K)$, then there is a model set $\Lambda \subset \mathbf{H}(\mathcal{O}_{K,S})$ such that $x(\lambda),y(\lambda) \in \mathbf{G}(\mathcal{O}_{K,S})$ for all $\lambda \in \Lambda$;
 \item suppose moreover that both $\mathbf{X}$ and $\mathbf{H}$ are unipotent $K$-subgroups, choose a place $v \in S$, let $\Lambda_X \subset \mathbf{X}(K_v)$ and $\Lambda_H \subset \mathbf{H}(K_v)$ be two approximate lattices commensurable with $\mathbf{X}(\mathcal{O}_{K,v})$ and $\mathbf{H}(\mathcal{O}_{K,v})$ respectively, then there is $l \geq 0$ and $\Lambda'_H$ a model set commensurable with $\Lambda_H$ and contained in $\Lambda_H^2$ such that $x(\Lambda'_H) \subset \Lambda_X^l$.
\end{enumerate} 
\end{lemma}

\begin{proof}
Because the map $\mathbf{X} \times \mathbf{Y} \rightarrow \mathbf{H}$ is a $K$-regular isomorphism, the two maps $x$ and $y$ are $K$-regular as well. Therefore, the entries of the matrix $x(h)$ are polynomials in the entries of $h$. Thus, $||x(h)||_S  \preceq ||h||_S$. The same applies to $y$. Hence, (1) is proved. Since $e\cdot e =e$, we have $x(e)=y(e)=e$. So we see that (2) is a consequence of Lemma \ref{Lemma: Pisot approximate rings and polynomials}. We now turn to the proof of (3). Upon considering $\Lambda_X^2 \cap \mathbf{X}(\mathcal{O}_{K,v})^2$ we may assume that $\Lambda_X \subset \mathbf{X}(\mathcal{O}_{K,v})^2$. The subgroup $\langle \Lambda_X \rangle$ has finite index in $\langle \mathbf{X}(\mathcal{O}_{K,v}) \rangle = \mathbf{X}(\mathcal{O}_{K,S'})$ where $S'$ is the union of $v$ and the set of Archimedean places of $K$, see \S \ref{Subsubsection: Pisot--Vijayaraghavan--Salem numbers of a number field}. Since $\mathbf{X}$ is unipotent it has the congruence subgroup property (see the introduction of \cite{PMIHES_1976__46__107_0} for this and definitions pertaining to the congruence subgroup property). So there is an ideal $I$ of $\mathcal{O}_{K,S'}$ such that $\mathbf{X}(I) \subset \langle \Lambda_X \rangle$. Since $x(e)=e$, $x(h) \in \mathbf{X}(I)$ for all $h \in \mathbf{H}(I)$. Therefore, if $h \in \Lambda \cap \Lambda_H^2 \cap \mathbf{H}(I)$ where $\Lambda$ is provided by (2), then $x(h)\in \langle \Lambda_X \rangle \cap \mathbf{X}(\mathcal{O}_{K,S})$. But $\mathbf{X}(\mathcal{O}_{K,S})$ is commensurable with $\Lambda_X$, so $\langle \Lambda_X \rangle \cap \mathbf{X}(\mathcal{O}_{K,S}) \subset \Lambda_X^l$ for some $l > 0$. 
\end{proof}

\subsection{Bounded generation by rank one approximate subgroups}

We start by extending a result of Borel's \cite[Prop. 15.6]{zbMATH07143875} concerning double co-sets under the left-action of an arithmetic subgroup and the right-action of a Borel $K$-subgroup. 

\begin{proposition}\label{Proposition: Generalisation of Borel's result}
Let $K$, $S$, $\mathbf{G}$ and $\Lambda$ be as in Assumption \ref{Assumption: Arithmeticity}. Let $\mathbf{P} \subset \mathbf{G}$ be a parabolic subgroup. There is a finite subset $F \subset \mathbf{G}(K)$ such that $\Lambda F \mathbf{P}(K) = \mathbf{G}(K)$.

Given $g \in \mathbf{G}(K)$ and $S_g:=\{v \in S_K: |g|_v > 1\}$, we can moreover choose $p \in \mathbf{G}(\mathcal{O}_{K,S_g})$ with $||p||_{S_g} \preceq ||g||_{S_g\setminus S}$ such that $gp^{-1} \in \Lambda F$. Here, the implied constants depend on $S_g$ only.
\end{proposition}

A key ingredient is Dirichlet's unit theorem. We will use it via the following corollary. 

\begin{corollary}
Let $S$ be any set of places of $K$ and let $\mathbf{A}$ be a $K$-split torus. Choose $v_0$ a place not equivalent to a place in $S$. Then $\mathbf{A}(\mathcal{O}_{K,\{v_0\} \cup S})$ is a uniform approximate lattice in $\mathbf{A}(\mathbb{A}_{K, S})$. 
\end{corollary}

\begin{proof}
Dirichlet's unit theorem \cite[Prop. VI.1.1]{NeukirchAlgebraicNumberTheory} implies that the above is true when $S$ contains all Archimedean places of $K$. Let $S'$ be the union of $S$ and the Archimedean places of $K$. Then $\mathbf{A}(\mathcal{O}_{K,\{v_0\} \cup S})$ is a model set obtained from the cut-and-project scheme $(\mathbf{A}(\mathcal{O}_{K,\{v_0\} \cup S'}), \mathbf{A}(\mathbb{A}_{K, S}), \mathbf{A}(\mathbb{A}_{K, S'\setminus S}) )$. So it is indeed a uniform approximate lattice (\cite[Prop 2.13]{bjorklund2016approximate}). 
\end{proof}

\begin{proof}
The proof will be essentially identical to Godement's proof of \cite[Prop. 15.6]{zbMATH07143875} from \cite{SB_1962-1964__8__201_0}. Fix $g \in \mathbf{G}(K)$. Suppose that $\mathbf{P}$ is minimal. We thus have the Langlands decomposition $\mathbf{P}(K)=\mathbf{M}(K)\mathbf{A}(K)\mathbf{N}(K)$ where $\mathbf{N}$ is unipotent, $\mathbf{A}$ is a $K$-split torus and $\mathbf{M}$ is a $K$-anisotropic reductive group. So $\Lambda_{mn}=(\mathbf{MN})(\mathcal{O}_{K,S_K\setminus S})$ is a uniform approximate lattice in $\mathbf{MN}(\mathbb{A}_{K}^{S})$ (Subsection \ref{Subsubsection: Pisot--Vijayaraghavan--Salem numbers of a number field}). Moreover, given any place $v_0 \in S$, $\Lambda_a:=\mathbf{A}(\mathcal{O}_{K,\{v_0\} \cup \left(S_K\setminus S\right)})$ is a uniform approximate lattice in $\mathbf{A}(\mathbb{A}_{K}^{S})$ by the above consequence of Dirichlet's unit theorem. Therefore, there is a compact subset $C$ such that $\mathbf{P}(\mathbb{A}_{K}^{S}) = C \Lambda_a \Lambda_{mn}$. But $\mathbf{P}(\mathbb{A}_K^S)$ is a co-compact subgroup in $\mathbf{G}(\mathbb{A}_K^S)$. There is thus a further compact subset $C'$ such that $\mathbf{G}(\mathbb{A}_K^S)=C'\Lambda_a\Lambda_{mn}$. 

Now, the map $\iota_S:\mathbf{G}(K) \rightarrow \mathbf{G}(\mathbb{A}_K^S)$ is a good model of the pair $(\mathbf{G}(\mathcal{O}_{K,S}), \mathbf{G}(K))$. Take $W$ a neighbourhood of the identity in $\mathbf{G}(\mathbb{A}_K^S)$ such that $\iota_S^{-1}(W) \subset \Lambda F_1$ for some finite subset $F_1 \subset \mathbf{G}(K)$. Since $\mathbf{G}(K) \subset \mathbf{G}(\mathbb{A}_K^S)$ is dense \cite[\S II.6]{MR1090825}, we may find a finite subset $F_2 \subset \mathbf{G}(K)$ such that $C' \subset WF_2$. Therefore, $\Lambda F_1F_2 \Lambda_a\Lambda_{mn}=\mathbf{G}(K)$. This proves the first part of Proposition \ref{Proposition: Generalisation of Borel's result}. 

Let us turn to the second part. Given $g \in \mathbf{G}(K)$ take any $ \lambda_a \in \Lambda_a$ and $\lambda_{mn} \in \Lambda_{mn}$ such that $g \lambda_{mn}^{-1}\lambda_a^{-1} \in \Lambda F$ where $F=F_1F_2$. Choose $\lambda \in \Lambda$ and $f \in F$ such that $g \lambda_{mn}^{-1}\lambda_a^{-1} = \lambda f$. Since the multiplication map $\mathbf{A}(K) \times \mathbf{MN}(K) \rightarrow \mathbf{P}(K)$ is a regular isomorphism, we know that for all $v \in S_K$, both $||\lambda_a||_v$ and $||\lambda_{mn}||_v$ are bounded above by a polynomial in $||\lambda_a\lambda_{mn}||_v$. In addition, if $S'$ denotes the set of Archimedean places $v$ such that $||\lambda f ||_v=1$, then $S_K \setminus S'$ is finite and $||\lambda_a||_v = ||\lambda_{mn}||_v=1$ for all $v \in S' \setminus S_g$. Write $S'_g:= S_g \cup S_K\setminus S'$.

If $v \in S$, then $||\lambda_{mn}||_v \leq 2$ (in fact, $1$ for $v$ non-Archimedean and $2$ for $v$ Archimedean). If $v \notin S$, then $||\lambda_{mn}||_v\preceq ||\lambda_a\lambda_{mn}||_v=||f^{-1}\lambda^{-1} g||_v$. But $||\lambda||_v =1$ and $F$ is finite. So $||\lambda_{mn}||_v$ is bounded by a polynomial in $||g||_v$. Multiplying the above inequalities, we find
$$ ||\lambda_{mn}||_{S_g} \preceq ||g||_{S_g\setminus S}.$$

Just as above we see that, for all $v \notin S$, already $||\lambda_a||_v$ is bounded by a polynomial in $||g||_v$. Moreover, if $v \in S \setminus \{v_0\}$ , then $||\lambda_a||_v \leq 2$. It remains to consider $v=v_0$. Since $S_g'$ is a finite set, $\mathbf{A}(A_{K,S_g'\setminus S})$ is compactly generated. We have $\lambda_a \in \mathbf{A}(\mathcal{O}_{K,\{v_0\} \cup S_g'\setminus S})$ and $\mathbf{A}(\mathcal{O}_{K,\{v_0\} \cup S_g'\setminus S})$ is a uniform approximate lattice in $\mathbf{A}(A_{K,S_g'\setminus S})$  as a consequence of Dirichlet's unit theorem. Therefore, there is a finite subset $F'$ of $\mathbf{A}(\mathcal{O}_{K,\{v_0\} \cup S_g'\setminus S})^2$ such that $\lambda_a$ has length at most $l:=C_1\sum_{v \in S_g' \setminus S} |\log(||\lambda_a||'_v)|$ in the word metric associated to $F'$ ( \cite[Thm. 3.4]{bjorklund2016approximate}). But the projection of $\mathbf{A}(K)$ to $\mathbf{A}(\mathbb{A}_{K,S_g \setminus S'})$ is injective. So $\lambda_a$ seen as an element of $\mathbf{A}(K)$ has length at most $l$ in the word metric associated with $F'$. There is thus a constant $C_2:=C_2(v_0)\geq 0$ such that we have 
$$|\log(||\lambda_a||_{v_0})| \leq C_2l.$$
Combining all of the above we find
$$ ||\lambda_a||_{S_g} \preceq ||g||_{S_g \setminus S}.$$
Setting $p= \lambda_a\lambda_{mn}$ we obtain the desired result.
\end{proof}

Using Proposition \ref{Proposition: Generalisation of Borel's result} we are able to generalise a result of Witte-Morris \cite{WitteMorris}, itself built upon a method due to Lubotzky--Mozes--Raghunathan \cite{LubotzkyMozesRaghunathan}.

\begin{proposition}\label{Proposition: Pushing into parabolic subgroups using root subgroups}
Let $K$ be a number field, $\mathbf{G}$ be an absolutely almost simple simply connected $K$-group and $v$ be a place of $K$. Let $\Lambda \subset \mathbf{G}(K_v)$ be an approximate subgroup commensurable with $\mathbf{G}(\mathcal{O}_{K,v})$. Let $\mathbf{T}$ be a maximal $K$-split torus and let $\mathbf{G}_1, \ldots, \mathbf{G}_n$ be the $K$-rank $1$ simple subgroups associated with the positive roots of $\mathbf{G}$ (and in the same order as in \cite[p6]{WitteMorris}). Let $\mathbf{P}$ be the minimal parabolic subgroup containing $\mathbf{T}$. Then for every $\gamma \in \Lambda$ there are $\gamma_1 \in \Lambda^2 \cap \mathbf{G}_1(K_v), \ldots, \gamma_n \in \Lambda^2 \cap \mathbf{G}_n(K_v)$ with $||\gamma_1||_v, \ldots,||\gamma_n||_v\preceq ||\gamma||_v$ such that $\gamma_1\cdots \gamma_n \gamma \in \mathbf{P}(\mathcal{O}_{K,v})$. 
\end{proposition}

In the following subsection, we will pair Proposition \ref{Proposition: Pushing into parabolic subgroups using root subgroups} with this observation: 

\begin{lemma}\label{Lemma: Approximate lattices and Langlands decomposition}
With $\Lambda$ and $\mathbf{G}$ as in Assumption \ref{Assumption: Arithmeticity}. Let $\mathbf{P}$ be a minimal parabolic subgroup of $\mathbf{G}$. Write the Langlands decomposition $\mathbf{P}= \mathbf{MAN}$ where $\mathbf{N}$ is the unipotent radical of $\mathbf{P}$, $\mathbf{M}$ is $K$-anisotropic and $\mathbf{A}$ is a $K$-split torus. Then $\Lambda^2 \cap \mathbf{P}(K)$ is commensurable with 
$$\Lambda_{mn}:=\left(\Lambda^2 \cap \mathbf{M}(K)\right)\left( \Lambda^2 \cap \mathbf{N}(K)\right).$$
\end{lemma}

\begin{remark}
Using \cite{zbMATH07143875}, Proposition \ref{Proposition: Generalisation of Borel's result} should in theory allow us to exhibit Siegel sets that are fundamental domains for approximate lattices. Thus, showing that the reduction theory of approximate lattices is, to some extent, a generalisation of the one of lattices. We mention here that related ideas have been investigated at length in \cite{https://doi.org/10.48550/arxiv.2204.01496}.
\end{remark}

\subsection{Efficient generation in root subgroups}
We study now the connectedness of arithmetic approximate subgroups of root subgroups. We start with a simple lemma that corresponds to the situation where $\mathbf{U}_{\alpha}$ is abelian. 

\begin{lemma}\label{Lemma: Sum-product phenomenon for quick generation}
With $K$ being a number field and $v$ being a place of $K$. Let $V$ be a $K_v$ vector space. Let $\Lambda \subset V$ be a uniformly discrete approximate subgroup. Suppose that there is $\alpha \in K_v$ with $|\alpha|_v < 1$ such that $\alpha \cdot \Lambda$ is commensurable with $\Lambda$. Then there is $F \subset V$ finite such that for every $u \in \Lambda$ there are $u_0=e,u_1, \ldots, u_n=u \in \Lambda$ such that: 
\begin{enumerate}
\item $n = O(|\alpha|_v^{-1}\log ||u||)$ where $||\cdot||$ denotes some fixed $K_v$-norm on $V$; 
\item for each $i\in \{0,\ldots,n-1\}$ there is $f_i \in F$ such that we have $u_i=\alpha u_{i+1} + f_i$. 
\end{enumerate}
\end{lemma}

\begin{proof}
Since $V$ is a vector space, we use additive notation. Take a finite subset $F_1 \subset \mathbf{U}_{\alpha}(K_v)$ such that $\alpha \cdot \Lambda \subset \Lambda + F_1$. Now define inductively $u_{-k} \in \Lambda$ as:
\begin{itemize}
\item $u_{0} = u$;
\item $u_{-k-1} = \alpha u_{-k} - f$ where $f \in F_1$ is chosen such that $u_{-k-1} \in \Lambda$.
\end{itemize}
We have that $||u_{-k-1}|| \leq |\alpha|_v||u_{-k}|| + \sup_{f \in F}||f||$ for all $k \geq 0$. Therefore, there is an integer $n \geq 0$ with $n = O(|\alpha|_v^{-1}\log ||u||)$ such that $$||u_{-n}|| \leq (1- |\alpha|_v)^{-1}\sup_{f \in F}||f|| =: C.$$ 
Set now $u_0=e$, $u_i = u_{i-n-1}$, $u_{n+1}=u$ and $F=F_1 \cup F_2$ where $F_2$ is the set of elements of $\Lambda$ of norm at most $C$. This satisfies the conclusions of Lemma \ref{Lemma: Connectedness of sorts in root subgroups}. 
\end{proof}

Now we can consider the general case. 

\begin{lemma}\label{Lemma: Connectedness of sorts in root subgroups}
With the notation from Assumption \ref{Assumption: Arithmeticity}. Suppose that $S=\{v\}$ and $\Lambda$ is not contained in a lattice. Let $\alpha$ be a $K$-root of $\mathbf{G}$. Let $\mathbf{T}$ be a maximal $K$-split torus of $\mathbf{G}$ and let $\mathbf{U}_{\alpha}$ denote the root subgroup associated to $\alpha$ and $\mathbf{T}$. Then there are $K$-subgroups $\mathbf{U}_1, \ldots, \mathbf{U}_r$ of $\mathbf{U}_{\alpha}$ and $\delta \in \Lambda^m \cap \mathbf{T}(K)$ for some integer $m \geq 0$ such that: 
\begin{enumerate}
\item $\mathbf{U}_i$ is normalised by $\delta$ for all $i \in \{1,\ldots, r\}$; 
\item $|\alpha(\delta)|_v< 1$; 
\item the multiplication map $m:\mathbf{U}_1 \times \cdots \times \mathbf{U}_r \rightarrow \mathbf{U}_{\alpha}$ is a $K$-regular isomorphism;
\item $\mathbf{U}_i(K)$ is isomorphic to the additive group $\mathbf{G}_a(K)$ of $K$ and the automorphism of $\mathbf{G}_a(K)$ that stems from conjugation by $\delta$ is simply multiplication by $\alpha(\delta)$;
\item if, for all $i \in \{1, \ldots, r\}$, $u_i$ denotes the composition of $m^{-1}$ and projection to $\mathbf{U}_i$, then there is $\Lambda'$ contained in and commensurable with  $\mathbf{U}_{\alpha}(\mathcal{O}_{K,v})$ such that $u_i(\Lambda') \subset \mathbf{U}_{\alpha}(K) \cap \Lambda^l$ for some integer $l$.
\end{enumerate}
\end{lemma}

\begin{proof}
We know that the group generated by $\Lambda$ is commensurable with $\mathbf{G}(\mathcal{O}_{K,S'})$ for some finite subset $S'$ containing $v$. Thus, the assumption implies that $S'$ is not equal to $\{v\}$. Since $S'$ has size at least $2$, we know that $\mathbf{T}(\mathcal{O}_{K,S'})$ is relatively dense in $\mathbf{T}(K_v)$ by  Dirichlet's unit theorem. But $\mathbf{T}(\mathcal{O}_{K,S'})$ is commensurable with $\langle \Lambda \rangle \cap \mathbf{T}(K_v)$, so there is an integer $m \geq 0$ such that $\Lambda^m$ contains $\delta$ such that $|\alpha(\delta)|_v < 1$.

 Let $\mathfrak{u}_{\alpha}$ denote the Lie algebra of $\mathbf{U}_{\alpha}(K_v)$. We know that $\mathfrak{u}_{\alpha}$ decomposes as the direct sum  $E_1 \oplus E_2$  of eigenspaces of the action of $\Ad(\delta)$ associated to the eigenvalues $\alpha(\delta)$ and $2\alpha(\delta)$ respectively. Moreover, $E_2$ is contained in the center of $\mathfrak{u}_{\alpha}$. Write $\exp: \mathfrak{u}_{\alpha} \rightarrow \mathbf{U}_{\alpha}(K_v)$ the exponential map. We can find a $K_v$-basis $(e_1, \ldots, e_r)$ of $\mathfrak{u}_{\alpha}$ adapted to the sum $E_1 \oplus E_2$. We may assume moreover that $\exp(e_i) \in \mathbf{U}_{\alpha}(\mathcal{O}_{K,v})$. Therefore, the group $U_i:=\exp\left(K_ve_i\right)$ is defined over $K$. Now the natural multiplication map $U_1 \times \cdots \times U_r \rightarrow \mathbf{U}_{\alpha}(K_v)$ is defined over $K$ and isomorphism. Indeed, the surjectivity comes from the BCH formula and the injectivity is immediate once we note that $U_i \cap \langle U_{i+1}, \ldots, U_r\rangle =\{e\}$ for all $i \in \{1, \dots, r\}$. According to the BCH formula, given $g \in \mathbf{U}_{\alpha}(K_v)$ we have $g=u_1(g) \cdots u_r(g)$ where the maps $g \mapsto u_i(g)$ are $K$-regular maps for all $i \in \{1,\ldots, r\}$. This proves (1)-(4). It remains only to prove (5). But $u_1, \ldots, u_r$ are finitely many $K$-regular maps and $u_1(e)=\ldots=u_r(e)=e$. We conclude by invoking Lemma \ref{Lemma: Reduction to subvarieties via regular isomorphism}.
\end{proof}

\subsection{Proof of Theorem \ref{THEOREM: EFFICIENT MESSY GENERATOIN}}
We start with two technical results.

\begin{lemma}\label{Lemma: distortion in parabolic subgroups}
Let $K$, $S$, $\mathbf{G}$ and $\Lambda$ be as in Assumption \ref{Assumption: Arithmeticity}. Let $\mathbf{P}$ be a minimal parabolic subgroup. Then $\Lambda^m \cap \mathbf{P}(\mathbb{A}_{K,S})$ is efficiently generated.
\end{lemma}

\begin{proof} In this proof we will use Lemmas \ref{Lemma: Approximate lattices and Langlands decomposition}, \ref{Lemma: Connectedness of sorts in root subgroups} and \ref{Lemma: Quasi-isometric decomposition approximate lattices in product} to reduce to proving Lemma \ref{Lemma: distortion in parabolic subgroups} in unipotent groups. As we state them, the successive reductions in fact imply that some approximate subgroup commensurable with $\Lambda^2 \cap \mathbf{P}(K)$ and contained in $\langle \Lambda \rangle$ is efficiently generated. But we know that this is sufficient according to Lemma \ref{Lemma: Efficient generation and commensurability}.

Write the Langlands decomposition $\mathbf{P}= (\mathbf{MAN})(K)$ where $\mathbf{N}$ is the unipotent radical of $\mathbf{P}$, $\mathbf{M}$ is $K$-anisotropic and $\mathbf{A}$ is a $K$-split torus.  We have that $\Lambda^2 \cap \mathbf{M}(\mathbb{A}_{K,S})$ is relatively dense in $\mathbf{M}(\mathbb{A}_{K,S})$ by \S \ref{Subsubsection: Pisot--Vijayaraghavan--Salem numbers of a number field}. But $\mathbf{M}(\mathbb{A}_{K,S})$ is compactly generated and undistorted in $\mathbf{G}(\mathbb{A}_{K,S})$. Therefore, by \cite[Thm 3.4]{bjorklund2016approximate}, there is a constant $C>0$ such that for every $\lambda \in \Lambda^2 \cap \mathbf{M}(\mathbb{A}_{K,S})$ there are $\lambda_0=e, \lambda_1, \ldots, \lambda_n=\lambda \in \Lambda^2 \cap \mathbf{M}(\mathbb{A}_{K,S})$ such that $||\lambda_{i+1}\lambda_i^{-1}||_S \leq C$ and $n \leq C \log ||\lambda||_S$.

Similarly, note that $\mathbf{A}(\mathbb{A}_{K,S})$ is isomorphic as a locally compact group to $\mathbb{Z}^n\times \mathbb{R}^m \times K$ where $K$ is a compact subgroup. Therefore, if we consider $\mathbb{Z}^n\times \mathbb{R}^m \times K \subset \mathbb{R}^{m+n} \times K$, then a consequence of Schreiber's theorem (\cite[Cor. 1]{machado2019infinite}) implies that that there is a constant $C>0$ such that for every $\lambda \in \Lambda^2 \cap \mathbf{A}(\mathbb{A}_{K,S})$ there are $\lambda_0=e, \lambda_1, \ldots, \lambda_n=\lambda \in \Lambda^2 \cap \mathbf{A}(\mathbb{A}_{K,S})$ such that $||\lambda_{i+1}\lambda_i^{-1}||_S \leq C$ and $n \leq C \log ||\lambda||_S$.

Let us now look at $\Lambda^2 \cap \mathbf{N}(\mathbb{A}_{K,S})$. We know that $\Lambda^2 \cap \mathbf{N}(\mathbb{A}_{K,S})$ is a uniform approximate lattice in $\mathbf{N}(\mathbb{A}_{K,S})$. Since $S$ is finite, we see that $\mathbf{N}(\mathbb{A}_{K,S}) = \prod_{v \in S}\mathbf{N}(K_v)$. By Lemma \ref{Lemma: Quasi-isometric decomposition approximate lattices in product}, we therefore have that $\Lambda^2 \cap \mathbf{N}(\mathbb{A}_{K,S})$ is commensurable with $\left(\Lambda^2 \cap \mathbf{N}(\mathcal{O}_{K,v_1})\right)\cdots\left( \Lambda^2 \cap \mathbf{N}(\mathcal{O}_{K,v_n})\right)$ where $v_1, \ldots, v_n$ is an enumeration of $S$. So, by Part 2. of Lemma \ref{Lemma: Quasi-isometric decomposition approximate lattices in product}, we only have to prove the case $S=\{v\}$. According to the structure of semi-simple groups, there are roots $\alpha_1, \ldots, \alpha_m$ relative to $\mathbf{A}$ such that the multiplication map $\mathbf{U}_{\alpha_1} \times \cdots \times \mathbf{U}_{\alpha_m} \rightarrow \mathbf{N}$ is a regular $K$-isomorphism. According to Lemma \ref{Lemma: Reduction to subvarieties via regular isomorphism}, it thus suffices to prove the results for $\Lambda^2 \cap \mathbf{U}_{\alpha}(K_v)$ for all $K$-roots $\alpha$. Let $\mathbf{U}_1, \ldots, \mathbf{U}_l$ be the $K$-subgroups and $\delta \in \mathbf{A}(\mathcal{O}_{K,v})\cap \Lambda^m$ - for some integer $m > 0$ - given by Lemma \ref{Lemma: Connectedness of sorts in root subgroups}. It suffices to prove Lemma \ref{Lemma: distortion in parabolic subgroups} for elements of $\Lambda^2 \cap \mathbf{U}_i(K_v)$ for $i \in \{1, \ldots, l\}$. Take $u \in \Lambda^2 \cap \mathbf{U}_i(K_v)$ for $i \in \{1, \ldots, l\}$. Let $\phi$ denote the $K$-isomorphism $\mathbf{U}_i \rightarrow \mathbf{G}_a$. Set $\Lambda_i = \phi(\Lambda^2 \cap \mathbf{U}_i(K_v))$. Then $|\phi(u)|_v \preceq ||u||_v$ where $| \cdot |_v$ denotes the absolute value on $\mathbf{G}_a(K_v) = K_v$ (the equality being regarded as an equality of sets). By Lemma \ref{Lemma: Sum-product phenomenon for quick generation} there are a constant $C > 0$ independent of $u$, an integer $ k\leq C \log |\phi(u)|_v$, a finite subset $F \subset \langle \Lambda \rangle \cap \mathbf{U}_i(K_v)$  and elements $u_0=e, u_1, \ldots, u_k=u \in \mathbf{U}_i(\mathcal{O}_{K,v})$ such that for all $j \in \{0, \ldots, k-1\}$ we have $\delta^{-1}u_{j-1}f_j\delta=u_j$ for some $f_j \in F$. The proof is complete.
\end{proof}

\begin{lemma}[Proof of Lemma 4.7, \cite{LubotzkyMozesRaghunathan}]\label{Lemma: LMS in rank one, technical bit}
Let $K$, $S$, $\mathbf{G}$ and $\Lambda$ be as in Assumption \ref{Assumption: Arithmeticity}. Suppose that $\rank_K\left(\mathbf{G}\right)=1$. Then for all $\lambda \in \Lambda$ there are a finite subset $\Sigma \subset \mathbf{G}(K)$, $\Theta_1, \ldots, \Theta_r \in \Lambda$  and a constant $B > 0$ such that 
\begin{enumerate}
\item $$\frac{1}{B} \sum_{i=1}^r ||\Theta_{i-1}^{-1}\Theta_i||_S \leq ||\lambda||_S \leq B\sum_{i=1}^r ||\Theta_{i-1}^{-1}\Theta_i||_S ;$$
\item for all $i \in \{1, \ldots, r\}$, either $||\Theta_{i-1}^{-1}\Theta_i||_S \leq B$ or there are $\xi_{i-1}, \xi_i \in \Sigma$ such that $\xi_{i-1}^{-1}\Theta_{i-1}^{-1}\Theta_i\xi_i \in \mathbf{N}(K)$.
\end{enumerate}
\end{lemma}

\begin{proof}
Let $\hat{S}$ be the union of $S$ with the set of infinite places of $K$. The proof of \cite[Lemma 4.7]{LubotzkyMozesRaghunathan} provides such $\Theta_i \in \mathbf{G}(O_{K,\hat{S}})$ and $\Sigma$. We only need to explain why we may choose the $\Theta_i$ in $\Lambda$ (seen here as a subset of $\mathbf{G}(O_{K,\hat{S}})$). To do so we will have to modify $\Sigma$ as well. 

Fix $\lambda \in \Lambda$. In the proof of \cite[Lemma 4.7]{LubotzkyMozesRaghunathan} we first choose a $\delta$-quasi-geodesic $h_0=e, \ldots, h_r=\lambda$ from $e$ to $\lambda$ in $\mathbf{G}(\mathbb{A}_{K,S})$ for some $\delta > 0$. There is then a finite subset of indices $\{0<i_1 < \ldots < i_r=n\}$ such that $||h_{i_j}^{-1}\Theta_{i_j}||_S \leq C$ for some constant $C=C(\Lambda)$. We may choose $h_i \in  \mathbf{G}(\mathbb{A}_{K,S}) (= \mathbf{G}(\mathbb{A}_{K,S}) \times \{e\} \subset \mathbf{G}(\mathbb{A}_{K,\hat{S}}))$ for all $i\in \{0, \ldots, r\}$. Therefore, there is $\Lambda'$ commensurable with $\mathbf{G}(\mathcal{O}_{K,S})$ and hence, with $\Lambda$ such that $\Theta_i \in \Lambda'$ for all $i \in \{1, \ldots, r\}$. Consider $X$ a finite subset of $\langle \Lambda, \Lambda' \rangle$ such that $\Lambda' \subset \Lambda X$. For every $i \in \{1, \ldots, r\}$, choose $f_i \in F$ such that $\Theta_i' :=\Theta_if_i^{-1} \in \Lambda$ and set $\xi_i':=f_i\xi_i$. Setting now $\Sigma'=F\Sigma$ we obtain the desired result. 
\end{proof}

\begin{lemma}\label{Lemma: LMS in rank one}
Let $K$, $S$, and $\mathbf{G}$ be as above. Suppose that $\rank_K\left(\mathbf{G}\right)=1$. Let $\Lambda$ be commensurable with $\mathbf{G}(\mathcal{O}_{K,S})$ such that $\langle \Lambda \rangle$ is not a lattice. Then $\Lambda$ is efficiently generated.
\end{lemma}

\begin{proof}
 Fix $\lambda \in \Lambda$. Let $\Theta_1, \ldots, \Theta_r$ and $\Sigma$ be as in Lemma \ref{Lemma: LMS in rank one, technical bit}. Let $i$ be such that $\xi_i\Theta_i^{-1}\Theta_{i+1}\xi_{i+1}^{-1} \in \mathbf{N}(\mathbb{A}_{K,S})$. We will use repeatedly that $\Sigma$ is contained in the commensurator of $\Lambda$. In particular, $\xi_i \Lambda^2 \xi_i^{-1}$ is commensurable with $\Lambda$. So $\Lambda_{n,i}=\Lambda^2 \cap \xi_i \Lambda^2 \xi_i^{-1} \cap \mathbf{N}(\mathbb{A}_{K,S})$ is commensurable with $\Lambda^2 \cap \mathbf{N}(\mathbb{A}_{K,S})$ and, hence, is a uniform approximate lattice. So there is a finite subset $F_i$, depending on $\xi_i$ only, such that there is $\lambda_i \in \Lambda_{n,i}$ with $\xi_i\Theta_i^{-1}\Theta_{i+1}\xi_{i+1}^{-1} = \lambda_if$ for some $f \in F_i$. Therefore, there is a finite subset $F_i' \subset \langle \Lambda \rangle$ such that $\xi_i^{-1}\lambda_i^{-1}\xi_i\Theta_i^{-1}\Theta_{i+1} \in F_i' \subset \xi_i^{-1}F_i\xi_{i+1}$. So $||\xi_i^{-1}\lambda_i^{-1}\xi_i||_S \preceq ||\Theta_i^{-1}\Theta_{i+1}||_S$. Now, $\xi_i^{-1}\lambda_i^{-1}\xi_i \in \Lambda^2$ and $\xi_i^{-1}\lambda_i^{-1}\xi_i \in \xi_i^{-1} \mathbf{N}(\mathbb{A}_{K,S})\xi_i$. Hence, according to Lemma \ref{Lemma: distortion in parabolic subgroups} there is a finite subset $F_i'' \subset \langle \Lambda \rangle$ - depending on $\xi_i$ only - such that $\lambda_{i,0}=e, \ldots, \lambda_{i,n}=\Theta_i^{-1}\Theta_{i+1}$ with $n \leq C_i \log ||\Theta_i^{-1}\Theta_{i+1}||$ such that for $l \in \{0; \ldots; n-1\}$  we have $\lambda_{i,l+1}\lambda_{i,0}^{-1}:= F_i''\lambda_{i,l}\lambda_{i,0}^{-1}F_i''$.
But
$$\sum_{i=1}^rC_i \log ||\Theta_{i-1}^{-1}\Theta_i||_S \leq C' \log ||\lambda||_S.$$
So, putting everything back together, $\Lambda$ is efficiently generated by $\bigcup_i F_i''$ which is indeed a finite subset contained in $\langle \Lambda \rangle$.
\end{proof}

\begin{proof}[Proof of Theorem \ref{THEOREM: EFFICIENT MESSY GENERATOIN}.]
Applying Proposition \ref{Proposition: Pushing into parabolic subgroups using root subgroups} to $\Lambda$, we can find $\mathbf{P}$, $F$, $\Lambda_1, \ldots, \Lambda_n \subset \langle \Lambda \rangle$ and $\mathbf{G}_1,\ldots, \mathbf{G}_n$ as in the conclusion of Proposition \ref{Proposition: Pushing into parabolic subgroups using root subgroups}. Since $\langle \Lambda \rangle \cap \mathbf{G}_i(K_v)$ is dense in $\mathbf{G}_i(K_v)$ for all $i$, we may enlarge all the $\Lambda_i$'s such that each of them is not contained in a lattice in $\mathbf{G}_i(K_v)$, but $\Lambda_i \subset \Lambda^m \cap \mathbf{G}_i(K_v)$ for some $m \geq 0$ independent of $i$. Take $\lambda \in \Lambda$. There are $\lambda_1 \in \Lambda_1 \subset \Lambda^m, \ldots, \lambda_n \in \Lambda_n \subset \Lambda^m$ and $\lambda_{n+1} \in \mathbf{P}(K) \cap \Lambda^{2n+1}$ such that $\lambda_1\cdots \lambda_{n+1}=\lambda$ and $||\lambda_1||_v, \ldots, ||\lambda_n||_v,||\lambda_{n+1}||_v \preceq ||\lambda||_v$. According to Lemma \ref{Lemma: LMS in rank one}, for every $i \in \{1, \ldots, n\}$ we have that $\Lambda_i$ is efficiently generated. By Lemma \ref{Lemma: distortion in parabolic subgroups}, $\Lambda^{mn+1} \cap \mathbf{P}(K)$ is efficiently generated. Thus,  $\Lambda$ is efficiently generated according to Lemma \ref{Lemma: Efficient generation and commensurability}.
\end{proof}

\section{Structural quasi-cocycle of an approximate lattice}\label{Section: Conjugation-multiplication phenomena}

Given an approximate subgroup $\Lambda$ of some group $\Gamma$ that commensurates it, an abelian subgroup $V$ equipped with a bornology $\mathcal{B}$ and an action of $\Gamma$ leaving $\mathcal{B}$ invariant we define the set of  \emph{quasi-cocycles} by 
$$ QC(\Lambda;V) := \{c: \Gamma^n \rightarrow V : \forall X \in \mathcal{B}_{\Lambda}, dc(X,X) \text{ bounded }\}$$
and $\widetilde{QC}(\Lambda;V)$ its quotient modulo bounded maps. With $\Lambda_s, \Lambda_r, G$ as in the setup of our main result (Theorem \ref{THEOREM: MAIN THEOREM}), quasi-cocycles provide an easy construction of (potentially) many approximate lattices. Write $Z$ the centre of $\Rad(G)$ and equip it with the bornology $\mathcal{B}_{\Lambda_z}$ coming from $\Lambda_z :=\Lambda_r^2 \cap Z$. Then the map 
$$ q \in QC(\Lambda_s;Z) \longmapsto \{(\lambda_s, q(\lambda_s)\lambda_r) : \lambda_s \in \Lambda_s, \lambda_r \in \Lambda_r\}. $$
where $G$ is identified with $\left(G/\Rad(G)\right) \ltimes \Rad(G)$, provides many approximate lattices that are all extensions of $\Lambda_s$ by $\Lambda_r$. Our goal in this section is to reverse this process. 

\subsection{Quasi-cocycles and semi-direct products}
Consider a semi-direct product $L \ltimes A$ of two groups $L$ and $A$, with $A$ abelian and in which the product $(l_1,a_1)(l_2,a_2)$ is given by $(l_1l_2, a_1 + l_1 \cdot a_2)$.
Let $\Lambda$ be an approximate subgroup in $L \ltimes A$. Write $\Lambda_A := \Lambda^2 \cap A$ and $\Lambda_L$ the projection of $\Lambda$ to $L$. For every $m \geq 0$ and every $\lambda \in \Lambda_L^m$, we can find $a(\lambda) \in A$ such that $(\lambda, a(\lambda)) \in \Lambda^m$. For all $m_1, m_2 \geq 0$, $\lambda_1 \in \Lambda_L^{m_1}$ and $\lambda_2 \in \Lambda_L^{m_2}$ we have 
$$ (\lambda_1\lambda_2, a(\lambda_1\lambda_2)) (\lambda_2, a(\lambda_2))^{-1}  (\lambda_1, a(\lambda_1))^{-1} \in A \cap \Lambda^{2(m_1+m_2)}.$$
So:
\begin{equation}
a(\lambda_1\lambda_2) - \lambda_1 \cdot a(\lambda_2) - a(\lambda_1) \in A \cap \Lambda^{2(m_1 + m_2)}\label{Eq: quasi-cocycle equation}
\end{equation} 
where $ \cdot $ denotes the action of $L$ on $A$. In particular, $ a \in QC(\Lambda_L;A)$ where $A$ is equipped with the bornology $\mathcal{B}_{\Lambda_A}$.  The interest in studying $a$ lies in the following simple observation: 

\begin{lemma}\label{Lemma: Small values of quasi-cocycle imply splitting}
If finitely many translates of $\Lambda_A$ cover $a(\Lambda_L)$, then $\left(L \times \{0\}\right) \cap \Lambda^2$ is an approximate subgroup commensurable with $\Lambda_L \times \{0\} \subset L \times \{0\}$. 
\end{lemma}

\begin{proof}
Let $F$ be a finite subset such that $a(\Lambda_L) \subset F\Lambda_A$. For every $\lambda_L$ there is $\lambda_A$  and $f \in F$ such that $(\lambda_L,f\lambda_A) \in \Lambda$. Hence, $(\lambda_L, e) \in F^{-1}\Lambda^2$. By Lemma \ref{Lemma: Intersection of commensurable sets}, $\Lambda^2 \cap L$ is thus commensurable with $\Lambda_L \times \{0\} \subset L \times \{0\} \subset L \ltimes A$. So the result is proven.
\end{proof}

\subsection{Central radical}Let us first consider the case of a trivial action on $A$. In this situation (\ref{Eq: quasi-cocycle equation}) becomes
$$\forall \lambda_1 \in \Lambda_L^n, \forall \lambda_2 \in \Lambda_L^m, a(\lambda_1\lambda_2) - a(\lambda_1) - a(\lambda_2) \in \Lambda^{n+m} \cap A.$$
It is interesting to compare (\ref{Eq: quasi-cocycle equation}) with the defining equations of quasi-cocycles with target a normed locally compact abelian group $H$. To do so, consider a good model $\tau: A \rightarrow H$ of $(\Lambda^8 \cap A, A)$ (\cite[Thm 1.7]{machado2019goodmodels}). Suppose that $H$ is compactly generated and fix a norm $| \cdot |$ - arising from any left-invariant distance - on $H$. If $q$ denotes $\tau \circ a$ we find for all $m_1, m_2 \geq 0$ and $\lambda_1 \in \Lambda^{m_1}, \lambda_2 \in \Lambda^{m_2}$ that 
$$|q(\lambda_1\lambda_2) - q(\lambda_1)  - q(\lambda_2)| \leq  C(m_1,m_2) < \infty $$
for some constants $C(m_1,m_2)$. In other words, $q$ is a quasi-cocycle and, even, a \emph{quasi-morphism}. Moreover, $a(\Lambda_L)$ is covered by finitely many translates of $\Lambda_A$ if and only if $q(\Lambda_L)$ is bounded in $H$. The problem is therefore reduced to proving that certain quasi-morphisms and quasi-cocycles are bounded. The additional assumption that $H$ is locally compact can be arranged easily:

\begin{lemma}\label{Proposition: the right good model in central case}
With $L,A, \Lambda$ and $\tau: A \rightarrow H$ as in the previous paragraph. Write $p:L \times A \rightarrow A$ the natural projection. If $\Lambda_L$ is an approximate lattice in the semi-simple $S$-adic group $L$ and is efficiently generated, then $\overline{\tau \circ p(\langle \Lambda \rangle)}$ is compactly generated.
\end{lemma}

\begin{proof}
Let $F_0$ be a finite subset of $\langle \Lambda \rangle$ such that $\Lambda^2 \subset \Lambda F_0$. We have $$\tau \circ p(\langle \Lambda \rangle) \subset \tau \circ p (\Lambda) \tau \circ p(\langle F_0 \rangle)$$ and, by the definition of $a$, $$\tau \circ p (\Lambda) \subset \tau \circ a(\Lambda_L) \tau (\Lambda_A).$$
Note that in locally compact abelian groups, a closed subgroup of a compactly generated subgroup is also compactly generated. Since $F_0$ is finite and $\tau (\Lambda_A)$ is relatively compact, it is enough to prove that $\tau \circ a(\Lambda_L)$ is contained in a compactly generated subgroup. 

Choose $F \subset \langle \Lambda_L \rangle$ finite such that $\Lambda_L$ is efficiently generated by $F$. For $\lambda \in \Lambda_L$ we can find $\lambda_0=e,\ldots, \lambda_m=\lambda \in \Lambda_L$ such that for all $l \in \{0, \ldots, m-1\}$ there is $i \leq l$ satisfying $\lambda_{l+1}\lambda_i^{-1} \in F\lambda_l\lambda_i^{-1} F$. So (\ref{Eq: quasi-cocycle equation}) implies 
 $$a(\lambda_{l+1}) \in a(\lambda_l) + 2a(F) + \Lambda^{100} \cap A$$
 where we use $100$ as an obvious upper bound of the correct power. So there is a finite subset $X$ such that $\langle a(\Lambda_L) \cup \left(\Lambda^4 \cap A \right)\rangle = \langle X \cup \left(\Lambda^4 \cap A\right) \rangle$. Since $\tau( X \cup \left(\Lambda^4 \cap A\right))$ is relatively compact, we are done.
\end{proof}

When the action of $L$ on $A$ (hence of $\Lambda_L$ on $\Lambda_A$) is not trivial, building a good model $\tau$ of $(\Lambda_A, A)$ that is equivariant with respect to the action of $\Lambda_L$ and lands in a nicely behaved target group is not as straightforward. We address this next.

\subsection{Reductive actions without invariant vectors}\label{Subsection: Representation-theoretic sum-product phenomenon}

From now on suppose that we are given $A = \prod_{v \in S} \mathbb{Q}_v^{n_v}$ and an $S$-adic reductive Zariski-closed subgroup $L \subset \GL(A)$. Let $\Lambda$ be an approximate lattice in $L \ltimes A$ such that $\Lambda_r := \Lambda \cap A$ is a uniform approximate lattice and the projection $\Lambda_{red}$ of $\Lambda$ to $L$ is an approximate lattice. Let $a \in QC(\Lambda_L;\Lambda_A)$ be defined as in \S \ref{Section: Conjugation-multiplication phenomena}. Our goal in this section is to find a good model $\tau: \Gamma \rightarrow H$  that turns $a$ into a quasi-cocycle with coefficients in a product of finite dimensional $\mathbb{Q}_v$-vector spaces, $v \in S$.  Let us first recall Meyer's theorem for $S$-adic unipotent groups \cite[Prop. 1.1]{mac2023sadic} in a special case. 

\begin{proposition}\label{Proposition: Meyer's theorem in S-adic vector spaces}
Let $\Lambda$ be an approximate lattice in $V := \prod_{k = 1}^n \mathbb{Q}_{p_k}^{n_k}$. There is an approximate lattice $\Lambda' \subset \langle \Lambda \rangle$ commensurable with $\Lambda$ such that:
\begin{enumerate}
\item $\Lambda'$ spans a finite dimensional $\mathbb{Q}$-vector space $W$;
 
\item $\langle \Lambda' \rangle$ is a discrete subgroup of $W\otimes \mathbb{A}_S$ when embedded diagonally; 
\item there is a subgroup $B=\prod_{v \in S} V_v$ of $W\otimes \mathbb{A}_S$, with $V_v$ a $\mathbb{Q}_v$-vector subspace of $W \otimes \mathbb{Q}_v$, such that $W\otimes \mathbb{A}_S \simeq A \times B$ and $\Lambda'$ is a model set associated to a cut-and-project scheme $(A,B,\Delta)$;
\item if $\gamma \in \GL(A)$ commensurates $\Lambda$ (or, equivalently, commensurates $\Lambda'$), then $\gamma(W)=W$ and $\gamma$ gives rise to $\tilde{\gamma} \in \GL(W \otimes \mathbb{A}_S)$ such that $\tilde{\gamma}(A)=A$ and $\tilde{\gamma}(B)=B$.
\end{enumerate}
\end{proposition}

We can now prove:

\begin{proposition}\label{Proposition: quasi-cocycle takes values in Lambda}
With $L, A, \Lambda$ and $a$ as above. Suppose that $A$ does not contain any non-trivial central element. Let $W \subset A$ be the $\mathbb{Q}$-vector subspace given by Proposition \ref{Proposition: Meyer's theorem in S-adic vector spaces} applied to $\Lambda_r$. Then $a=a_1 + a_2$ with $a_1$ a $1$-cocycle and $a_2(\Lambda_{red}) \subset \frac{1}{m}\langle \Lambda_r \rangle$ for some integer $m>0$.
\end{proposition}

\begin{proof}
By Proposition \ref{Proposition: Meyer's theorem in S-adic vector spaces}, the conjugation action of $\langle \Lambda \rangle$ on $A$ stabilises $W$. Write $\rho_S$ the natural lift of that action to $W \otimes \mathbb{A}_S$ ((4) of Proposition \ref{Proposition: Meyer's theorem in S-adic vector spaces}) and $\pi: W \otimes \mathbb{A}_S \rightarrow A$ the projection. Let $\mathbf{L}$ denote the Zariski-closure of $\rho_S(\langle \Lambda \rangle)$ in $\GL(W \otimes \mathbb{A}_S)$ and let $\mathbf{U}$ denote its unipotent radical. Then $\mathbf{L}(\mathbb{A}_S)$ is contained in $\GL(A) \times \GL(B)$ where we have identified $\GL(A)$ with $\{g \in \GL(W \otimes \mathbb{A}_S): g(A)=A, g_{|B}=id_{B}\}$ and symmetrically for $\GL(B)$. Moreover, the natural projection $\pi: \GL(A) \times \GL(B) \rightarrow \GL(A)$ sends $\rho_S(\gamma)$ to $\rho(\gamma)$ for all $\gamma \in \langle \Lambda \rangle$ ((4) of Proposition \ref{Proposition: Meyer's theorem in S-adic vector spaces}). Then $\pi(\mathbf{U}(W \otimes \mathbb{A}_S))$ is a unipotent subgroup of $L$ that is normalised by $\langle \Lambda \rangle$. So $\pi(\mathbf{U}(\mathbb{Q}))=e$ since $\langle \Lambda \rangle$ is Zariski-dense in the reductive group $L$. But the projection of $W$ to $A$ is injective. So the restriction of $\pi$ to $\mathbf{L}(\mathbb{Q})$ is injective as well. Therefore, $\mathbf{U}(\mathbb{Q})$ is trivial i.e. $\mathbf{L}$ is reductive.

We claim that for any element $x \in W \otimes \mathbb{A}_S$, the $\mathbb{Q}$-span $W(x)$ of $\langle \Lambda_{red} \rangle \cdot x$ is finite-dimensional and the Zariski-connected component of $e$ of the Zariski-closure of $\langle \Lambda_{red} \rangle$ seen as a subset of $\GL(W(x))$ is reductive. Indeed, on the one hand, this is obvious when $x=y \otimes q$ with $y \in W$ and $q \in \mathbb{A}_S$. On the other hand, a general $x$ can be written $x=\sum_{i=1}^r x_i \otimes q_i$ where for all $i \in \{1, \ldots, r\}$, $x_i \in W$ and $q_i \in \mathbb{A}_S$. So 
$$W(x) \subset W(x_1 \otimes q_1) + \cdots + W(x_r \otimes q_r).$$  
So $W(x)$ is indeed finite-dimensional and the image of $\langle \Lambda_{red} \rangle$ in $\GL(W(x))$ has Zariski-closure with reductive Zariski-connected component of $e$. This proves our claim. Now, for all integers $n \geq 0$ define $W_n$ as the $\mathbb{Q}$-vector space generated by $\rho\left(\langle \Lambda_{red} \rangle\right) \left(\Lambda^n \cap A\right)$. By construction, $\Lambda^n \cap A$ is covered by finitely many translates of $W$ (Lemma \ref{Lemma: Intersection of approximate subgroups}). So there are $x_1, \ldots, x_r \in A$ such that 
$\Lambda^n \cap A \subset \{x_1,\ldots, x_r\} + W$. Hence, 
$$W_n \subset W + W(x_1) + \ldots + W(x_2) $$
which means that $W_n$ is finite dimensional, the image of $\langle \Lambda_{red} \rangle$ in $\GL(W_n)$ is Zariski-dense in a reductive subgroup and $W_n$ contains $W$ as a sub-representation of $\langle \Lambda_{red} \rangle$. By Weyl's complete reducibility, there is therefore a sub-$\langle \Lambda_{red} \rangle$-representation $W^{\perp}$ such that $W \oplus W^{\perp} = W_n$. Let $\pi_{\perp}: W_{n} \rightarrow W^{\perp}$ denote the projection to $ W^{\perp}$ parallel to $W$. Then $\pi_{\perp}$ commutes with the natural action of $\langle \Lambda_{red} \rangle$. In addition, $$\pi_{\perp}(\Lambda^n \cap A)\subset F_n:=\{\pi_{\perp}(x_1),\ldots, \pi_{\perp}(x_2)\}.$$ Let $F_{n-2} \subset F_n$ denote $\pi_{\perp}(\Lambda^{n-2} \cap A)$. We have 
$$\Lambda_{red} \cdot (F_{n-2}) = \pi_{\perp}(\Lambda_{red}\cdot(\Lambda^{n-2} \cap A)) \subset \pi_{\perp}(\Lambda^n \cap A) \subset F_n.$$
Since $F_n$ is finite, for any $x \in F_{n-2}$,  $\Lambda_x:=\Lambda_{red}^2 \cap \Stab(x)$ is commensurable with $\Lambda_{red}$ (Corollary \ref{Corollary: Intersection with stabiliser}). Since $\Lambda_x$ is an approximate lattice, the Borel density theorem for approximate lattices (\cite[\S A.11]{hrushovski2020beyond}) implies that $x$ is stabilised by $L$. But $L$ has no non-trivial invariant vectors, so $x=0$. As a conclusion, $\Lambda^{n-2} \cap A \subset W$ for all $n \geq 0$.

Let $W'$ denote the smallest $\mathbb{Q}$-vector subspace stable under the action of $\langle \Lambda_{red} \rangle$ containing $a(\langle \Lambda_{red} \rangle)$. Although $W'$ is \emph{a priori} not finite-dimensional, it is an ascending union of finite dimensional sub-representations of $\langle \Lambda_{red} \rangle$. Proceeding by induction and with the same arguments as above, we find that there is a projection $\pi_{\perp}: W' \rightarrow W$ that commutes with the action of $\langle \Lambda_{red} \rangle$.  Write $a_2= \pi_{\perp} \circ a$ and $a_1= a - a_2$. For all $\gamma \in \langle \Lambda_{red} \rangle$, we have $a_1(\gamma) \in \ker \pi_{\perp}$ and $a_2(\gamma) \in W$. Thus, (\ref{Eq: quasi-cocycle equation}) implies that for all $\lambda_1 \in \Lambda_{red}^{k_1}, \lambda_2 \in \Lambda_{red}^{k_2}$ 
\begin{equation*}
a_1(\lambda_1\lambda_2) - \lambda_1 \cdot a_1(\lambda_2) - a_1(\lambda_1) \in \Lambda^{2(k_1 + k_2)} \cap \ker \pi_{\perp}
\end{equation*} 
and 
\begin{equation*}
a_2(\lambda_1\lambda_2) - \lambda_1 \cdot a_2(\lambda_2) - a_2(\lambda_1) \in \Lambda^{2(k_1 + k_2)} \cap W.
\end{equation*} 
Since $\langle \Lambda \rangle \cap A \subset W$ by the preceding paragraph,  we have that $a_1$ is a $1$-cocycle and $a_2$ satisfies (\ref{Eq: quasi-cocycle equation}). 

Since $\Lambda^{100} \cap W$ is laminar, we can find by Corollary \ref{Corollary: Bounded approximate subgroups of of automorphisms are relatively compact} an approximate subgroup $\Xi \subset \langle \Lambda \rangle \cap W$ commensurable with $\Lambda^2 \cap W$ and containing $\Lambda^{100} \cap W$, and a subgroup $N \subset \langle \Lambda \rangle \cap W$ containing $\langle \Xi \rangle$ as a finite index subgroup such that $\Lambda_{red}$ is covered by finitely many translates of the normaliser $N_{\langle \Lambda \rangle}(N)$. Define $\Lambda_{red}':= \Lambda_{red}^2 \cap N_{\langle \Lambda \rangle}(N)$. As $\Lambda^2 \cap W$ spans $W$, we can furthermore find an integer $m_0 >0$ such that $N_0:=m_0N \subset \langle\Lambda^2 \cap W\rangle$. Suppose that $\Lambda_{red}'$ is not contained in a lattice. By Theorem \ref{THEOREM: EFFICIENT MESSY GENERATOIN}, it is efficiently generated by some finite subset $F \subset \langle\Lambda_{red}' \rangle$. For every $\lambda \in \Lambda_{red}'$ there are $\lambda_0=e,\ldots, \lambda_m=\lambda \in \Lambda_{red}'$ such that for all $l \in \{0, \ldots, m-1\}$ there is $i \leq l$ satisfying $\lambda_{l+1}\lambda_i^{-1} \in F\lambda_l\lambda_i^{-1} F^{-1}$. Since $N$ spans $W$, there is $m_1 \geq m_0$ such that $a_2(F) \subset \frac{1}{m_1}N$. So (\ref{Eq: quasi-cocycle equation}) combined with the fact that $\Lambda_{red}'$ normalises $N$ implies 
 $$a_2(\lambda_{l+1}) \in F \cdot a_2(\lambda_l) + \frac{1}{m_1}N.$$
 Since $F$ normalises $N$ as well, we find by induction that $a_2(\lambda_{l+1}) \in \frac{1}{m_1}N$. Now, $\Lambda_{red}$ is covered by finitely many translates of $\Lambda_{red}'$ and $N$ spans $W$, so we can find an integer $m_2 \geq m_1$ such that $a_2(\Lambda_{red}) \subset \frac{1}{m_2}N$. 
 
 It remains to deal with the case $\Lambda_{red}'$ contained in a lattice. Then we proceed as above invoking finite generation of said lattice (\cite[\S IX.3]{MR1090825}) rather than efficient generation of $\Lambda_{red}'$. 
\end{proof}

\begin{remark}\label{Remark: Decomp quasi-cocycles}
As the choice of $W^{\perp}$ in the proof of Proposition \ref{Proposition: quasi-cocycle takes values in Lambda} is not usually unique, the decomposition $a$ into $a_1 + a_2$ is far from being unique. However, if $b_1 + b_2$ denotes a second decomposition of $a$ with $b_1$ a $1$-cocycle and $b_2$ taking values in $W$, then $\delta:= a_1 - b_1 = b_2 - a_2$. The first equality yields that $\delta$ is a $1$-cocycle, while the second equality shows that it takes values in $W$. So all decompositions are equal modulo $Z^1(\langle \Lambda_{red} \rangle; W)$.  
\end{remark}

 We can construct the relevant good model alluded to above in a key case.

\begin{lemma}\label{Proposition: the right model to look quasi-cocycle through, no invariant vector case}
Let $G$ be an $S$-adic linear group. Suppose that $G = S \ltimes U$ where $S$ is semi-simple, $U$ is unipotent abelian and $S$ fixes no non-trivial element of $U$. Let $\Lambda \subset G$ be an approximate lattice such that $U \cap \Lambda^2$ is an approximate lattice in $U$ and such that the projection of $\Lambda$ to $S$ is a lattice $\Gamma$. For all $\gamma\in \Gamma$ let $a(\gamma) \in U$ be chosen such that $(\gamma, a(\gamma)) \in \Lambda$. Then there is a subgroup $\Delta \subset U$ and a group homomorphism $\tau: \Delta \rightarrow \mathbb{R}^d$ with relatively dense image such that: 
\begin{enumerate}
\item $\Delta$ is normalised by $\Lambda$. In particular, $\Gamma$ acts on $\Delta$;
\item $\Gamma$ acts by uniformly bounded linear operators on $\mathbb{R}^d$ and $\tau$ is $\Gamma$-equivariant; 
\item $\tau$ is a good model of an approximate subgroup $\Xi \subset \Delta$ commensurable with $\Lambda^2 \cap U$;
\item there is a $1$-cocycle $b_{phy}: \Gamma \rightarrow U$ such that $a-b_{phy}$ takes values in $\Delta$.
\end{enumerate}
\end{lemma}

\begin{proof}
By Proposition \ref{Proposition: quasi-cocycle takes values in Lambda} there is a $\mathbb{Q}$-vector subspace $W \subset U$ invariant under $\Gamma$ such that $W$ contains $\langle \Lambda \rangle \cap U$ and a $1$-cocycle $b_{phy}$ such that $q:=b_{phy} - a$ takes values in $W$. The proof of Proposition \ref{Proposition: quasi-cocycle takes values in Lambda} gives moreover a $\mathbb{Q}$-linear map $\tau_0: W \rightarrow \prod_{v \in S} \mathbb{Q}_v^{n_v}$, a subgroup $\Delta$ containing $q(\Gamma)$ and $\Lambda^4 \cap W$, and normalised by $\Gamma$, such that $(\tau_0)_{|\Delta}$ is a good model of an approximate subgroup commensurable with $\Lambda^2 \cap A$. Finally, there is an action of $\Gamma$ by uniformly bounded linear operators on $\prod_{v \in S} \mathbb{Q}_v^{n_v}$ making $\tau_0$ $\Gamma$-equivariant. We will show that the $\mathbb{Q}_v^{n_v}$ component $q_v$ of $\tau_0 \circ q$ is bounded if $v$ is non-Archimedean. Let $X \subset \Gamma$ be a finite generating subset of $\Gamma$ \cite[\S IX.3]{MR1090825}. Since $\Gamma$ acts by uniformly bounded operators, there is a compact open subgroup $O$ of $ \mathbb{Q}_v^{n_v}$ that is stable set-wise under $\Gamma$ and that contains both $q_v(X)$ and the projection of $\tau_0( \Lambda^4 \cap W)$. Then the quasi-cocycle relation (\ref{Eq: quasi-cocycle equation}) satisfied by $q$ implies that $q_v(\Gamma) \subset O$. So there is a compact open subgroup $O$ of $ \prod_{v\in S \setminus \{\infty\}} \mathbb{Q}_v^{n_v}$ such that $\tau_0 (\Delta) \subset \mathbb{R}^{n_{\infty}} \times O$. Hence, $\tau$ defined as the composition of $\tau_0$ and the projection $\mathbb{R}^{n_{\infty}} \times O\rightarrow \mathbb{R}^{n_{\infty}}$ is as required. 
\end{proof}

\subsection{Boundedness of quasi-cocycles and splitting}
 With the notation of \S \ref{Subsection: Representation-theoretic sum-product phenomenon}. Let $a$ be as above and let $a_1$ and $a_2$ be as in Proposition \ref{Proposition: quasi-cocycle takes values in Lambda}. We want to prove that in many situations both $a_1$ and $a_2$ `vanish'. Throughout this paper, the $1$-cocycle $a_1$ will be easily dealt with invoking well-established cohomology results for higher-rank lattices due to Margulis. Our main focus is therefore to study $a_2$. 

We start by rephrasing the `vanishing' of $a_2$ in terms of quasi-cocycles with target a finite-dimensional vector space. We have that $a_2$ takes values in the $\mathbb{Q}$-span $W$ of $ \Lambda_r $ and $a_2\left( \Lambda_{red}\right) \subset \frac{1}{m}\langle \Lambda_r \rangle$. For all $v \in S$, we have moreover linear maps $f_v: W \rightarrow V_v$ towards a $\mathbb{Q}_v$-vector space equipped with an action of $\langle \Lambda_{red} \rangle$, denoted by $\cdot$, such that $f_v(\lambda\cdot w) = \lambda \cdot f_v(w)$ and such that the diagonal map $f = \prod f_v$ restricted to $\langle \Lambda_r \rangle$ is a good model of $\Lambda_r$. For all $v \in S$, the set $V_v$ corresponds to the $\mathbb{Q}_v$-factor of the space $B$ from Proposition \ref{Proposition: Meyer's theorem in S-adic vector spaces}. Since $f_v(\Lambda_r)$ is bounded in $V_v$,  \eqref{Eq: quasi-cocycle equation} implies that $q_v:=f_v \circ a_2$ is a quasi-cocycle with respect to any norm on $V_v$. In addition:

\begin{proposition}\label{Proposition: Boundedness implies splitting}
Let $a, a_1, a_2$ and $q_v$ be as above. If $q_v(\Lambda_{red})$ is bounded for all $v \in S$, then $a_2(\Lambda_{red})$ is covered by finitely many translates of $\Lambda_r$. In this case, the subset 
 $$\Xi:=\{(\lambda, a_1(\lambda)) \in G: \lambda \in \Lambda_{red}\}$$
 is an approximate subgroup covered by finitely many translates of $\Lambda$. 
 
 If, moreover, $a_1$ is a $1$-co-boundary, then there is a reductive Levi subgroup $L' \subset G$ such that $\Lambda^2 \cap L'$ is an approximate lattice in $L'$. 
\end{proposition}

\begin{proof}
Since $q_v$ is bounded for all $v$, $a_2(\Lambda_{red})$ is covered by finitely many translates of $\Lambda_r$. Let $F$ be a finite subset such that $a_2(\Lambda_{red}) \subset \Lambda_rF$. Then for all $\lambda \in \Lambda_{red}$ we have 
$$(\lambda, a_1(\lambda)\Lambda_rF) \cap \Lambda \neq \emptyset.$$
Hence, 
$$\Xi \subset F^{-1}\Lambda^2.$$
Moreover, $\Xi$ is an approximate subgroup since it is a homomorphic image of $\Lambda_{red}$. As in the previous section, identify $L$ with the subgroup $L \times \{e\}$ of $G=L \ltimes A$. If $a_1$ is a co-boundary, then there is $u \in A$ such that for all $\gamma \in \langle \Lambda_{red} \rangle$ we have $a_1(\gamma)=\gamma \cdot u - u$ i.e. 
$$(e,u^{-1})(\gamma, e)(e,u) = (\gamma, a_1(\gamma)).$$
Set $L' := (e,u^{-1}) L (e,u)$.  So, as in Lemma \ref{Lemma: Small values of quasi-cocycle imply splitting}, we have for all $\lambda_{red} \in \Lambda_{red}$, $$(\lambda_{red}, e)^{(e,u)} = (\lambda_{red}, a_1(\lambda_{red})) \in (e,-a_2(\Lambda_{red}))\Lambda \subset F^{-1}\Lambda^2.$$ By Lemma \ref{Lemma: Intersection of commensurable sets}, $L' \cap \Lambda^2$ is an approximate lattice commensurable with $(\Lambda_{red}, e)^{(e,u)}$.
\end{proof}

Note that the scope of Proposition \ref{Proposition: Boundedness implies splitting} is wider than that of Lemma \ref{Lemma: Small values of quasi-cocycle imply splitting}. The difference can be regarded as accounting for the possibility of choosing the `wrong' Levi subgroup at first. 

\subsection{Commutativity and vanishing}
We prove now a first result implying the vanishing of the quasi-cocycle. This will illustrate our approach in Subsection \ref{Subsection: a criterion for boundedness of quasi-cocycles}.  We use commutation relations to show the boundedness of the quasi-cocycle adapting a well-established strategy \cite[Lemma 5.1]{BurgerMonod}.

\begin{lemma}\label{Lemma: Commutativity implies boundedness of quasi-cocycles}
Suppose that $L=G_1G_2$ is an almost direct product of two closed subgroups and that there are $\Lambda_1 \subset \Lambda_{red}^2 \cap G_1$ and $\Lambda_2 \subset \Lambda_{red}^2 \cap G_2$ approximate lattices such that $\Lambda_{red}$ is commensurable with $\Lambda_1\Lambda_2$. Suppose that $\langle \Lambda_1 \rangle$ and $\langle \Lambda_2 \rangle$ (equivalently, $G_1$ and $G_2$) act each without non-trivial invariant vectors on $A$. If $q_v$ is as above, then $q_v(\Lambda_{red})$ is bounded. Moreover, there is a reductive Levi subgroup $L' \subset G$ such that $L' \cap \Lambda^2$ is an approximate lattice in $L'$.
\end{lemma}

\begin{proof}
Let $|| \cdot ||_v$ denote a norm on $V_v$. For all $\lambda_1 \in \Lambda_1, \lambda_2 \in \Lambda_2$,  \eqref{Eq: quasi-cocycle equation} yields
$$ q_v(\lambda_1\lambda_2) = q_v(\lambda_1) + \lambda_1 \cdot q_v(\lambda_2) + \delta(\lambda_1,\lambda_2) $$
and 
$$ q_v(\lambda_1\lambda_2) = q_v(\lambda_2) + \lambda_2\cdot q_v(\lambda_1) + \delta'(\lambda_1,\lambda_2)$$
where  $||\delta(\lambda_1,\lambda_2)||_v,||\delta'(\lambda_1,\lambda_2) ||_v\leq C < \infty$ uniformly in $\lambda_1, \lambda_2$. 

Therefore, 
$$||(\lambda_2 - \id) \cdot q_v(\lambda_1)||_v \leq  ||\left(\lambda_1 -\id\right)\cdot (q_v(\lambda_2))||_v + 2C. $$
But, for a fixed $\lambda_2$, the right-hand-side is uniformly bounded by some constant $C(\lambda_2) > 0$ as $\lambda_1$ runs through $\Lambda_1$. So there is a constant $C'(\lambda_2)>0$ such that $$\inf_{x \in \ker (\lambda_2 - \id)} ||q_v(\lambda_1) - x||_v \leq C'(\lambda_2)$$ for all $\lambda_1 \in \Lambda_1$. We claim that $\langle \Lambda_2 \rangle$ acts without non-trivial invariant vector on $V_v$. Indeed, otherwise, there is a non-trivial invariant vector for the action of $\langle \Lambda_2 \rangle$ on $W \otimes \mathbb{A}_S$ (recall that we are using the notation from \S \ref{Subsection: Representation-theoretic sum-product phenomenon}). Since the action of $ \langle \Lambda_2 \rangle$ arises from an action on $W$, this means that there is a non-trivial invariant vector in $W$ and, in turn, in $A$ as $W$ projects injectively to $A$. A contradiction. Since $V_v$ is finite dimensional, there is therefore $X \subset \Lambda_2$ finite such that the family $X$ has no non-trivial common invariant vector in $V_v$ i.e. 
$$ \bigcap_{\lambda_2 \in X}\ker \left(\lambda_2 - \id \right)=\{0\}.$$ Then $$||q_v(\lambda_1)||_v \leq  C''(X) < \infty.$$
Moreover,  $a_1$ restricts to a $1$-cocycle on $\langle \Lambda_1 \rangle \times \langle \Lambda_2 \rangle$ with coefficients in $A$. Since the sets of invariant vectors $A^{\langle \Lambda_1 \rangle}$ and $A^{\langle \Lambda_2 \rangle}$ are reduced to $\{e\}$, the extension-restriction exact sequence implies that $(a_1)_{|\langle \Lambda_1 \rangle \times \langle \Lambda_2 \rangle}$ is a $1$-co-boundary. We conclude by invoking Proposition \ref{Proposition: Boundedness implies splitting}. 
\end{proof}

As an immediate corollary we have: 

\begin{corollary}\label{Corollary: Boundedness quasi-cocycle case centre}
With $q_v$ as above. Suppose that $\pi(\langle \Lambda_{red} \rangle \cap Z_L)$ has no non-trivial invariant vector where $Z_L$ denotes the center of $L$. Then $q_v(\Lambda_{red})$ is bounded. 
\end{corollary}

The commutation relation we want to utilise in the following is not as clear-cut as the one we assume in Lemma \ref{Lemma: Commutativity implies boundedness of quasi-cocycles}. We need to induce the cocycle considered to the ambient group to take advantage of it. This is achieved in \S \ref{Section: Induction for star-approximate lattices and Boundedness of quasi-cocycles}.

\subsection{Quasi-cocycles and Meyer sets}

The objective of this section is to provide a partial converse to Proposition \ref{Proposition: Boundedness implies splitting}. The results we establish are key parts of the proof of Theorem \ref{THEOREM: MAIN THEOREM}. We start with a general criterion, consequence of the uniqueness of quasi-models, see Corollary \ref{Corollary: Uniqueness of quasi-models, local version}. In this first part, we work with a set-up slightly more general than above. Recall that given a group $S$ acting by automorphisms on an abelian group $A$ and a $2$-cocycle $\alpha \in Z^1(S;A)$ we denote by $S \ltimes _{\alpha} A$ the group whose underlying subset is the product $S \times A$ and with multiplication defined by 
$$ (s_1,a_1)(s_2,a_2) = (s_1s_2,a_1 + s_1 \cdot a_2 + \alpha(s_1,s_2))$$
for all $s_1,s_2 \in S$ and $a_1,a_2 \in A$. As in \S \ref{Section: Cohomology for approximate subgroups}, when $S \ltimes_{\alpha} A$ is a locally compact group, we implicitly assume that the decomposition is obtained from a Borel locally bounded section of the projection $S \ltimes_{\alpha} A \rightarrow S$. 

\begin{proposition}\label{Proposition: Cocycle criterion laminarity}
Let $\Lambda$ be an approximate lattice in a locally compact group $S \ltimes_{\alpha} A$ with $A$ abelian and $\alpha$ taking finitely many values. Suppose that $\Lambda^2 \cap A$ is an approximate lattice in $A$. Write $\Lambda_s$ the projection of $\Lambda$ to $S$ and for every $m \geq 0$ and $\lambda_s \in \Lambda_s^m$ choose $a(\lambda_s) \in A$ such that $(\lambda_s, a(\lambda_s)) \in \Lambda^m$. Suppose that $\Lambda$ is a model set and write $\tau_s: \langle \Lambda_s \rangle \rightarrow H_s$ and $\tau_r: \langle \Lambda \rangle \cap A \rightarrow H_r$ good models of $\Lambda_s$ and $\Lambda^2 \cap A$ respectively (which exist by \cite[Lem. 3.3]{machado2019goodmodels}) with $H_r \simeq \mathbb{R}^n$, $H_r$ equipped with a continuous action of $H_s$ making $\tau_r$ equivariant for the induced actions of $\langle \Lambda_s \rangle$ and $\tau_r(\langle \Lambda \rangle \cap A)$ relatively dense. Then there is $b: \langle \Lambda_s \rangle \rightarrow H_r$ bounded on $\Lambda_s$ such that 
$$\gamma=(\gamma_s, \gamma_r) \in \langle \Lambda \rangle \mapsto \tau_r(\gamma_r - a(\gamma_s) ) - b(\gamma_s)$$ is a $1$-cocycle, where $\gamma_s$ and $\gamma_r$ denote the components of $\gamma$ in $S$ and $A$ respectively.
\end{proposition}

\begin{proof}
The gist of the proof is to use the data of $\tau_s, \tau_r$ and $a$ to build a quasi-model of sorts of $\Lambda$. We then conclude by appealing to the uniqueness of quasi-models (Corollary \ref{Corollary: Uniqueness of quasi-models, local version}). Define the map 
\begin{align*}
 f:\langle \Lambda \rangle & \longrightarrow H_s \ltimes H_r \\
 (\gamma_s, \gamma_r) & \longmapsto (\tau_s(\gamma_s), \tau_r(\gamma_r - a(\gamma_s)))
\end{align*}
where $\gamma_s$ and $\gamma_r$ are defined as in the statement. Then $f(\Lambda)$ is relatively compact and for every compact subset $C \subset H_s \ltimes H_r$, $f^{-1}(C)$ is covered by finitely many translates of $\Lambda$. Furthermore, if $g$ denotes the map giving the second coordinate of $f$, for all $\gamma_1, \gamma_2 \in \langle \Lambda \rangle$ we have $$dg(\gamma_1, \gamma_2) = \tau_r(\alpha(\gamma_1, \gamma_2)  - da(\gamma_{1,s}, \gamma_{2,s}))$$ and $\alpha(\gamma_1, \gamma_2)  - da(\gamma_{s,1}, \gamma_{s,2}) \in \langle \Lambda \rangle \cap A$. But $\alpha$ takes finitely many values and $da(\gamma_{s,1}, \gamma_{s,2})$ takes values in $\Lambda^4 \cap A$ as $\gamma_1, \gamma_2$ range though $\Lambda^*:=\Lambda (\langle \Lambda \rangle \cap A)$. So $dg(\Lambda^*, \Lambda^*)$ is bounded. Finally, $\tau_r(\langle \Lambda \rangle \cap A)$ is relatively dense in $H_r$. So $f$ and $\Lambda$ satisfy the conditions of Corollary \ref{Corollary: Uniqueness of quasi-models, local version}. Hence, we can find $b_0: \langle \Lambda \rangle \rightarrow H_r$, uniformly bounded over every subset commensurable with $\Lambda^*$, such that $g - b_0$ is a $1$-cocycle. Notice that the action of $\langle \Lambda \rangle \cap A$ is trivial on $A$, so for any $1$-cocycle $c$ and any $\gamma \in \langle \Lambda \rangle$, the map $\gamma_r \in \langle \Lambda \rangle \cap A \mapsto c(\gamma_r\gamma) - c(\gamma)$ is a group homomorphism. Similarly, the map $\gamma_r \in \langle \Lambda \rangle \cap A \mapsto  g(\gamma_r \gamma) - g(\gamma)$ is a group homomorphism. Hence, $\gamma_r \in \langle \Lambda \rangle \cap A \mapsto b_0(\gamma_r\gamma) - b_0(\gamma)$ is a group homomorphism and is uniformly bounded. In other words, $b_0$ factors through the projection of $\langle \Lambda \rangle$ to $S$.
\end{proof}

We can improve upon Proposition \ref{Proposition: Cocycle criterion laminarity} in certain situations. The key is to show that the good model $\tau_r$ can be extended in such a way that $\tau_r(\lambda_r) - \tau_r(a(\lambda_s))$ is well defined and equal to $\tau_r(\lambda_r - a(\lambda_s))$, while remaining equivariant with respect to the action of $\langle \Lambda_s \rangle$.

\begin{proposition}\label{Proposition: Equivalence laminarity/non-trivial cohomology class}
Let $G$ be an $S$-adic linear group. Suppose that $G = S \ltimes A$ where $S$ is semi-simple and $A$ is abelian. Let $\Lambda \subset A$ be an approximate lattice such that $A \cap \Lambda^2$ is an approximate lattice in $A$ and such that the projection of $\Lambda$ to $S$ is a lattice $\Gamma$. There are a subgroup $\Delta \subset A$ and a good model $\tau: \Delta \rightarrow L$ depending only on $\Lambda^2 \cap A$ and $\Gamma$, where $L\simeq \mathbb{R}^n$ is equipped with an action of $\langle \Lambda_s \rangle$ making $\tau$ equivariant, that satisfy the following.

 For all $\gamma\in \Gamma$, let $a(\gamma) \in A$ be chosen such that $(\gamma, a(\gamma)) \in \Lambda$. The following are equivalent: 
\begin{enumerate}
\item $\Lambda$ is laminar; 
\item there is a $1$-cocycle $b_{phy}: \Gamma \rightarrow A$ such that $a= b_{phy} + q$, $q$ takes values in $\Delta$ and $\tau \circ q$ is within bounded distance of a $1$-cocycle.  
\end{enumerate}
\end{proposition}

\begin{proof} First of all, we prove (2) $\Rightarrow$ (1).  Let $\rho: \Gamma \rightarrow \GL_n(\mathbb{R})$ be the map obtained by the action on the range of $\tau$.  Let $b_{int}$ be a $1$-cocycle such that $b_{int}(\gamma_s) - \tau \circ q(\gamma_s)$ is bounded as $\gamma_s$ ranges through $\Gamma$. Let $K$ denote $\overline{\rho(\Gamma)}$, then $K$ is a compact subgroup. For every $(\gamma_s, \gamma_r) \in \langle \Lambda \rangle$ we have $\gamma_r- b_{phy}(\gamma_s) \in \Delta$ by assumption. Define thus $f(\gamma_s,\gamma_r) = (\alpha(\gamma_s), \tau(\gamma_r - b_{phy}(\gamma_s)) - b_{int}(\gamma_s))$. Then $f: \langle \Lambda \rangle \rightarrow K \ltimes \mathbb{R}^d$ is a group homomorphism. Moreover, $f(e, \lambda) =(e, \tau(\lambda))$ so is bounded on $\lambda \in \Lambda^2 \cap A$. And for $\gamma_s \in \Gamma$, $f(\gamma_s, a(\gamma_s)) = (\alpha(\gamma_s), \tau \circ q(\gamma_s) -  b_{int}(\gamma_s))$ so is bounded independently of $\gamma_s$. Therefore, $f(\Lambda)$ is relatively compact. One readily sees that $f$ is a good model of $\Lambda$. 

Let us prove (1) $\Rightarrow$ (2). We first treat two cases. 

When the action of $S$ on $A$ is trivial: choose a good model $\tau_r': p(\langle \Lambda \rangle) \rightarrow H$ of $\Lambda \cap A$ where $H$ is some locally compact abelian group. By Lemma \ref{Proposition: the right good model in central case}, since $\Lambda_s$ is efficiently generated, $H$ is compactly generated. Upon quotienting out a compact subgroup of $H$ we can therefore assume $H=\mathbb{R}^m \times \mathbb{Z}^n$ and by expending the range we can assume $H=\mathbb{R}^{m+n}$. Now, $\tau_r$ restricted to $\langle \Lambda \rangle \cap A$ is also a good model. Set $\tau_s: \Gamma \rightarrow \{0\}$, and apply Lemma \ref{Proposition: Cocycle criterion laminarity} to $\Lambda$, $\tau_r$ and $\tau_s$. We have that there is $b:\Gamma \rightarrow H$ bounded such that  $$\gamma=(\gamma_s, \gamma_r) \in \langle \Lambda \rangle \mapsto \tau_r(\gamma_r -  a(\gamma_s)) - b(\gamma_s),$$
where here again $\gamma_s$ and $\gamma_r$ denote the $S$ and $A$ components respectively, is a $1$-cocycle. But 
$$ \tau_r(\gamma_r -  a(\gamma_s)) - b(\gamma_s) = \tau_r(\gamma_r) -  \tau_r \circ a(\gamma_s) - b(\gamma_s)$$
and $(\gamma_s, \gamma_r) \in \langle \Lambda \rangle \mapsto \tau_r(\gamma_r)$ is also a $1$-cocycle. So $(\gamma_s, \gamma_r) \in \langle \Lambda \rangle \mapsto \tau_r \circ a(\gamma_s) - b(\gamma_s)$ is a $1$-cocycle which settles the first case.

When the action on $A$ has no invariant vector. Let $\tau_r$ and $b_{phy}: \Gamma \rightarrow A$ denote the good model and the $1$-cocycle given by Lemma \ref{Proposition: the right model to look quasi-cocycle through, no invariant vector case}. What we are looking for is precisely Proposition \ref{Proposition: Cocycle criterion laminarity}. Indeed, according to Proposition \ref{Proposition: Cocycle criterion laminarity} the map $\psi: (\gamma_s, \gamma_r) \in \langle \Lambda \rangle \rightarrow \tau_r(\gamma_r - a(\gamma_s)) - b(\gamma_s)$ is a $1$-cocycle of $\langle \Lambda \rangle$ for some bounded map $b$. Since $(\gamma_s, \gamma_r) \in \langle \Lambda \rangle \rightarrow \tau_r(\gamma_r - b_{phy}(\gamma_s))$ is a $1$-cocycle of $\langle \Lambda \rangle$, $\tau_r(b_{phy}(\gamma_s) - a(\gamma_s)) - b(\gamma_s)$ is a $1$-cocycle of $\langle \Lambda \rangle$ and, hence, of $\langle \Lambda_s \rangle$.

In general, according to \cite[Prop. 4.17]{mac2023sadic} we have that $A= Z_G \times U$ where $U$ is unipotent without non-trivial invariant element and such that both $\Lambda^2 \cap Z_G$ and $\Lambda^2 \cap U$ are approximate lattices in $Z_G$ and $U$ respectively. Let $a = a_z + a_u$ be the decomposition of $a$ into the $Z_G$ component and the $U$ component. Consider the projections $\Lambda_u$ and $\Lambda_z$ of $\Lambda$ to $G/Z_G$ and $G/U$ respectively. Let $\Delta_z \subset Z_G$ and $\Delta_u \subset U$ be the subgroups stabilised by $\Gamma$ and $\tau_z: \Delta_z \rightarrow \mathbb{R}^m$ and $\tau_u: \Delta_u \rightarrow \mathbb{R}^n$ be the good models given by the first and second paragraph of this proof respectively. Consider the subgroup $\Delta:= \Delta_z \times \Delta_u$ and the map $\tau$ defined as the product map from $\Delta$ to $\mathbb{R}^m \times \mathbb{R}^n$. Since $\Lambda$ is laminar, $\Lambda_z$ and $\Lambda_u$ are laminar as well (\cite[Lem. 3.3]{machado2019goodmodels}). Hence, there is a $1$-cocycle $b_{phy}: \Gamma \rightarrow U$ such that $\tau_z \circ a_z$ and $\tau_u \circ \left(a_u - b_{phy} \right)$ are well-defined and within bounded distance of a $1$-cocycle in $\mathbb{R}^m$ and $\mathbb{R}^n$ respectively. 
\end{proof}

An elaboration on the above argument would certainly enable us to establish a similar result for more general semi-simple parts. However, since we will show that quasi-cocycles are ``trivial'' in all other cases we consider, we keep with the above assumptions for the sake of simplicity.  As a first application of this characterisation:

\begin{proof}[Proof of Proposition \ref{Proposition: A counter-example to a general Meyer-type theorem}.]
This is a direct from Proposition \ref{Proposition: Equivalence laminarity/non-trivial cohomology class}.

\end{proof}

\section{$L^p$-Induction and boundedness of quasi-cocycles}\label{Section: Induction for star-approximate lattices and Boundedness of quasi-cocycles}
In this section, we describe an induction process that allows us to obtain a quasi-cocycle defined on $G$ from a quasi-cocycle defined on an $\star$-approximate lattice. As in \cite{zbMATH05936993} and \cite{BurgerMonod}, the main idea is to define a well-suited (dynamical) cocycle on the invariant hull. We point the interested reader to \cite{bjorklund2020analytic} and \cite{machado2020apphigherrank} for related results on approximate lattices.

\subsection{Integrable fundamental domain and spectral gap}

Under a spectral gap assumption, we can ensure that $s$ satisfies an integrability condition. To do so, we adapt a strategy due to Maucourant concerned with building integrable fundamental domains of lattices (see \cite[Thm 5.3]{zbMATH07128510}).

\begin{lemma}\label{Lemma: Integrable Borel section}
 Let $\Gamma$ be a lattice in a product of locally compact $G \times H$ and let $d$ be a left-invariant proper distance on $G$ compatible with the topology. Suppose that $G \times H/\Gamma$ has spectral gap as a $G$-space and $W$ is a compact subset such that the projection of $W^2 \times W_0$ to $\Gamma$ is injective. Then there are $\alpha >0$ and a fundamental domain $W^2 \times W_0 \subset \mathcal{F} \subset G \times W_0$ such that 
$$ \int_{\mathcal{F}}e^{\alpha d(e,g)}d\mu_{G \times H}(g,h) < \infty. $$
\end{lemma}

\begin{proof}
The conclusion does not change if $d$ is replaced by a distance $d'$ such that $d'(\cdot,\cdot) \leq ad(\cdot,\cdot) + b$. Therefore, we may assume that $d$ is the word metric arising from a compact subset $K \subset G$ containing $e$ to be chosen as follows. Since $G \times H/\Gamma$ has spectral gap as a $G$-space, there are a generating compact subset $K \subset G$ and $c \in (0;1)$ such that for every Borel subset $B \subset G \times H/\Gamma$ with $\nu(B) \leq 1/2$ there is $k \in K$ such that 
\begin{equation}
\nu\left((kB) \Delta B\right) > c\nu(B). \label{Eq: Spectral gap}
\end{equation}

Since the projection of $\Gamma$ to $H$ is dense, the projection of $G \times W_0$ to $G \times H/\Gamma$ is surjective. Let $\mathcal{F}_1 :=W^2 \times W_0$. Define for all $n\geq 0$ the subset $$\mathcal{F}_n:=K^n \cdot\left( W \times W_0\right) \setminus \bigcup_{m < n} G^n\cdot \left(W \times W_0\right)\Gamma.$$
Since $K$ generates $G$, $\mathcal{F}:=\bigcup_n\mathcal{F}_n \subset G \times W_0$ is a fundamental domain of $\Gamma$ containing $W \times W_0$.  Following the computations as in the second paragraph of the proof of  \cite[Thm 5.3]{zbMATH07128510} we also have
$$\int_{\mathcal{F}}e^{\alpha d(e,g)}d\mu_{G \times H}(g,h) < \infty. $$
\end{proof}

To apply Lemma \ref{Lemma: Integrable Borel section} to our problem we invoke a consequence of property $(\tau)$ \cite{MR1953260}. It shows that Lemma \ref{Lemma: Integrable Borel section} is satisfied for model sets in $S$-adic semi-simple groups. We follow the strategy of \cite[Thm 1.12]{MR1719827}.

\begin{proposition}\label{proposition: Spectral gap}
Let $G,H$ be a product of $S$-adic algebraic semi-simple groups such that $G \times H$ has $S$-rank at least $2$. Let $\Gamma$ be an irreducible lattice and suppose that $G$ is non-compact. Then the $G$-action on $(G \times H)/ \Gamma$ has spectral gap.
\end{proposition}

\begin{proof}
Since $\Gamma$ is arithmetic, there is $L \rightarrow G \times H$ with finite kernel and finite index image such that $L$ is a product of groups of points of a simply connected $K$-group $\mathbb{H}$ over inequivalent completions of $K$ and $\Gamma$ is commensurable with $\Gamma'$ an arithmetic lattice in $L$ (\cite[Theorem 1 and IX.1.6.(i)]{MR1090825}). As a consequence of \cite{MR1953260}, we know that the action of any non-compact simple factor of $L$ acts with no almost invariant vectors on the space $L^2_0(L/\Gamma')$ of $L^2$ functions with $0$ mean. Therefore, the restriction of the quasi-regular representation $L^2_0(G \times H/\Gamma)$ to any non-compact factor of $G \times H$ has no almost invariant vector either (\cite[\S 3.1]{MR1719827}).
\end{proof}

\subsection{Borel cocycles and return times}\label{Subsection: Borel cocycles and return times}

From now on fix $G \times H$ a product of second countable locally compact groups and $\Gamma \subset G \times H$ a lattice.  Let $W_0 \subset H$ be  a neighbourhood of the identity to be determined later.  Let $\Lambda \subset G$ be an approximate lattice containing the model set given by $W_0$ such that $\langle \Lambda \rangle$ is equal to the projection of $\Gamma$.  Let  $W$ denote the ball of radius $\epsilon$ centred at $e$ for some small $\epsilon > 0$ to be chosen later. Let $\mathcal{F} \subset G \times H$ be a fundamental domain as in Lemma \ref{Lemma: Integrable Borel section} and $s: G \times H/\Gamma \rightarrow \mathcal{F}$ be the section.    We can build a cocycle $\alpha_s: \left(G \times H\right) \times G \times H/\Gamma \rightarrow G$ defined for all $g\in G, h \in H$ and $X \in G\times H/\Gamma$ by :
$$ \alpha_s(g,h,X)= s((g,h)X)^{-1}(g,h)s(X).$$
In particular, 
$$p_G(\alpha_s(g,e, X) )\in p_G((gX)^{-1}gX) = p_G(X^{-1}X \cap G \times W_0^2) \subset \Lambda^2.$$ These cocycles have now been well studied (\cite{bjorklund2020analytic, machado2020apphigherrank}). We will study them here in relation with certain transverse subsets.

\begin{proposition}[\S 4, \cite{BjorklundHartnickKarasik}]\label{Proposition: Transverse measure}
Let $\Lambda$,  $G$, $H$ be as above.  For all $g \in G$, define a cross-section by 
$$\mathcal{T}_{g,W_0}:= \{ X \in G \times H/\Gamma : g \in X \cap G \times W_0 \}.$$
 Then there is a finite Borel measure $\eta_g$ defined on $\mathcal{T}_g$ that satisfies the following: 
\begin{enumerate}
\item  for all $B \subset \mathcal{T}_{g,W_0}$ Borel and $h \in G$ such that $hB \subset \mathcal{T}_{g',W_0}$, 
$$\eta_{g'}(hB)=\eta_g(B);$$ 
\item for $W$ as above and $B \subset \bigcup_{g \in W} \mathcal{T}_{g,W_0}$ Borel
$$ \nu(B) = \int_W\eta_g( B \cap \mathcal{T}_{g,W_0})d\mu_G(g) = \int_{\mathcal{T}_{e,W_0}}\mu_G(\{g \in G: gX \in B\})d\eta_e(X).$$
\end{enumerate}
\end{proposition} 

The main result of this section is:  

\begin{lemma}\label{Lemma: Return times are large}
Let $g \in G$ and $B \subset \mathcal{T}_{g,W_0}$ be Borel. Suppose that $\eta_g(B) > 0$. Then the set of \emph{return times}, defined as
$$ \mathcal{R}(B):= \{h \in G: B \cap hB \neq \emptyset \}$$
is commensurable with $g\Lambda g^{-1}$ (and contained in $g\Lambda^2g^{-1}$). 
\end{lemma}

\begin{proof}
Choose a symmetric compact  neighbourhood of the identity $W' \subset G$ such that $g\Lambda^4g^{-1} \cap W'^2 =\{e\}$. For all $X \in G \times H/\Gamma$ set $T_{hit}(X,B):=\{g \in G: gX \in B\}$.  For all $U \subset G$ open,  the map $\Phi_U: X \mapsto \sum_{g \in T_{hit}(X,B)^{-1}}\mathbf{1}_U(g)$ is Borel and $G$-equivariant and its Haar integral is $\mu_G(U)$ for some normalisation of the Haar measure of $G$ independent of $U$ (e.g. \cite[\S 5.2]{bjorklund2016approximate}).  Moreover, $$T_{hit}(X,B)T_{hit}(X,B)^{-1} \subset \mathcal{R}(B) \subset g\Lambda^2g^{-1}.$$ We exploit a variant of Rusza's covering lemma: let $F \subset g\Lambda g^{-1}$ be a finite subset such that for all $X \in G \times H/\Gamma$,  the subsets $T_{hit}(X,B)^{-1}f$ for $f \in F$ are pairwise disjoint.  Since these subsets are disjoint we have
\begin{align*}
\sum_{f \in F} \mu_G(Uf^{-1}) \leq  \sum_{f \in F} \mu_{G \times H/\Gamma}(\Phi_{Uf^{-1}}) =\sum_{g \in T_{hit}(X,B)^{-1}F}\mathbf{1}_U(g) \leq \frac{\mu_G(UW')}{\mu_G(W')}
\end{align*}
where we have used disjointness to go from the second to the third term and $F^{-1}T_{hit}(X,B)T_{hit}(X,B)^{-1}F\subset g\Lambda^4g^{-1}$ to go from the third to the last term.  Since $G$ is unimodular, the first term is equal to $|F|\mu_G(U)$. So $F$ is finite. Therefore,  if $F \subset g\Lambda^2g^{-1}$ is chosen of maximal size satisfying the disjointness property, we find that 
$$g\Lambda g^{-1} \subset T_{hit}(X,B)T_{hit}(X,B)^{-1} F \subset \mathcal{R}(B)F$$
which concludes the proof. 
\end{proof}

\begin{remark}
Underpinning the proof of Lemma \ref{Lemma: Return times are large} are ideas recently developed in the study of approximate lattices \cite{BjorklundHartnickKarasik,  machado2020apphigherrank}.  Most of the results regarding the induction process could be recast in this more general framework.  We do not do that for he sake of brevity. 
\end{remark}
\subsection{$L^p$-Induction of quasi-cocycles}

With $G,H,\Gamma, W_0$, $\Lambda$, $s$ and $\alpha_s$ be as in \S \ref{Subsection: Borel cocycles and return times} above.  Fix a Banach space $B$ equipped with a norm $|| \cdot ||$ (here, $B$ is a Banach space over a local field $k$ that may be non-Archimedean).  Cocycles $\alpha_s$ can be used efficiently to induce quasi-cocycles.

\begin{definition}
Let $b: \langle \Lambda \rangle \rightarrow B$ be a quasi-cocycle of $\Lambda$ (i.e. an element of $QC(\Lambda;B)$). Define 
$$\hat{b}(g,h):= b\circ p_G\circ \alpha_s(g^{-1},h^{-1},\cdot)^{-1}.$$ 
\end{definition}

Then $\hat{b}(g,h)$ is measurable, takes values in $B$ and satisfies a quasi-cocycle-like behaviour. For all $g_1, g_2\in G$, $h_1, h_2 \in H$ and all $X \in G \times H/\Gamma$
\begin{equation}
||\hat{b}(g_1g_2,h_1h_2)(X) - \hat{b}(g_1,h_1)(X) - \pi(p_G\circ \alpha_s(g_1^{-1},h_1^{-1},X))^{-1}\cdot \hat{b}(g_2,h_2)((g_1,h_1)^{-1}X)|| \leq C \label{Eq: quasi-cocycle equation for induction}
\end{equation}
where  $C$ is a constant independent of $(g_1,h_1)$ and $(g_2,h_2)$.  The values of $\hat{b}$ and the values of $b$ are moreover closely related.

\begin{proposition}\label{Proposition: Boundedness of induced quasi-cocycle implies boundedness of quasi-cocycle}
Let $b: \langle \Lambda \rangle \rightarrow B$ be a quasi-cocycle of $\Lambda$. Suppose that there is some $p> 0$ such that $$\sup_{g \in G}||\hat{b}(g,e)||_p < \infty.$$
Then $b(\Lambda)$ is bounded. 
\end{proposition}

Proposition \ref{Proposition: Boundedness of induced quasi-cocycle implies boundedness of quasi-cocycle} generalises the usual induction process for lattices in bounded cohomology (see \cite[Prop. 13.2.6]{MR1840942}) as well as a result obtained in \cite{bjorklund2020analytic} under the assumption that $\Lambda$ is a uniform model set. Our strategy is close to Ozawa's \cite{zbMATH05936993}.

\begin{proof}
Define $l: G \times G \times H /\Gamma\rightarrow \mathbb{R}$ by 
$$l(g,X):=\min\{||\hat{b}(g,e)(X)||^p,1\}.$$
Since $L^p(G \times H /\Gamma,B) \subset L^q(G \times H /\Gamma,B)$ for $q \leq p$, we may assume that $p \leq 1$. Then for any $x,y \in \mathbb{R}_{\geq 0}$ we have $(x+y)^p \leq x^p + y^p$. So (\ref{Eq: quasi-cocycle equation for induction}) implies
$$l(gh,X) \leq C(l(g,X) + l(h,g^{-1}X))$$ 
where $C$ is a constant depending on the constant appearing in \eqref{Eq: quasi-cocycle equation for induction} and an upper bound of the operator norm of elements of $\pi(\Lambda)$. Besides, $l(g,\cdot)$ is bounded in $L^1(G \times H /\Gamma)$ uniformly in $g$ by assumption. According to \cite[Thm 8]{zbMATH05936993}, there is a map $f \in L^1(G \times H /\Gamma)$ such that for all $g \in G$ and $\nu$-almost all $X \in G \times H /\Gamma$ we have 
$$l(g,X) \leq f(X) + f(g^{-1}X).$$
By a Fubini-type argument there is a Borel subset $\Omega\subset G \times H /\Gamma$ of measure $1$ such that 
for all $X \in \Omega$ and almost all $g \in G$ we have 
\begin{equation}
l(g,X) \leq f(X) + f(g^{-1}X). \label{Eq: Ozawa}
\end{equation}
According to part (4) of Proposition \ref{Proposition: Transverse measure} there is $g \in W$ such that $\eta_g(\mathcal{T}_{g,W_0} \cap \Omega) = \eta_g(\mathcal{T}_{g,W_0})$. Define $$\mathcal{T}'_{g,C}:=\mathcal{T}_{g,W_0} \cap f^{-1}([0;C]) \cap \Omega.$$ For $C$ large enough, $\eta_g(\mathcal{T}'_{g,C}) > \frac{1}{2}\eta_g(\mathcal{T}_{g,W_0})$. Similarly, define $$\mathcal{T}''_{g,C}:=\{X \in \mathcal{T}_{g,W_0}: \mu_G(\{w \in W: f(wX) \leq C\})>0\}.$$ Again, for $C$ sufficiently large $\eta_g(\mathcal{T}''_{g,C}) > \frac{1}{2}\eta_g(\mathcal{T}_{g,W_0})$ (Part (4) of Proposition \ref{Proposition: Transverse measure}). So the set $\mathcal{T}_{g,C}:= \mathcal{T}'_{g,C} \cap \mathcal{T}''_{g,C}$ has positive $\eta_g$-measure for $C$ sufficiently large. Take $h$ in the set $\mathcal{R}(\mathcal{T}_{g,C})$ of return times and choose $X \in \mathcal{T}_{g,C}$ such that $hX \in \mathcal{T}_{g,C}$. Since $X \in \mathcal{T}'_{g,C}$ we have $f(X) \leq C$. Since $hX \in \mathcal{T}''_{g,C}$ there is $w \in W$ such that $f(wh X) \leq C$ and (\ref{Eq: Ozawa}) is satisfied for $X$ and $wh$. Thus, 
$$||b \left( p_G \circ \alpha_s(wh, e, X)^{-1}\right)||^p \leq l((wh)^{-1}, X) \leq f(X) + f(whX) \leq 2C.$$
But, for every $Y \in \mathcal{T}_{g,W_0}$ and $w \in W$ we have $p_G \circ s(wY)=wg$ (as $W^2 \times W_0 \subset \mathcal{F}$). So
$$p_G \circ \alpha_{s}(wh, e,X)=p_G\left(s(whX)^{-1}(wh,e)s(X)\right)= g^{-1}hg.$$ So $b(g^{-1}\mathcal{R}(\mathcal{T}_{g,C})g) \leq 2C$. By Lemma \ref{Lemma: Return times are large} the set $\mathcal{R}(\mathcal{T}_{g,C})$ is commensurable with $g\Lambda g^{-1}$ and contained in $g\Lambda^2g^{-1}$. Since $b$ is a quasi-cocycle, $b(\Lambda)$ is bounded.
\end{proof}

\subsection{A criterion for boundedness of quasi-cocycles}\label{Subsection: a criterion for boundedness of quasi-cocycles}
With $G,H,$ $\Gamma,W_0,$ $\Lambda, s,b$ and $B$ as in \S \ref{Subsection: Borel cocycles and return times}.  Choose a number field $K$, an absolutely almost simple $K$-group $\mathbf{G} \subset \SL_n$ and two disjoint finite subsets $S_1, S_2 \subset S_K$ whose union contains all the places $v$ such that $\mathbf{G}$ is $K_v$-anisotropic. Suppose from now on that $G = \mathbf{G}(\mathbb{A}_{K,S_1})$, $H=\mathbf{G}(\mathbb{A}_{K,S_2})$ and $\Gamma \subset \mathbf{G}(\mathbb{A}_{K,S_1 \cup S_2})$ is commensurable with $\mathbf{G}(\mathcal{O}_{K,S_1 \cup S_2})$.

\begin{proposition}\label{Proposition: Integrability of quasi-cocycle}
 With the notation as above. We have:
\begin{enumerate}
 \item if $\Lambda$ is uniform, then $\alpha_s(g,h,\cdot)$ takes finitely many values and so does $\hat{b}(g,h)$; 
 \item if $\Lambda$ is non-uniform, then there is $p > 0$ such that $\hat{b}(g,h)$
belongs to $L^p(G \times H/ \Gamma, B)$ for all $g\in G$ and $h \in H$; 
 \item if $\Lambda$ is non-uniform and $\pi$ is the trivial action, then  for all $p>0$, $g \in G$ and $h \in H$, $\hat{b}(g,h)$
belongs to $L^p(G \times H/ \Gamma, B)$. 
\end{enumerate} 
\end{proposition}

\begin{proof}
When $\Gamma$ co-compact, we see that $s$ takes values in a compact subset of $G \times H$. Thus, $\alpha_s(g,h)(\cdot)$ takes values in a compact subset of $\Gamma$ i.e. $\alpha_s(g,h)(\cdot)$ only takes finitely many values.  

Suppose now that $\Gamma$ is not co-compact. Fix $(g,h) \in G \times H$. Then $\alpha_s(g,h)$ takes values in $G \times W^{-1}hW \cap \Gamma$. So $p_{G} \circ \alpha_s$ takes values in an approximate subgroup $\Lambda'$ commensurable with $\Lambda$ and contained in $\langle \Lambda \rangle$. By Proposition \ref{Proposition: Bounds on quasi-cocycles from efficient generation}, for all $\lambda  \in \Lambda'$ we have,
$$ ||b(\lambda)|| \preceq ||\lambda||_{S_1}. $$
In particular, there are $c_1,c_2 > 0$ such that,
$$ ||b(\lambda)|| \leq c_1||\lambda||_{S_1}^{c_2}.$$
Since $\Lambda'$ is commensurable with $\Lambda$, the constant $c_2$ is seen not to depend on $h$, but the constant $c_1$ might. Thus, 
\begin{align*}
\int_{G \times H / \Gamma} ||\hat{b}(g,h)(y)||^pd\mu(y) \leq c_1^p \int_{G \times H / \Gamma} ||p_{G}(\alpha_s(g,h)(y))||_{S_1}^{pc_2} d\mu(y) .
\end{align*}
So it suffices to prove that 
$$\int_{G \times H / \Gamma} ||p_{G}(\alpha_s(g,h)(y))||_{S_1}^{pc_2} d\mu(y) < \infty.$$
But $||\cdot||_{S_1}$ is an operator norm, so, 
\begin{align*}
||p_{G}(\alpha_s(g,h)(y))||_{S_1} \leq ||p_{G}(s((g,h)y))^{-1}||_{S_1}||g||_{S_1}||p_{G}(s(y))||_{S_1}. 
\end{align*}
As $\mathbf{G} \subset \SL_n$, there are $c_3,c_4 \geq 1$ such that
$$||p_{G}(s((g,h)y))^{-1}||_{S_1} \leq c_3||p_{G}(s((g,h)y))||_{S_1}^{c_4}.$$
By the Cauchy--Schwartz inequality and the $G \times H$-invariance of $\mu$, it is thus enough to show that 
$$\int_{G \times H / \Gamma} ||p_{G}(s((g,h)y))||_{S_1}^{2pc_2c_4} d\mu(y) = \int_{G \times H / \Gamma} ||p_{G}(s(y))||_{S_1}^{2pc_2c_4} d\mu(y)< \infty$$
for any $(g,h) \in G \times H$. But the latter is finite as soon as $2pc_2c_4$ is sufficiently small, because $\log ||p_{G}( x^{-1}y )||_{S_1}$ is bounded above by a left-invariant distance on $G$ (Lemma \ref{Lemma: Integrable Borel section}). 
 
If moreover $\pi$ is trivial, then Corollary \ref{Corollary: Logarithmic bound quasi-morphisms} implies that  for all $\lambda  \in \Lambda$
$$ ||b(\lambda)|| =O\left(\log(||\lambda||_{S_1})\right). $$
Repeating the same computations as in the previous paragraph proves (3). 
\end{proof}

We are now ready to prove the main result of this section. 

\begin{proposition}\label{Proposition: Quasi-cocycle is bounded}
 With $\Lambda$, $b$ and $B$ as above. If $B$ has finite dimension, then $b(\Lambda)$ is a bounded subset of $B$. 
\end{proposition}

\begin{proof}
 First of all, according to Proposition \ref{Proposition: Boundedness of induced quasi-cocycle implies boundedness of quasi-cocycle} and the discussion preceding Proposition \ref{Proposition: Integrability of quasi-cocycle} it suffices to prove that $\hat{b}$ is bounded in $L^p(G \times H/\Gamma)$ for some $p > 0$. Moreover, if $\Lambda$ is an irreducible $\star$-approximate lattice in $G$ and $C$ denotes the maximal compact factor of $G$, then the projection of $\Lambda$ to $G/C$ is injective (in fact, injective over $\langle \Lambda \rangle$) and its image is a $\star$-approximate lattice. So we will assume that $G$ has no compact factor. By Margulis' superrigidity theorem \cite[Thm 2]{MR1090825} the map $\pi: \langle \Lambda \rangle \rightarrow \GL(B)$ comes from a continuous group homomorphism $\phi: G \times H \rightarrow \GL(B)$ with kernel containing $G$. Since $H$ is semi-simple, $B$ splits as a direct sum $\bigoplus B_i$ of irreducible representations of $H$.  Considering the projection $b_i$ of $b$ to each irreducible factor, we see that it suffices to consider the case $B$ irreducible. Also, if $H$ is compact, then $\Lambda^m$ is a lattice in $G$ for some integer $m>0$ and $G$ has higher-rank. In this case, the result is due to Burger and Monod \cite[Thm 1.1]{BurgerMonod}. So suppose from now on that $H$ is non-compact. Choose $p$ such that $\hat{b}(g,h)\in L^p(G \times H/ \Gamma, B)$ (Proposition \ref{Proposition: Integrability of quasi-cocycle}). For all $g \in G, h\in H$ and $y \in G\times H /\Gamma$ write
\begin{align*}
\hat{b}(g,h)(y)=\hat{b}(g,e)(y) +\pi(\alpha_s(g^{-1},e,y))^{-1}\hat{b}(e,h)(g^{-1}y) + \delta(g,h)(y);\\
\hat{b}(g,h)(y)=\hat{b}(e,h)(y) + \pi(\alpha_s(e,h^{-1},y))^{-1}\hat{b}(g,e)(h^{-1}y) + \delta'(g,h)(y).
\end{align*}
Since $\alpha_s(g,h)$ takes values in $\Gamma \cap \left(G \times W^{-1}hW\right)$, there is an integer $m(h) > 0$ - independent of $g$ - such that $p_{G} \circ \alpha_s(g,h)$ takes values in $\Lambda^{m(h)}$. Since $b \in QC(\Lambda; B)$ we can thus choose $||\delta(g,h)||, ||\delta'(g,h)|| \leq C(h)$ for some constant $C(h)$ independent of $g$ (but depending on $h$). 
Hence,
$$\int_{G \times H/\Gamma}||\pi(\alpha_s(e,h^{-1},y))^{-1}\hat{b}(g,e)(h^{-1}y) - \hat{b}(g,e)(y)||^pd\mu(y) < \delta(h)$$
where $\delta(h)$ is a constant depending on $h$ only. We will now study three cases: first of all, when $\pi(H)$ is unbounded; then when $\pi(H)$ is compact non-trivial; lastly when $\pi(H)$ is trivial.

Suppose that $\pi(H)$ is unbounded, assume as we may that $p \leq 1$ and let $T$ be a maximal split torus in $H$. Then $\pi(H)$ is a closed semi-simple subgroup of $\GL(B)$ and $\pi(T)$ is a maximal split torus in $\pi(H)$. The result follows from the following lemma: 

\begin{lemma}\label{Lemma: Spectral gap through norms}
Let $S \subset \SL_n(k)$ be a non-compact Zariski-closed semi-simple subgroup where $k$ is a characteristic $0$ local field. Let $K' \subset S$ be a maximal compact subgroup and $T \subset S$ a maximal split torus. Suppose that the action on $k^n$ is irreducible. Then there are finitely many cones $C_1, \ldots, C_m \subset k^n$ (i.e. subsets stable under multiplication by a scalar), elements $k_1, \ldots, k_m \in K'$, a constant $c > 0$ and a neighbourhood of the identity $W' \subset S$ such that:
\begin{enumerate}
\item $k^n = C_1 \cup \ldots \cup C_n$; 
\item For all $i \in \{1, \ldots, m\}$ there is non-trivial weight $\alpha$ of $k^n$ with respect to $T$ such that for all $v_i \in C_i \setminus \{0\}$,  $w \in W'$ and $t \in T$,
$$ ||tk_iw\cdot v_i|| \geq \frac{|\alpha(t)|}{c}||v_i||.$$
\end{enumerate}
\end{lemma}
 
 The proof of Lemma \ref{Lemma: Spectral gap through norms} can be found at the end of this section. Let us show how to conclude using it.   Let $C_1, \ldots, C_m$, $k_1, \ldots, k_m$, $c$ and $W'$ be as in Lemma \ref{Lemma: Spectral gap through norms}. Suppose from now on that $W \subset W'$ and that the operator norm of an element of $W$ is bounded by $2$. For every $i \in \{1, \ldots, r\}$ we can choose $t_i \in T$ such that $||t_ik_iw \cdot v||/||v|| > 4r$ for all $v \in C_i, w \in W$. 

Since $B= \bigcup_{i=1}^r C_i$, there is $i=i(g) \in \{1, \ldots, r\}$ such that 
$$ \int_{G \times H/\Gamma} \mathbf{1}_{\hat{b}(g,e)(y) \in C_i}||\hat{b}(g,e)(y)||^pd\mu(y) \geq r^{-1} \int_{G \times H/\Gamma}||\hat{b}(g,e)(y)||^pd\mu(y).$$
Since $\pi(\alpha_s(e,(t_ik_i)^{-1},y))^{-1} = \pi(s(y))^{-1}\pi(e,t_ik_i)\pi(s((e,t_ik_i)^{-1}y))$, $s$ takes values in $W$ and $\mu$ is left-invariant, we find
\begin{align*}
\int_{G \times H/\Gamma} ||\pi(\alpha_s(e,(t_ik_i)^{-1},y))^{-1}\hat{b}(g,e)((t_ik_i)^{-1}y)||^pd\mu(y) > 2 \int_{G \times H/\Gamma} ||\hat{b}(g,e)(y)||^pd\mu(y).
\end{align*}
But by sub-additivity of the map $x \mapsto x^p$ we find
$$\int_{G \times H/\Gamma}||\pi(\alpha_s(e,h^{-1},y))^{-1}\hat{b}(g,e)(h^{-1}y)||^p - ||\hat{b}(g,e)(y)||^pd\mu(y) < \delta(h)$$
for all $g \in G$, $h\in H$. Whence,
$$\int_{G \times H/\Gamma} ||\hat{b}(g,e)(y)||^pd\mu(y) \leq \sup_{1\leq i \leq r}\delta(t_ik_i).$$

\medbreak
Let us now consider the case $\pi(H)$ compact non-trivial. Then $\Gamma$ must be uniform according to Godement's criterion (Proposition \ref{Proposition: BHC}).  Although not entirely necessary, it is simpler to first deal directly with the case where $B$ is defined over a non-Archimedean local field in order to avoid integration of functions taking values in non-Archimedean fields. Since $\pi(H)$ is compact, there is a compact open subgroup $O$ of $B$ that is invariant setwise under $\pi(H)$. By taking a dilation of $O$, we can assume that $O$ contains the bounded subset $\{b(\lambda_1\lambda_2) - b(\lambda_1) - \lambda_1 \cdot b(\lambda_2) : \lambda_1, \lambda_2 \in \Lambda^2\}$ as well as $b(F)$ for any finite subset $F \subset \Lambda^2$ chosen in advance. Then $b(\Lambda) \subset O$ because $\Lambda$ is uniform and, hence, `coarsely connected', see \cite[Thm 3.4]{bjorklund2016approximate}. So suppose that $B$ is a Banach space over $\mathbb{R}$. There are two sub-cases: if $H$ is a compact subgroup, then $\Lambda$ is commensurable with a lattice and we have already dealt with this case. Suppose that $H$ is unbounded.  So we may take any $p \geq 0$ (Proposition \ref{Proposition: Integrability of quasi-cocycle}). By Proposition \ref{proposition: Spectral gap} the representation of $H$ on $L^2(G \times H /\Gamma)$ has spectral gap. According to \cite[Lemma 4]{MR1415573}, $L^2(G \times H /\Gamma, B)$ seen as a unitary representation of $H$ has spectral gap too. We claim moreover that $L^2(G \times H /\Gamma, B)^{H}=\{0\}$. Indeed, $H = L \times K$ where $L$ is the kernel of the restriction of $\pi$ to $H$. Since $H$ is non-compact whereas $\pi(H)$ is, $L$ is unbounded. But by the Howe--Moore property, the action of $L$ on $G \times H/\Gamma$ is ergodic. So $L^2(G \times H /\Gamma, B)^{H} \subset L^2(G \times H /\Gamma, B)^{L} $ which is equal to the set of constant functions. Now, $\pi$ has no non-trivial invariant vectors in $B$. So $L^2(G \times H /\Gamma, B)^{H}=\{0\}$ as required. Therefore, there are a compact subset $K \subset G$ and a constant $c > 0$ such that if $f \in L^2(G \times H /\Gamma, B)$ there is $k \in K$ with $||k \cdot f - f||_2 > c||f||_2$. Suppose that for all $w \in  W$ the operator norm of $\pi(w)-\id$ is bounded above by $c/4$. Then, as above, we have for all $k \in K$
\begin{equation}
\int_{G \times H/\Gamma}||\pi(\alpha_s(e,k,y))\hat{b}(g,e)(k^{-1}y) - \hat{b}(g,e)(y)||^2d\mu(y) < C_2 \label{Eq: proof boundedness end}
\end{equation}
where $C_2=C_2(K)$.
But for all $k \in K$, $g \in G$ and $y\in G \times H /\Gamma$ we have
\begin{align*}\pi(\alpha_s(e,k,y)) - \pi(k)= & (\left(\pi(s(k^{-1}y))-\id\right)\pi(k)\pi(s(y)) + \pi(k)\left((\pi(s(y))-\id)\right)).
\end{align*}
Since $s$ takes values in $W$ we deduce 
$$||\pi(\alpha_s(e,k,y))\hat{b}(g,e)(k^{-1}y) - \pi(k)\hat{b}(g,e)(k^{-1}y)|| \leq c/2||\hat{b}(g,e)(k^{-1}y)||.$$
Using the above in combination with (\ref{Eq: proof boundedness end}) and using the triangle inequality we find 
$$||\pi(k)\hat{b}(g,e)(k^{-1}\cdot) - \hat{b}(g,e)(\cdot)||_2 - \frac{c}{2}||\hat{b}(g,e)(\cdot)||_2 \leq \sqrt{C_2}.$$
So 
$$\frac{c}{2}||\hat{b}(g,e)(\cdot)||_2 \leq \sqrt{C_2}$$
according to the spectral gap inequality.

Suppose finally that $\pi(H)=\{e\}$.  When $B$ is defined over a non-Archimedean field, that $b(\Lambda)$ is bounded is a direct consequence of the triviality of the action combined with efficient generation of $\Lambda$. Suppose now that $B=\mathbb{R}$. Then $\hat{b}(g,h) \in L^2(G \times H /\Gamma, B)$. Write $\hat{b}_1(g,h)$ the orthogonal projection of $\hat{b}(g,h)$ to the set $L^2(G \times H /\Gamma, B)^{G \times H}$ of constant functions and $\hat{b}_2(g,h)$ its complement i.e. the orthogonal projection of $\hat{b}(g,h)$ to $L^2_0(G \times H /\Gamma, B)$ the space of $L^2$ functions with $0$ mean. Note that both $L^2(G \times H /\Gamma, B)^{G \times H}$ and $L^2_0(G \times H /\Gamma, B)$ are stable under the action of $H$. In addition, the action of $H$ is trivial on $L^2(G \times H /\Gamma, B)^{G \times H}$ and without invariant vectors on $L^2_0(G \times H /\Gamma, B)$. 

In this situation, the quasi-cocycle relation implies for all $g_1, g_2 \in G$
\begin{align*}
C  \geq  \int_{G \times H / \Gamma} ||\hat{b}(g_1g_2,e)(y)-\hat{b}(g_1,e)(y) - \hat{b}(g_2,e)(g_1^{-1}y)||d\mu(y). 
\end{align*}
for some constant independent of $g_1$ and $g_2$. Therefore, we have by the triangle inequality
\begin{align*}
C  \geq  ||\hat{b}_1(g_1g_2,e)(y)-\hat{b}_1(g_1,e)(y) - \hat{b}_1(g_2,e)(y)||. 
\end{align*}
So $\hat{b}_1(g,e)$ is uniformly bounded when $g$ runs through $G$ by \cite[Lemma 6.1]{BurgerMonod}. 

On the other hand, for all $g \in G, h \in H$, 
$$\int_{G \times H/\Gamma}||\pi(\alpha_s(e,h^{-1},y))^{-1}\hat{b}(g,e)(h^{-1}y) - \hat{b}(g,e)(y)||^2d\mu(y) < \delta(h)$$ 
becomes 
$$\int_{G \times H/\Gamma}||\hat{b}(g,e)(h^{-1}y) - \hat{b}(g,e)(y)||^2d\mu(y) < \delta(h).$$
Which can be rewritten
$$\int_{G \times H/\Gamma}||\hat{b}_2(g,e)(h^{-1}y) - \hat{b}_2(g,e)(y)||^2d\mu(y) < \delta(h).$$
Exploiting spectral gap of the action of $H$ on $L^2_0(G \times H /\Gamma, B)$ we conclude as in the previous case that $\hat{b}_2(g,e)$ is uniformly bounded when $g$ runs through $G$. So $\hat{b} = \hat{b}_1 + \hat{b}_2$ is bounded as well. 
\end{proof}


\begin{proof}[Proof of Lemma \ref{Lemma: Spectral gap through norms}.]
We will first show that for every $v \in k^n$ there is $k(v) \in K$ such that $k(v)\cdot v \notin \Fix(T)$. We argue by contradiction. Then $K \cdot v \subset \Fix(T)$, so $KTK\cdot v = K \cdot v$ is compact. But $KTK=S$ so $S \cdot v$ is compact. Moreover, the $S$-subrepresentation $V$ generated by $v$ is contained in $\Fix(T)$. Since $S$ acts irreducibly on $k^n$,  $V = k^n$ and $T$ is trivial. A contradiction. 

Write $k^{n}=\bigoplus_{\alpha \in \Sigma} E_{\alpha} \oplus \Fix(T)$ where $\Sigma$ is the set non-trivial weights. Since $B$ is finite dimensional, all norms are equivalent, and we may assume that $|| \cdot ||$ is a sum of norms on each of the direct summands of the above decomposition. Take $v \in k^n\setminus \{0\}$ and let $k(v)$ be as above. There is $\alpha \in \Sigma$ such that the projection $p_{\alpha}(k(v)\cdot v)$ of $k(v)\cdot v$ to $E_{\alpha}$ is non-trivial. Let $c_v$ denote $||p_{\alpha}(k(v)\cdot v)||/||v||$. By continuity, there is a neighbourhood of the identity $W_v \in S$ and a neighbourhood $U_v$ of $v \in k^n$ such that for all $w \in W_v$ and $v' \in U_v$, 
$$||p_{\alpha}(k(v)w\cdot v')||/||v'|| \geq c_v/2.$$
Then if $t\in T$, 
\begin{align*}
||tk(v)w\cdot v'|| & \geq ||p_{\alpha}(tk(v)w\cdot v')|| \\
				  & = |\alpha(t)| ||p_{\alpha}(k(v)w\cdot v')|| \\
				  &\geq \frac{2|\alpha(t)|}{c_v} ||v'||.
\end{align*}
So set $C_v = kU_v$. By compactness of $\mathbf{P}^1(k^n)$ there are $v_1, \ldots, v_m$ such that 
$$C_{v_1} \cup \ldots \cup C_{v_m} = k^n.$$ Then $C_i:=C_{v_i}$, $W :=\bigcap_{i=1}^n W_{v_i}$ and $c := \min_{1 \leq i \leq n} c_{v_i}/2$ work. 
\end{proof}

\section{Structure of approximate lattices in linear groups}\label{Section: Proof of the main theorems}
We are now ready to conclude the proof of our main results (Theorem \ref{THEOREM: FIRST THEOREM} and Theorem \ref{THEOREM: MAIN THEOREM}). In this section we will primarily be interested in the following setup:

\begin{assumption}\label{Assumption: final proof}
Let $\Lambda$ be an approximate lattice in an $S$-adic linear group $G$. We assume that $G$ is Zariski-connected and that $\langle \Lambda \rangle$ is Zariski-dense in $G$. We denote by $\Rad(G)$ its solvable radical, by $G_s$ the quotient $G/\Rad(G)$, by $\Lambda_s$ the projection of $\Lambda$ to $G_s$ and by $\Lambda_r$ the intersection $\Lambda^2 \cap \Rad(G)$.
\end{assumption}

\subsection{Reduction to an abelian radical}

We perform in this section the first and main reduction. 

\begin{lemma}\label{Lemma: Reduction to abelian radical}
Let $\Lambda, G$ satisfy Assumption \ref{Assumption: final proof}.  Let $N$ denote the nilpotent radical of $G$ and $Z_N$ its centre. There is a reductive Levi subgroup $R \subset G$ such that $\Lambda^2 \cap L$ is an approximate lattice in $L:=RZ_N$.  Moreover,  the projection of $\Lambda^2 \cap L$ to $L/[L,L]Z_N$ is uniformly discrete.
\end{lemma}

\begin{proof}
According to Theorem \ref{Theorem: Radical is hereditary} and Lemma \ref{Lemma: Intersection with centre}, $\Lambda^2 \cap Z_N$ is an approximate lattice. By Proposition \ref{Proposition: Intersection and projections approximate lattices w/ closed subgroups},  the projection $\Lambda'$ of $\Lambda$ to $G/Z_N$ is an approximate lattice as well. By Proposition \ref{Proposition: Arithmeticity from action on subgroups}, $\Lambda'$ is in fact a generalized arithmetic approximate lattice. So there is a reductive Levi subgroup $S'$ of $G/Z_N$ such that $\Lambda'^2 \cap S'$ is an approximate lattice in $S'$. Denote by $L$ the inverse image of $S'$ in $G$.  Then $L \cap \Lambda^2$ is an approximate lattice in $L$ by Proposition \ref{Proposition: Intersection and projections approximate lattices w/ closed subgroups}. Finally, $L$ is of the desired form. The `moreover' part of the statement also follows from arithmeticity of $\Lambda'$. 
\end{proof}

The next result ensures that the assumptions we will make on $\langle \Lambda_s \rangle$ are passed down to $\Lambda^2 \cap L$. 

\begin{lemma}\label{Lemma: Reduction to abelian radical, higher-rank part}
With $\Lambda, G, L$ as in Lemma \ref{Lemma: Reduction to abelian radical}.  There is an approximate lattice $\Lambda' \subset L$ commensurable with $\Lambda^2 \cap L$ and whose image in $G/\Rad(G)$ generates a finite index subgroup of $\langle \Lambda_s \rangle$.
\end{lemma}

\begin{proof}
Choose $\gamma_s \in \langle \Lambda_s \rangle$ and take $\gamma \in \langle \Lambda \rangle$ any representative of it.  We have that $\gamma\Lambda^2\gamma^{-1} \cap L^{\gamma}$ is an approximate lattice in $L^{\gamma}$. But $\gamma \Lambda^2 \gamma^{-1}$ is commensurable with $\Lambda$ so $\Lambda^2 \cap L^{\gamma}$ is an approximate lattice in $L^{\gamma}$ (Lemma \ref{Lemma: Intersection of approximate subgroups}).  According to \cite[Cor. 4.18]{mac2023sadic}, there is therefore $u \in \Comm_G(\Lambda) \cap \Rad(G)$ such that $L^{u}=L^{\gamma}$.  Hence, one deduces that $\gamma u^{-1}$ commensurates $\Lambda^2 \cap L$.  Define $$\Gamma:=\Comm_L(\Lambda^2 \cap L) \cap \pi^{-1}\left( \langle \Lambda_s \rangle\right). $$ Then $\Lambda^2 \cap L \subset \Gamma$ and $\left(\Lambda^2 \cap L\right)^{\gamma u^{-1}} \subset \Gamma$.  So $\pi(\Gamma)$ contains a subgroup $\Gamma_s$ normalised by $\gamma_s$ and containing an approximate lattice commensurable with $\Lambda_s$.  Similarly, $\Gamma_s$ contains a subgroup normalised by  $\langle \Lambda_s \rangle$. But $\langle \Lambda_s \rangle$ is an almost direct product of higher-rank lattices and rank one lattices.  Since $\Gamma_s$ contains an approximate lattice commensurable with $\Lambda_s$ it contains a finite index subgroup of all the rank one factors. Furthermore, it must intersect every higher-rank factor in an infinite subset. Since it is normal, it contains a finite index subgroup of all higher-rank factors by Margulis normal subgroup theorem \cite[Thm. 4]{MR1090825}. Take finally $F \subset \Gamma$ finite such that $\pi(F)$ generates $\Gamma_s$ (such a set exists by \cite[\S IX.3]{MR1090825}).  Then $\left(\Lambda^2 \cap L\right) \cup F \cup F^{-1}$ is the approximate lattice we are looking for. 
\end{proof}
The overall strategy is to proceed by induction on the dimension of the radical of the ambient group. We start by treating the case of an abelian radical.

\subsection{Reduction to $Z_{G,1}$}
Recall that we defined $Z_{G,1}$ in the introduction as 
$$Z_{G,1} = \langle g \in G: \exists N \in \mathcal{N}, N \subset \Stab(g) \rangle$$
where $\mathcal{N}$ denotes the set of Zariski-closed normal subgroup of $\Rad(G)$ such that $G/N$ has $S$-rank one and the projection $\Lambda$ to $G/N$ is contained in a lattice. We can strengthen Lemma \ref{Lemma: Reduction to abelian radical}. 

\begin{lemma}\label{Lemma: End of proof}
Let $\Lambda$ be an approximate lattice in an $S$-adic linear algebraic group $G$. There is a Zariski-closed subgroup $L \subset G$ such that: 
\begin{enumerate}
\item $L \cap \Lambda^2$ is an approximate lattice in $L$; 
\item $L$ contains a Levi subgroup and $L \cap \Rad(G) = Z_{G,1}$. 
\end{enumerate}
\end{lemma}

\begin{proof}

According to Lemmas \ref{Lemma: Reduction to abelian radical} and \ref{Lemma: Reduction to abelian radical, higher-rank part},  we  may assume that $G = G_{red} \ltimes Z_N$ with $G_{red}$ reductive. Let $\Lambda_{red}$ denote the projection of $\Lambda$ to $G_{red}$. By \cite[Prop. 4.17]{mac2023sadic}, $Z_N = Z_{G,1} \oplus \bigoplus_{\Delta} U^{\Delta}$ where $\Delta$ is a normal subgroup of $\langle \Lambda_s \rangle$ such that $\langle \Lambda_{red} \rangle/\Delta$ is not a rank one lattice. Moreover, all $\Delta$'s arise as the intersections of $\langle \Lambda_{red} \rangle$ with a Zariski-closed subgroup $N_{\Delta}$ of $G/Z_N$. Since $G/Z_N$ is reductive, we have three cases: (1) $G_{red}/N_{\Delta}$ is not semi-simple; (2) $G_{red}/N_{\Delta}$ is semi-simple but the projection of $\Lambda_{red}$ is not irreducible; (3) the projection of $\Lambda_{red}$ to $G_{red}/N_{\Delta}$ is irreducible, and is not contained in a discrete subgroup or $G_{red}/N_{\Delta}$ has $S$-rank at least $2$. Choose any $\Delta$, and define $\widehat{U}_{\Delta}:=Z_{G,1} \oplus \bigoplus_{\Delta' \neq \Delta} U^{\Delta'}$. By \cite[Prop. 4.17]{mac2023sadic}, $\Lambda^2 \cap \widehat{U}_{\Delta}$ is an approximate lattice in $\widehat{U}_{\Delta}$. So the projection of $\Lambda$ to $G/\widehat{U}_{\Delta}$ is a lattice. Since $U_{\Delta}$ is unipotent, one checks that $G/\widehat{U}_{\Delta}$ further projects to $G_{\Delta}:=\left(G_{red}/N_{\Delta}\right) \ltimes U_{\Delta}$.  By Proposition \ref{Proposition: Arithmeticity from action on subgroups} the projection $\Lambda_{\Delta}$ of $\Lambda$ to $G_{\Delta}$ is an approximate lattice. But, in cases (1) - (3), we find a reductive Levi subgroup $R_{\Delta} \subset G_{\Delta}$ such that $\Lambda_{\Delta}^2 \cap R_{\Delta}$ is an approximate lattice in $R_{\Delta}$ (and $\langle \Lambda_{\Delta} \rangle \cap  R_{\Delta}$ is dense when in case (3)). Indeed, in case (3) $G_{\Delta}$ is in the situation of \S \ref{Subsection: Representation-theoretic sum-product phenomenon}. But degree $1$ cohomology vanishes \cite[IV.5.8]{MR1090825} and quasi-cocycles are bounded (Theorem \ref{Theorem: Vanishing in higher-rank}) so we can apply Proposition \ref{Proposition: Boundedness implies splitting}; in case (2) apply Lemma \ref{Lemma: Commutativity implies boundedness of quasi-cocycles}; in case (1) first notice that the projection of $\Lambda_{red}$ intersects the centre of $G_{red}/N_{\Delta}$ in a lattice (Theorem \ref{Theorem: Radical is hereditary}), then apply Corollary \ref{Corollary: Boundedness quasi-cocycle case centre}. Then the inverse image $L'$ of $R_{\Delta}$ has intersection with $Z_N$ equal to $\hat{U}_{\Delta}$ and intersects $\Lambda^2$ in an approximate lattice.  Lemma \ref{Lemma: Reduction to abelian radical, higher-rank part} moreover implies that there is $\Lambda' \subset L'$ commensurable with $\Lambda^2 \cap L'$ such that $\langle \Lambda' \rangle$ projects densely to a simple factor of $G$ if and only if $\langle \Lambda\rangle$ does. We can thus repeat this process inductively until we find $L''$ such that $L'' \cap N = Z_{G,1}$. Note that, to ensure that the induction process goes smoothly, one can make sure to deal with all $\Delta$'s of type (3) first.

But we wanted $\Rad(G) \cap L'' =Z_{G,1}$.  By Lemma \ref{Lemma: Reduction to abelian radical} the projection of $\Lambda^2 \cap L''$ to $G/[G,G]Z_N$ is uniformly discrete.  So $L'' \cap [G,G]Z_N \cap \Lambda^2$ is an approximate lattice in $L :=L'' \cap [L,L]Z_N$. But $L$ contains a Levi subgroup of $G$ and $\Rad(L) = Z_{G,1}$.  Finally,  according to Lemma \ref{Lemma: Reduction to abelian radical, higher-rank part} we can find $\Lambda' \subset L$ an approximate lattice such that the projection of $\langle\Lambda'\rangle$ to some simple  factor of $G$ is not discrete if and only if the projection of $\langle \Lambda \rangle$ is not discrete. 
\end{proof}

\subsection{Proof of Theorems \ref{THEOREM: FIRST THEOREM} and \ref{THEOREM: MAIN THEOREM}}

\begin{proof}[Proof of Theorem \ref{THEOREM: MAIN THEOREM}.]
Let us start with building the relevant cut-and-project scheme.  By \cite[Prop. 4.17]{mac2023sadic} $Z_{G,1}=Z_G \times U$ and $\Lambda^2 \cap Z_G$, $\Lambda^2 \cap U$ are approximate lattices in $Z_G$ and $U$ respectively.. Furthermore, the action by conjugation on the unipotent subgroup $U$ has no non-trivial invariant vector. Define $\tau_z: Z_G \rightarrow V_z$ as a good model of an approximate subgroup commensurable with $\Lambda^2 \cap Z_G$ and whose target does not contain a non-trivial connected compact subgroup.  Define $W$ as the $\mathbb{Q}$-span of $\Lambda^2 \cap U$ and $\tau_u: W \rightarrow B \times \mathbb{A}^S \otimes W$ the diagonal homomorphism of the map $W \rightarrow B$ from Proposition \ref{Proposition: Meyer's theorem in S-adic vector spaces} and the natural map $W \rightarrow \mathbb{A}^S \otimes W$.  The group homomorphism $\tau_u$ is a good model of an approximate lattice commensurable with $\Lambda^2 \cap U$.  Let $\tau=\tau_z \times \tau_u$, $\Delta$ be the graph of $\tau$ and $V$ denote $V_z \times B \times \mathbb{A}^S \otimes W$.  Then $(Z_{G,1}, V, \Delta)$ is a cut-and-project scheme (Proposition \ref{Proposition: Equivalence good models and model sets}).

We can now build the map $\Phi$. Let $\alpha$ denote a class in $Z^1(\Gamma, Z_{G,1}) \times \widetilde{QC}(\Gamma, V) / Z^1(\Gamma, \Delta)$. Take a representative of $\alpha$ i.e. a pair made of a $1$-cocycle $b_{phys}: \Gamma \rightarrow Z_{G,1}$ and a quasi-cocycle $\bar{q}: \Gamma \rightarrow V$. Since $\tau(\Delta)$ is dense in $V$, upon taking another representative in the same class, we can assume that there is $q: \Gamma \rightarrow \Delta$ such that  $\tau \circ q = \bar{q}$. In particular, $q \in QC(\Gamma; Z_{G,1})$ where $Z_{G,1}$ is equipped with the bornology arising from $\Lambda^2 \cap Z_{G,1}$. Write $a:=b + q$. Define  $\Gamma_a:=\{(\gamma, a(\gamma)) \in G_s \ltimes \Rad(G): \Gamma \in \Gamma\}$ where we have chosen an arbitrary representation of $G$ as a semi-direct product using the Levi decomposition. Then $\Phi(\alpha)$ defined as the class of $\Lambda_a:= \Gamma_a\Lambda_{s,>1}\Lambda_r \cup (\Gamma_a\Lambda_{s,>1}\Lambda_r)^{-1}$ is as desired.  In particular, any other choice of representatives of $\alpha$ yields a commensurable approximate lattice. 

Let us show that $\Phi$ is a bijection.  We want to show surjectivity.  Take $\overline{\Lambda} \in Comm_G(\Lambda_s; \Lambda_r)$ and $\Lambda \in \overline{\Lambda}$.  Let $L$ be given by Lemma \ref{Lemma: End of proof} applied to $\Lambda$. Since the intersection of $L$ with $\Rad(G)$ is $Z_{G,1}$ and $L$ contains a semi-simple Levi subgroup $S$ of $G$, $Z_{G,1}$ is equal to $\Rad(L)$. In particular, $\Rad(L)$ is abelian. Therefore, there is $L \rightarrow S \ltimes Z_{G,1}$ surjective with finite kernel. Let $\tilde{\Lambda}$ be the image of $\Lambda$ in $S \ltimes Z_{G,1}$.  As in Section \ref{Section: Conjugation-multiplication phenomena} for every $\gamma \in \Gamma$, choose $a(\gamma) \in Z_{G,1}$ such that $(\gamma, a(\gamma)) \in \tilde{\Lambda}$.  By Proposition \ref{Proposition: quasi-cocycle takes values in Lambda}, there is a unique class $$\alpha_{L,S,\Lambda} \in H^1(\langle \Lambda_s \rangle,Z_{G,1}) \times \widetilde{QC}(\Lambda_s;V)/H^1(\langle \Lambda_s \rangle,\Delta)$$ associated with $a$ (see also Remark \ref{Remark: Decomp quasi-cocycles}).  But $H^1(\langle \Lambda_s \rangle,Z_{G,1}) = H^1(\Gamma_1,Z_{G,1})$ \cite[IV.5.8]{MR1090825} and $\widetilde{QC}(\Lambda_s;V)=\widetilde{QC}(\Gamma_1;V)$ by our vanishing result (Theorem \ref{Theorem: Vanishing in higher-rank}). So $H^1(\langle \Lambda_s \rangle,Z_{G,1}) \times \widetilde{QC}(\Lambda_s;V)/H^1(\langle \Lambda_s \rangle,\Delta) \simeq H^1(\Gamma, Z_{G,1}) \times \widetilde{QC}(\Gamma, V) / H^1(\Gamma, \Delta)$. Then $\alpha_{L,S,\Lambda}$ is the class we are looking for. 

One can make the crucial observation that $\alpha_{L,S,\Lambda}$ does not depend on either the choice of $L$, $S$ or $\Lambda \in \overline{\Lambda}$.  Indeed, let $a'$ be obtained via another choice $S'$ of $S$ in $L$. There is $z \in Z_{G,1}$ such that $S'^z=S$. We find therefore that  for all $\gamma \in \Gamma$ we have $(\gamma,e)^z(e,a'(\gamma)) \in \tilde{\Lambda}$. So $(\gamma, (\id - \gamma) \cdot z + a'(\gamma)) \in \Lambda$. But $(\gamma, a(\gamma)) \in \Lambda$. Thus, $a' - a + (\id - \gamma) \cdot z$ takes values in $\Lambda^2$.  Suppose now that $a'$ is obtained via another choice $L'$ of $L$ and $S' \subset L'$. Then by \cite[Prop. 4.18]{mac2023sadic} we know first that there is $u \in \Comm_G(\Gamma)$ in the unipotent radical such that $L'^u=L$ and $z \in Z_{G,1}$ such that $S'^{zu}=S$. Since $\Lambda^u$ is commensurable with $\Lambda$ and $u$ is in the unipotent radical we find again that  $$(\gamma, (\id - \gamma) \cdot z + a'(\gamma)) \in F + \Lambda^2 \cap Z_{G,1}$$ for some finite subset $F \subset Z_{G,1}$.  As a consequence of a similar line of reasoning,  the reverse construction also does not depend on the choice of $\Lambda$ in the class considered. 

Since these two constructions are inverse to each other, $\Phi$ is indeed a bijection.  We have proved (1).  It remains to prove (2) and (3).  So choose $\alpha, a, S$ and $\Lambda:=\Lambda_a$ as in the second paragraph.  Write $L=SZ_{G,1}$.

Part (3) is proved along similar lines: suppose that there is a $1$-co-boundary $b_{phy}$ of $\Gamma$ with coefficients in $Z_{G,1}$ such that 
$$(a-b_{phys})(\Lambda) \subset F + \Lambda^2 \cap Z_{G,1}$$
where $a$ is chosen as in the second paragraph.
Then $b_{phys}(\gamma) = (id - \gamma)\cdot z $ for some $z \in Z_{G,1}$. Then Proposition \ref{Proposition: Boundedness implies splitting} implies that $z^{-1}Sz \cap \Lambda^2$ is an approximate lattice in $z^{-1}Sz$. Conversely, suppose that there is a Levi subgroup $S'$ of $G$ such that $S' \cap \Lambda^2$ is an approximate lattice in $S'$.  Then $\alpha_{L, S',\Lambda}$ clearly vanishes where $L=S'Z_{G,1}$.

Let us now prove (2). Assume first that $\Lambda$ is laminar. Then $\Lambda^2 \cap L$ is laminar as well \cite[Lem. 3.3]{machado2019goodmodels}. So $a$ must be of the desired form by Proposition \ref{Proposition: Equivalence laminarity/non-trivial cohomology class}. Conversely, Proposition \ref{Proposition: Equivalence laminarity/non-trivial cohomology class} implies that if $a$ is of the desired form, $\Lambda^2 \cap L$ is laminar. But $\Lambda$ is commensurable with $\left(\Lambda^2 \cap L\right)\left( \Lambda^2 \cap \Rad(G) \right)$. So $\Lambda$ is laminar according to Lemma \ref{Lemma: Laminarity from laminarity in decompositions}.

\end{proof}

\section{Remarks and consequences}\label{Subsection: Remarks and consequences}

\emph{Generalisations to other groups.} The method described in this paper extends to other frameworks.  In Lie groups,  we find the same examples as described above, along with one more family related to central extensions of simple Lie groups. We find in particular non-laminar approximate lattices in central extensions of any simple Lie group with Hermitian symmetric space (some of which have higher-rank). This shows that Hrushovski's guess was right for Lie groups upon replacing $H_b^2(\Gamma_s;V)$ by $H_b^2(\Lambda_s;V)$. 
 
Over fields of positive characteristic,  most approximate lattices are commensurable with lattices, extending \cite{hrushovski2020beyond} from semi-simple groups.

Using ideas from \cite{MR1840942},  one may possibly extend the vanishing result (Theorem \ref{Theorem: Vanishing in higher-rank}) to all uniform approximate lattices in any locally compact groups.  Moreover,  approximate lattices in locally compact groups admit a weak form of Theorem \ref{Theorem: Radical is hereditary}.  Combining these two ideas, it seems likely that one can understand uniform approximate lattices of non-linear locally compact groups in terms of the bounded cohomology with trivial coefficients of their ambient group.  Such a result would provide the non-linear counterpart to our work and complete the classification of approximate lattices to all locally compact groups.

\emph{Arithmeticity.}Combined with a simple Krull--Schmidt argument, Theorem \ref{THEOREM: MAIN THEOREM} enables us to generalise the arithmeticity theorem for approximate lattices \cite[Thm 7.4]{hrushovski2020beyond} to perfect groups:

\begin{theorem}[Arithmeticity in perfect groups]\label{Theorem: Arithmeticity for perfect groups}
Let $\Lambda$ be an irreducible approximate lattice in a perfect $S$-adic linear group $G$ without compact factor.  Suppose that there is no rank one factor $p:G \rightarrow S$ of $G$ such that $p(\Lambda)$ is contained in a lattice.There is then a number field $K$, a finite set of places $S$ of $K$ and a linear $K$ group $\mathbf{G}$ such that we can find an isogeny $\phi:\mathbf{G}(\mathbb{A}_{S,K}) \rightarrow G$ such that $\phi(\mathbf{G}(\mathcal{O}_{K,S}))$ is commensurable with $\Lambda$. 
\end{theorem}

It is worthwhile to compare the above results with Mostow's arithmeticity results \cite{MR289713} for lattices in Lie groups with radical.  While under a higher-rank condition approximate lattices are indeed arithmetic,  Theorem \ref{THEOREM: MAIN THEOREM} provides approximate lattices in rank one that cannot be deformed (in the sense of \cite{MR289713}) into arithmetic ones. 

\emph{Around cohomology}.  Our work shows how bounded cohomology for approximate subgroups captures many components of the structure of approximate subgroups - and even more so for approximate lattices.  This leads to a number of natural questions: what properties do approximate lattices associated with quasi-cocycles of more specific kind (say $\Aut$-invariant) have? What is the meaning of higher degree cohomology groups of approximate subgroups (and their vanishing)? How is the vanishing of these groups related to properties such as amenability? Suppose that all bounded cohomology groups associated with some $\Lambda$ vanish, is there a universal (independent of $\Lambda$) $m > 0$ such that $\Lambda^m$ has a good model? This questions are in particular motivated by the exotic examples found in \cite{mac2023definitions}.

\section{Appendix}
\begin{proof}
We will follow closely the proof of \cite[Lemma 4.1]{WitteMorris}, the notation we use here are the ones from \cite{WitteMorris} rather than the ones used up to now in this work. Since this result is used crucially in the proof of the main theorem of this paper, we nevertheless present a complete proof. We use the notation of \cite[Notation 3.5, p6]{WitteMorris}. In section \cite[Notation 3.5]{WitteMorris} simply replace $\mathbb{Q}$ by $K$ and $\mathbb{Z}$ by $\mathcal{O}_{K,v}$. For each simple root $\alpha$ of $\Phi^+$ define $\omega_{\alpha}$ as in \cite[Definition 3.8]{WitteMorris} - replacing every occurrence of $\mathbb{Q}$ and $\mathbb{Z}$ by $K$ and $\mathcal{O}_{K,v}$ respectively. In particular, we may choose $\omega_{\alpha}$ such that $\omega_{\alpha}(\Lambda) \subset \mathcal{O}_{K,v}$. 

Recall that $\mathbf{U}^+$ denotes the unipotent group spanned by the roots of $\Phi^+$ and $\mathbf{P}^-$ is the minimal parabolic $K$-group that is opposite to $\mathbf{U}^+$. We have $(\alpha_i)_{i \in \{1, \ldots, k\}}$ an enumeration of $\Phi^+$ (in the order chosen in \cite{WitteMorris}) and $\mathbf{U}_i^+$ the root subgroup associated with $\alpha_i$. We set $\overleftarrow{\mathbf{U}_i^+}:=\mathbf{U}_1^+\cdots \mathbf{U}_{i-1}^+$ and $\overrightarrow{\mathbf{U}_i^+}:=\mathbf{U}_{i+1}^+\cdots \mathbf{U}_k^+$ for all $i \in \{1, \ldots, k\}$. The choice of the enumeration makes the multiplication map $\mathbf{U}_1^+ \times \ldots \times \mathbf{U}_k^+ \rightarrow \mathbf{U}^+$ a $K$-isomorphism.  The multiplication map $\mathbf{U}^+ \times \mathbf{P}^- \rightarrow \mathbf{G}$ induces a $K$-regular isomorphism towards a $K$-open subset $\Omega$ of $\mathbf{G}$. Denote $u^+: \Omega \rightarrow \mathbf{U}^+$, $p^{-}: \Omega \rightarrow \mathbf{P}^-$ the regular maps arising from the inverse of the multiplication map of the previous sentence. We also denote $u_i^+: \Omega \rightarrow \mathbf{U}_i^+$, $\overleftarrow{u_i}^+: \Omega \rightarrow \overleftarrow{\mathbf{U}_i}^+$ and $\overrightarrow{u_i}^+: \Omega \rightarrow \overrightarrow{\mathbf{U}_i}^+$ the obvious obtained from $u^+$.

\begin{lemma}[Lemma 4.11,\cite{LubotzkyMozesRaghunathan}]
We have: 
\begin{enumerate}
\item $K[\Omega]=K[\mathbf{G}][\frac{1}{\prod_{\alpha \in \Phi} \omega_{\alpha}}]$; 
\item Fix an integer $m >0$, for all $\gamma \in \Lambda^m \cap \Omega$ and all $i$, we have $||p^-(\gamma)||_v \preceq ||\gamma||_v$ and $||u_i^+(\gamma)|| \preceq ||\gamma||$. 
\end{enumerate}
\end{lemma}

While $(1)$ is explicitly contained in \cite{LubotzkyMozesRaghunathan}, $(2)$ requires an explanation. So, according to $(1)$, the entries of $p^-(\gamma)$ are rational functions in the entries of $\gamma$ whose denominators are polynomials in the $\omega_{\alpha}$. Since $\omega_{\alpha}(\Lambda) \subset \mathcal{O}_{K,v}$ for all $\alpha$, we have $\omega_{\alpha}(\Lambda^m)$ covered by finitely many translates of $\mathcal{O}_{K,v}$ (Lemma \ref{Lemma: Pisot approximate rings and polynomials}). The $|\cdot|_v$-absolute value of the denominators are thus bounded below by some $c>0$ independent of $\gamma$ (Lemma \ref{Lemma: Pisot approximate rings and polynomials} again). Thus, $||p^-(\gamma)||_v \preceq ||\gamma ||_v$. The same argument applies to $u_i^+(\gamma)$. 

We are now in a position to show the crux of the proof:

\begin{lemma}[Lemma 4.1, \cite{WitteMorris}]\label{Lemma: Generalisation 4.1}
Let $\Lambda'$ be an approximate lattice commensurable with $\Lambda$. Let $i \in \{1,\ldots, k\}$ and let $F_i$ be a finite subset $F_{i+1}$ of $\mathbf{U}^+(K)$. There is a finite subset of $\mathbf{U}^+(K)$ such that if $\lambda \in \Lambda' \cap \Omega(K)$ and $\overleftarrow{u_i^+}(\lambda) \in F_i$, then there exists $x_i \in ^{\overleftarrow{u_i^+}(\lambda)}\mathbf{G}_i(K_v) \cap \Lambda'^{12}$ with 
\begin{enumerate}
\item $x_i\lambda \in \Omega(K)$ and $\overleftarrow{u_{i+1}^+}(x_i\lambda) \in F_{i+1}$;
\item $||x_i||_v \preceq ||\lambda||_v$. 
\end{enumerate}
\end{lemma}

Recall that for an element $g \in G$ and a subset $X \subset G$, we denote by $^{g}X$ the subset $gXg^{-1}$. Following the proof of \cite[Lemma 4.1]{WitteMorris} we write
$$ u=u^+(\lambda), \overleftarrow{u}=\overleftarrow{u_i^+}(\lambda), \hat{u}=u_i^+(\lambda), \overrightarrow{u}=\overrightarrow{u_i^+}(\lambda), p=p^-(\lambda),$$
see \cite[Notation 3.5, p6]{WitteMorris}.

Let $\Lambda_i$ denote $\Lambda'^2 \cap ^{\overleftarrow{u}^{-1}}\Lambda'^2 \cap \mathbf{G}_i(K_v)$. Then $\Lambda_i \subset \mathbf{G}_i(K)$ is commensurable with $\mathbf{G}_i(\mathcal{O}_{K,v})$ and $^{\overleftarrow{u}}\Lambda_i \subset \Lambda'^2$. 

According to Proposition \ref{Proposition: Generalisation of Borel's result} there is a finite subset $F^0$ of $\mathbf{G}_i(K)$ such that $\mathbf{G}_i(K)= \Lambda_i F^0 \mathbf{P}_i^-(K)$. We will show that upon considering $\Lambda_i^2$ instead of $\Lambda_i$, we can take $F^0 \subset \mathbf{U}_i^+(K)$. We have that $\mathbf{U}_i^+\mathbf{P}_i^-$ is Zariski-open in $\mathbf{G}_i$. Therefore, by the Borel density theorem for approximate lattices \cite[A.11]{hrushovski2020beyond}, for any $f \in F^0$, $\Lambda_if \cap  \mathbf{U}_i^+(K)\mathbf{P}_i^-(K) \neq \emptyset$. Thus we may multiply $f$ on the left by an element of $\Lambda_i$ and on the right by an element of $\mathbf{P}_i^-(K)$ to obtain $f' \in \mathbf{U}_i^+(K)$. This proves our claim.

Hence, we have $\hat{u}=xfq$ with $x \in \Lambda_i^2, f\in F^0 \subset \mathbf{U}_i^+(K)$ and $q \in \mathbf{P}_i^-(K)$. Following the computations in \cite[pp.10-11]{WitteMorris}, if we set $x_i=^{\overleftarrow{u}}x^{-1}$ and $F_{i+1}=F_iF^0$, we have 
$$x_i\lambda \in \Omega(K) \text{ and } \overleftarrow{u_{i+1}^+}(x_i\lambda) = \overleftarrow{u}f \in F_iF^0=F_{i+1}.$$ So $(1)$ is proved.

We will now prove that we can ensure (2) as well. Write $\widehat{U}_{\mathcal{O}_{K,v}}:= \Lambda_i^2\cap ^{\overleftarrow{u}}\mathbf{U}_i^+(K_v)$. By \S \ref{Subsubsection: Pisot--Vijayaraghavan--Salem numbers of a number field}, $\widehat{U}_{\mathcal{O}_{K,v}}$ is a uniform approximate lattice in $^{\overleftarrow{u}}\mathbf{U}_i^+(K_v)$. Therefore, upon multiplying $\lambda$ by an element $u'$ of $\hat{U}_{\mathcal{O}_{K,v}}$ on the left of size $||u'||_v \preceq ||u_i^+(\lambda)||_v \preceq ||\lambda||_v$ we may assume $||u_i^+(\lambda)||_v = O(1)$. Moreover, $u'x \in \Lambda_i^4$ so we may assume from the start that $||\hat{u}||_v=O(1)$ and $x \in \Lambda_i^4$. 

Write the Langlands decomposition $\mathbf{P}_i^-=\mathbf{M}_i\mathbf{A}_i\mathbf{N}_i$ where $\mathbf{A}_i$ is the connected component of the intersection of the fixed $K$-split torus $\mathbf{T}$ with $\mathbf{G}_i$. Write $ q=ma$ with $m \in (\mathbf{M}_i\mathbf{N}_i)(K)$ and $a \in \mathbf{A}_i(K)$. By \S \ref{Subsubsection: Pisot--Vijayaraghavan--Salem numbers of a number field}, since $\mathbf{M}_i\mathbf{N}_i$ is $K$-anisotropic,  $^{f^{-1}}\Lambda_i^2 \cap (\mathbf{M}_i\mathbf{N}_i)(K_v)$ is a co-compact approximate lattice in $(\mathbf{M}_i\mathbf{N}_i)(K_v)$. Therefore, there is $m' \in \Lambda_i^2 \cap ^{f}(\mathbf{M}_i\mathbf{N}_i)(K_v)$ such that $||^{f^{-1}}m'm||_v =O(1)$. But $xm'^{-1} \in \Lambda_i^6$ and 
$$xfq=xfma = xm'^{-1}f(^{f^{-1}}m'm)a.$$
So upon modifying $x \in \Lambda_i^6$ we may assume that $||m||_v = O(1)$ to begin with. Choose now a simple $K$-root $\alpha$ that is not orthogonal to $\Phi_i$ (where $\Phi_i$ is the set of roots that are scalar multiples of $\alpha_i$). Then, following computations in \cite[p.11]{WitteMorris}, we find 
$$||x||_v \preceq \max(|\omega_{\alpha}(a)|_v,|\omega_{\alpha}(a^{-1})|_v)$$ where we have used that $\mathbf{A}_i$ is one-dimensional and $\alpha$ not trivial on $\mathbf{A}_i$. 
In addition, 
$$|\omega_{\alpha}(a^{-1})|_v = |\omega_{\alpha}(xf)|_v.$$
But $\phi: x \mapsto \omega_{\alpha}(xf)$ is a $K$-regular map. So $\phi(\Lambda_i^4)$ is discrete in $K_v$ by Lemma \ref{Lemma: Pisot approximate rings and polynomials}. Since moreover $xf=\hat{u}q^{-1} \in (\mathbf{U}^+\mathbf{P}^-)(K)=\Omega(K)$, we have 
$$|\omega_{\alpha}(a^{-1})|_v > \frac{1}{O(1)}.$$

Now, $|\omega_{\alpha}(a^{-1})|_v = \frac{1}{|\omega_{\alpha}(a)|_v}$. So 
$$ \max(|\omega_{\alpha}(a)|_v,|\omega_{\alpha}(a^{-1})|_v) = \max(O(1), |\omega_{\alpha}(a^{-1})|_v).$$

Following \cite[p.11]{WitteMorris} again we have, 
$$|\omega_{\alpha}(^{\overleftarrow{u}}x^{-1}\lambda)|_v= |\omega_{\alpha}(a)|_v\cdot |\omega_{\alpha}(p)|_v. $$

As above, we find $|\omega_{\alpha}(a)|_v\cdot |\omega_{\alpha}(p)|_v \geq \frac{1}{O(1)}$.

So $$|\omega_{\alpha}(a^{-1})|_v = \frac{1}{|\omega_{\alpha}(a)|_v} \preceq |\omega_{\alpha}(p)|_v \preceq ||p||_v \preceq ||\lambda||_v.$$
Thus, indeed,
$$||x||_v \preceq ||\lambda||_v.$$ This concludes the proof of the lemma.

It remains only to piece all the ingredients together. Suppose first that $\lambda \in \Lambda^2 \cap \Omega(K_v)$. Set $F_1 =\{e\}$. Then successive applications of Lemma \ref{Lemma: Generalisation 4.1} imply the existence of finite subsets $F_2 , \ldots, F_k \subset \mathbf{U}^+(K)$ and  $x_1 \in \Lambda^{13} \cap \mathbf{G}_1(K_v), \ldots, x_k \in \Lambda^{13^k} \cap \mathbf{G}_k(K_v)$ such that $||x_i||_v \preceq ||\lambda||_v$ for all $i \in \{1, \ldots, k\}$ and $u^+(x_k \cdots x_1\lambda) \in F_{k+1}$. Since $F_{k+1}$ is finite, there exists a further finite subset $F_{k+2} \subset \Lambda^{13^k}$ such that there is $x_{k+1} \in F_{k+2}$ with $u^+(x_{k+1}x_k \cdots x_1\lambda)=\{e\}$ i.e. $x_{k+1}x_k \cdots x_1\lambda \in \mathbf{P}^-(K) \cap \Lambda^m$ for some integer $m > 0$. 

Notice now that $\Omega$ is Zariski-open in $\mathbf{G}$ and $\Lambda$ is Zariski-dense (\cite[A.11]{hrushovski2020beyond}). So by the descending chain condition on Zariski-closed subsets, there is $F_0 \subset \Lambda$ such that $\Lambda \subset F_0\Omega(K)$. So $\Lambda \subset  F_0\left(\Lambda^2 \cap \Omega(K)\right)$. Which concludes the proof. 
\end{proof}

\end{document}